\documentclass[11pt,leqno]{amsart}
\usepackage[foot]{amsaddr}

\usepackage{amsthm,amsfonts,amssymb,amsmath,amscd,latexsym}
\usepackage{graphicx,exscale,cite,epsfig}

\usepackage{hyperref,xcolor}
\usepackage{fullpage}

\renewcommand{\geq}{\geqslant}
\renewcommand{\leq}{\leqslant}

\def\eps{\varepsilon}
\newcommand{\Id}{{\rm Id }}

\newcommand\ba{\begin{equation}\begin{aligned}}
\newcommand\ea{\end{aligned}\end{equation}}
\newcommand{\be}{\begin{equation}}
\newcommand{\ee}{\end{equation}}

\newtheorem{theorem}{Theorem}[section]

\newtheorem{corollary}[theorem]{Corollary}
\newtheorem{conjecture}[theorem]{Conjecture}
\newtheorem{lemma}[theorem]{Lemma}
\newtheorem{definition}[theorem]{Definition}
\newtheorem*{definition*}{Definition}

\theoremstyle{remark}

\newtheorem*{remark*}{Remark}

\numberwithin{equation}{section}

\begin{document}

\title[Unstable Stokes waves]{Unstable Stokes waves}

\author{Vera~Mikyoung~Hur}
\address{Department of Mathematics, University of Illinois Urbana-Champaign, Urbana, IL 61801 USA}
\email{verahur@math.uiuc.edu}
\thanks{VMH is supported by the NSF award DMS-2009981.}

\author{Zhao Yang}
\address{Academy of Mathematics and Systems Science, Chinese Academy of Sciences, Beijing 100190 China}
\email{yangzhao@amss.ac.cn}



\begin{abstract}
We investigate the spectral instability of a $2\pi/\kappa$ periodic Stokes wave of sufficiently small amplitude, traveling in water of unit depth, under gravity. Numerical evidence suggests instability whenever the unperturbed wave is resonant with its infinitesimal perturbations. This has not been analytically studied except for the Benjamin--Feir instability in the vicinity of the origin of the complex plane. Here we develop a periodic Evans function approach to give an alternative proof of the Benjamin--Feir instability and, also, a first proof of spectral instability away from the origin. Specifically we prove instability near the origin for $\kappa>\kappa_1:=1.3627827\dots$ and instability due to resonance of order two so long as an index function is positive. Validated numerics establishes that the index function is indeed positive for some $\kappa<\kappa_1$, whereby there exists a Stokes wave that is spectrally unstable even though it is insusceptible to the Benjamin--Feir instability. The proofs involve center manifold reduction, Floquet theory, and methods of ordinary and partial differential equations. Numerical evaluation reveals that the index function remains positive unless $\kappa=1.8494040\dots$. Therefore we conjecture that all Stokes waves of sufficiently small amplitude are spectrally unstable. For the proof of the conjecture, one has to verify that the index function is positive for $\kappa$ sufficiently small. 
\end{abstract} 

\maketitle 

\tableofcontents


\section{Introduction}\label{sec:intro}

Stokes in his 1847 paper \cite{Stokes1847} (see also \cite{Stokes1880}) made significant contributions to periodic waves at the free surface of an incompressible inviscid fluid in two dimensions, under the influence of gravity, traveling in permanent form with a constant velocity. For instance, he successfully approximated the solution when the amplitude is sufficiently small. The existence of Stokes waves was proved in the 1920s for small amplitude \cite{Nekrasov1921, Levi-Civita;stokes, Struik;stokes} and in the early 1960s for large amplitude \cite{Krasovskii1960, Krasovskii1961}. Therefore it came as a surprise in the mid 1960s when Benjamin \cite{Benjamin;BF} and Whitham \cite{Whitham;BF} (see also references cited in \cite{ZO;review} for others) discovered that a Stokes wave in sufficiently deep water, so that $\kappa h>1.36278\dots$, is unstable to long wavelength perturbations---namely, the Benjamin--Feir or modulational instability. Here $\kappa$ denotes the wave number of the unperturbed wave, and $h$ the fluid depth. 
In the 1990s, Bridges and Mielke \cite{BM;BF} analytically studied spectral instability of a Stokes wave of sufficiently small amplitude, rigorously justifying the formal arguments of \cite{Benjamin;BF, Whitham;BF} and others. See also \cite{NS,BMV,Yang;infinite} for infinite depth. Some fundamental issues remain open, however.  

In the 1980s, McLean \cite{McLean;finite-depth} (see also \cite{McLean;infinite1, McLean;infinite2} for infinite depth) numerically found spectral instability away from the origin of the complex plane. The Benjamin--Feir instability, by contrast, refers to the spectrum in the vicinity of the origin. Further numerical findings (see, for instance, \cite{FK}) suggested instability whenever the unperturbed wave is ``resonant'' with its infinitesimal perturbations (see \eqref{def:resonance}). Recently, Deconinck and Oliveras \cite{DO} focused the attention to resonance and numerically discovered ``bubbles'' of instability. Therefore there seems to exist a Stokes wave that is spectrally unstable even though it is insusceptible to the Benjamin--Feir instability. To the best of the authors' knowledge, this has not been analytically studied. We remark, however, more recently, Creedon, Deconinck, and Trichtchenko \cite{creedon2021highfrequency} gave analytical and numerical evidence of instability for almost all values of $\kappa h$. The present purpose is to make rigorous spectral analysis, to elucidate the numerical findings of \cite{McLean;finite-depth, FK, DO} and others and, also, justify the formal argument of \cite{creedon2021highfrequency} and, in the process, to give a first proof of spectral instability of a Stokes wave of sufficiently small amplitude away from the origin of the complex plane. 

Gardner's periodic Evans function \cite{Gardner;evans1, Gardner;evans2} is a powerful tool for locating and tracking essential spectrum for a wide variety of PDEs in one dimension, from viscous conservation laws \cite{Oh2003, Serre;evans} to the generalized Kruamoto--Sivashinsky equations \cite{KS;evans}, to generalized Korteweg--de Vries equations \cite{BJ;gKdV}, and to the nonlinear Klein--Gordon equations \cite{Klein-Gordon;evans}. Recently, the second author and his collaborators \cite{Johnson2019} devised a periodic Evans function methodology for discontinuous roll waves of the inviscid Saint--Venant equations.   
Despite the success in one dimension, however, periodic Evans function techniques have rarely been implemented in higher dimensions. Deng and Nii \cite{DN;evans} introduced an infinite-dimensional Evans function approach for elliptic eigenvalue problems in cylindrical domains, but for stability and instability of solitary waves. Also it is highly nontrivial to construct an unstable bundle, whose first Chern number equals to the number of eigenvalues, and evaluate such a topological quantity. Oh and Sandstede \cite{OS;evans} defined an approximate Evans function for periodic traveling waves in cylindrical domains. Unfortunately, it is incapable of exactly locating the spectrum. Here we develop a periodic Evans function approach for Stokes waves of sufficiently small amplitude which discloses spectral information through explicit calculations. 
We pause to remark that Haragus and Scheel \cite{HS;cg-solitary} defined an Evans function for capillary-gravity solitary waves and proved stability to finite wavelength perturbations for sufficiently small amplitude. To the best of the authors' knowledge, however, a periodic Evans function has not been proposed for the water wave problem. 

We begin in Section~\ref{sec:formulation} by ``flattening'' the free boundary and reformulating the water wave problem as first order PDEs with respect to the $x$ variable in $\mathbb{R}\times (0,1)$ in the $(x,y)$ plane (see \eqref{eqn:ww}), where $x$ is in the direction of wave propagation, and $y$ opposite to gravitational acceleration. There are other ways to fix the free boundary, for instance, reformulating the problem in terms of quantities at the fluid surface alone. The resulting equations become nonlocal (see, for instance, \cite{BMV, NS}), however, and periodic Evans function techniques are inapplicable. Recently, the first author and her collaborators (see, for instance, \cite{BH;fKdV,HJ;Whitham,HP;BW}) performed spectral perturbation analysis in the vicinity of the origin of the complex plane for small values of the Floquet exponent, determining modulational stability and instability for a class of nonlinear dispersive equations, permitting nonlocal operators. It is highly nontrivial, however, to extend the argument to the spectrum away from the origin. 

After working out in Section~\ref{sec:Stokes} the small amplitude asymptotics for periodic traveling waves (see \eqref{eqn:stokes exp}), in Section~\ref{sec:L}, we formulate the spectral stability problem for a Stokes wave of sufficiently small amplitude as first order ODEs with respect to the $x$ variable in an infinite-dimensional function space (see \eqref{eqn:LB}). In Section~\ref{sec:eps=0}, we focus the attention to zero amplitude and verify that the spectrum is the imaginary axis and parametrized by the Floquet exponent through the dispersion relation (see \eqref{eqn:sigma}). Also, at each point of the imaginary axis, we define a finite dimensional eigenspace (see Definition~\ref{def:proj}). In Section~\ref{sec:reduction}, we turn the attention to nonzero amplitude. We take a center manifold reduction approach (see, for instance, \cite{Mielke;reduction}) to reduce the spectral stability problem in an infinite-dimensional function space, to the finite dimensional eigenspace (see \eqref{eqn:LB;u1} and \eqref{eqn:A}), whereby we define Gardner's periodic Evans function (see \eqref{def:Delta}). In Section~\ref{sec:proj}, we work out explicit formulae of the projection operator onto the eigenspace. This is among the most technical parts of the proof. 

We remark that for zero Floquet exponent, the linearized operator of the water wave problem about a Stokes wave of sufficiently small amplitude, not necessarily zero amplitude, has four eigenvalues at the origin of the complex plane. Bridges and Mielke \cite{BM;BF} examined whether the eigenvalues enter the right half plane as the Floquet exponent increases, to determine modulational stability and instability. (See also \cite{BMV, NS} for the infinite depth.) Unfortunately, one does not expect to be able to exactly locate the spectrum away from the origin for nonzero amplitude. Our approach, instead, takes advantage of that for zero amplitude, the spectrum is explicitly characterized through the dispersion relation, and we examine how the spectrum varies as the amplitude increases.

In Section~\ref{sec:BF}, we determine the power series expansion of the periodic Evans function for the spectrum in the vicinity of the origin of the complex plane, to give an alternative proof of the Benjamin--Feir instability for $\kappa h>1.3627827\dots$. See Theorem~\ref{thm:BF} and Corollary~\ref{cor:BF}. Following along the same line of argument, one can show instability due to resonance of order two (see \eqref{def:resonance}), provided that $0.86430\dots<\kappa h<1.00804\dots$. In Section~\ref{sec:high-freq}, we take matters further and relate the zeros of the periodic Evans function to those of a quadratic Weierstrass polynomial, to prove spectral instability due to resonance of order two, provided that $\operatorname{ind}_2(\kappa h)>0$ (see \eqref{def:ind2}). See Theorem~\ref{thm:unstable2}. Calculating the monodromy matrix of the periodic Evans function involves extremely long and tedious algebraic manipulations, for which we use Symbolic Math Toolbox in MATLAB. The MATLAB scripts generated during the course of the project can be made available upon request.

Numerical evaluation reveals that $\operatorname{ind}_2(\kappa h)>0$ (see \eqref{def:ind2}) unless $\kappa h=1.84940\dots$. Our result agrees with that from a formal perturbation method (see, for instance, \cite{creedon2021highfrequency}) and explains numerical findings (see \cite{McLean;finite-depth, FK, DO} among others). Also our result analytically confirms ``bubbles" of unstable spectrum (see \cite{DO, creedon2021highfrequency} among others). Importantly our result analytically confirms Stokes waves of sufficiently small amplitude that are spectrally unstable even though they are insusceptible to the Benjamin--Feir instability. See Corollary~\ref{cor:high-freq}. Furthermore our result leads to the conjecture that {\em all} Stokes waves of sufficiently small amplitude are spectrally unstable. For the proof of the conjecture, it remains to verify that $\operatorname{ind}_2(\kappa h)>0$ for $\kappa h$ sufficiently small.

Our approach is robust and can accommodate the effects of infinite depth \cite{Yang;infinite} surface tension \cite{capillary-gravity}, constant vorticity \cite{constant-vorticity}, transversal perturbations, among many others. This is a subject of future investigation. Also it can be useful for other PDEs in higher dimensions, for instance, the equations in \cite{DN;evans, OS;evans}. 

\noindent {\bf Acknowledgement.}~The authors are grateful to Massimiliano Berti and Benard Deconinck for valuable discussions, and the anonymous referee for helpful suggestions. 

\section{The water wave problem}\label{sec:formulation}

The water wave problem, in the simplest form, concerns the wave motion at the free surface of an incompressible inviscid fluid in two dimensions, lying below a body of air, acted on by gravity, when the effects of surface tension are negligible. Although an incompressible fluid can have variable density, we assume for simplicity that the density $=1$. Suppose for definiteness that in Cartesian coordinates, the wave propagation is in the $x$ direction, and the gravitational acceleration in the negative $y$ direction. Suppose that the fluid at rest occupies the region $\{(x,y)\in\mathbb{R}^2: 0<y<h\}$, where $h>0$ is the fluid depth. Let 
\[
y=h+\eta(x,t),\quad x\in\mathbb{R},
\] 
denote the fluid surface at time $t$, and $y=0$ the rigid bed. Physically realistic is that $h+\eta(x,t)>0$ for all $x\in\mathbb{R}$. Throughout we assume an irrotational flow, whereby a velocity potential $\phi(x,y,t)$ satisfies
\begin{subequations}\label{eqn:ww;h}
\begin{align}
&\phi_{xx}+\phi_{yy}=0&& \text{for $0<y<h+\eta(x,t)$,}
\intertext{subject to the boundary condition}
&\phi_y=0&&\text{at $y=0$.}
\intertext{The kinematic and dynamic boundary conditions at the fluid surface,}
&\left.\begin{aligned}
&\eta_t-c\eta_x+\eta_x\phi_x=\phi_y\\
&\phi_t-c\phi_x+\frac12(\phi_x^2+\phi_y^2)+g\eta=0\\
\end{aligned}\,\right\}&&\text{at $y=h+\eta(x,t)$}, 
\end{align}
\end{subequations}
state that each fluid particle at the surface remains so for all time and that the pressure is constant at the fluid surface, where $c\neq0,\in\mathbb{R}$ is the velocity of the wave, and $g>0$ the constant of gravitational acceleration. We assume that there is no motion in the air. 

We begin by recasting \eqref{eqn:ww;h} in dimensionless variables. Rather than introducing new notation for all the variables, we choose, wherever convenient, to write, for instance, $x\mapsto x/h$. This is to be read ``$x$ is replaced by $x/h$'', so that hereafter the symbol $x$ will mean a dimensionless variable. With the understanding, let 
\begin{equation}\label{def:dimensionless}
\begin{gathered}
x\mapsto x/h,\qquad y\mapsto y/h,\qquad t\mapsto ct/h, \\
\text{and}\qquad\eta\mapsto \eta/h,\qquad \phi\mapsto \phi/(ch).
\end{gathered}
\end{equation}
Correspondingly let
\begin{equation}\label{def:mu}
\mu=gh/c^2
\end{equation}
denote the (dimensionless) inverse square of the Froude number. 
Inserting \eqref{def:dimensionless} and \eqref{def:mu} into \eqref{eqn:ww;h} we arrive at
\ba\label{eqn:ww;1}
&\phi_{xx}+\phi_{yy}=0&& \text{for $0<y<1+\eta(x,t)$,}\\
&\phi_y=0&&\text{at $y=0$,}\\
&\eta_t-\eta_x+\eta_x\phi_x=\phi_y&&\text{at $y=1+\eta(x,t)$,} \\ 
&\phi_t-\phi_x+\frac12(\phi_x^2+\phi_y^2)+\mu\eta=0\quad&&\text{at $y=1+\eta(x,t)$.}  
\ea 

It is advantageous to formulate the spectral stability problem for \eqref{eqn:ww;1} as first order ODEs with respect to the $x$ variable (see \eqref{eqn:spec}). We introduce
\be\label{def:u}
u=\phi_x.
\ee 
Notice that \eqref{eqn:ww;1} is a free boundary problem, of which $\eta$ is part of the solution. We make the change of variables
\begin{equation}\label{def:y}
 y\mapsto \frac{y}{1+\eta(x,t)},
\end{equation}
transforming the fluid region $\{(x,y)\in\mathbb{R}^2: 0<y<1+\eta(x,t)\}$ into $\mathbb{R}\times (0,1)$, whereby ``flattening'' the free boundary. Clearly \eqref{def:y} is well defined so long as $1+\eta(x,t)>0$ for all $x\in\mathbb{R}$, particularly, when $\eta$ is small. Substituting \eqref{def:u} and \eqref{def:y} into \eqref{eqn:ww;1} we use the chain rule and make a straightforward calculation to arrive at
\ba\label{eqn:ww}
&\phi_x-\frac{y\eta_x}{1+\eta}\phi_y-u=0&&\text{for $0<y<1$,}\\
&u_x-\frac{y\eta_x}{1+\eta}u_y+\frac{1}{(1+\eta)^2}\phi_{yy}=0&&\text{for $0<y<1$,}\\
&\phi_y=0&&\text{at $y=0$},\\
&\eta_t+(u-1)\eta_x-\frac{1}{1+\eta}\phi_y=0&&\text{at $y=1$},\\
&\phi_t-u+\frac{(u-1)\eta_x}{1+\eta}\phi_y+\frac12u^2-\frac{1}{2(1+\eta)^2}\phi_y^2+\mu\eta=0\quad&&\text{at $y=1$}.
\ea 

There are other ways to fix the free boundary, for instance, reformulating the problem in terms of quantities at the fluid surface alone. The resulting equations become nonlocal (see, for instance, \cite{BMV, NS}), however, for which the periodic Evans function and other ODE techniques are inapplicable.   

\section{Stokes waves of sufficiently small amplitude}\label{sec:Stokes}

By a Stokes wave we mean a temporally stationary and spatially periodic solution of \eqref{eqn:ww}. Therefore $\phi$, $\eta$ and
\begin{equation}\label{eqn:u;stokes}
u=\phi_x-\frac{y\eta_x}{1+\eta}\phi_y,
\end{equation}
by the first equation of \eqref{eqn:ww}, satisfy
\begin{subequations}\label{eqn:stokes}
\begin{align}
&u_x-\frac{y\eta_x}{1+\eta}u_y+\frac{1}{(1+\eta)^2}\phi_{yy}=0&&\text{for $0<y<1$,}
\intertext{subject to the boundary conditions}
&\phi_y=0&&\text{at $y=0$}
\intertext{and}
&(u-1)\eta_x-\frac{1}{1+\eta}\phi_y=0&&\text{at $y=1$},\label{eqn:stokes;K}\\
&u-\frac12u^2-\frac{1}{2(1+\eta)^2}\phi_y^2-\mu\eta=0&&\text{at $y=1$},\label{eqn:stokes;D}
\end{align}
\end{subequations}
where \eqref{eqn:stokes;D} follows from \eqref{eqn:stokes;K} and
\[
u-\frac{(u-1)\eta_x}{1+\eta}\phi_y-\frac12u^2+\frac{1}{2(1+\eta)^2}\phi_y^2-\mu\eta=0
\quad\text{at $y=1$},
\]
by the fifth equation of \eqref{eqn:ww}. 

Stokes in his 1847 paper \cite{Stokes1847} (see also \cite{Stokes1880}) made significant contributions to waves of the kind, for instance, successfully approximating the solution when the amplitude is small. The existence of Stokes waves was rigorously established by Nekrasov \cite{Nekrasov1921} (see also \cite{Kuznetsov;survey}) and \mbox{Levi-Civita} \cite{Levi-Civita;stokes} in the infinite depth, Struik \cite{Struik;stokes} in the finite depth, for small amplitude, and Krasovskii \cite{Krasovskii1960, Krasovskii1961} and others for large amplitude. It would be impossible to give a complete account here and we encourage the interested reader to some excellent surveys \cite{Toland;survey}. We pause to remark that one can take a spatial dynamics approach, proposed by Kirchg\"assner \cite{Kirchgassner;ww} and further developed by others (see \cite{Groves;survey} and references therein), to work out the existence for \eqref{eqn:stokes} and \eqref{eqn:u;stokes}, for instance, in
\[
(\phi,u,\eta)\in C^k(\mathbb{R}; H^2(0,1) \times H^1(0,1)\times \mathbb{R})
\quad \text{for any $k\geq0,\in\mathbb{Z}$}.
\] 
See, for instance, \cite[Section~3]{Mielke;reduction} for some details.

We turn the attention to the small amplitude asymptotics of periodic solutions of \eqref{eqn:stokes} and \eqref{eqn:u;stokes}. In what follows, 
\[
\text{$\eps\in\mathbb{R}$ denotes the dimensionless amplitude parameter},
\]
and suppose that
\begin{equation}\label{eqn:stokes exp}
\begin{aligned}
\phi(x,y;\eps)=&\bar{\phi}_1x\eps+\phi_1(x,y)\eps+\bar{\phi}_2x\eps^2+\phi_2(x,y)\eps^2+\bar{\phi}_3x\eps^3 +\phi_3(x,y)\eps^3+O(\eps^4), \\
\eta(x;\eps)=&\eta_1(x)\eps+\eta_2(x)\eps^2+\eta_3(x)\eps^3+O(\eps^4),\\
\mu(\eps)=&\mu_0+\mu_1\eps+\mu_2\eps^2+\mu_3\eps^3+O(\eps^4)
\end{aligned}
\end{equation}
as $\eps\to 0$, where $\bar{\phi}_1,\phi_1,\bar{\phi}_2,\phi_2,\bar{\phi}_3,\phi_3,\dots$, $\eta_1,\eta_2,\eta_3,\dots$, $\mu_0,\mu_1,\mu_2,\mu_3,\dots$ are to be determined and, hence, 
\[
u(x,y;\eps)=u_1(x,y)\eps+u_2(x,y)\eps^2+u_3(x,y)\eps^3+O(\eps^4)
\]
as $\eps\to 0$, where $u_1,u_2,u_3,\dots$ can be determined in terms of  $\bar{\phi}_1,\phi_1,\bar{\phi}_2,\phi_2,\bar{\phi}_3,\phi_3,\dots$ and $\eta_1,\eta_2,\eta_3,\dots$ by \eqref{eqn:u;stokes}. We assume that $\phi_1,\phi_2,\phi_3,\dots$ and  $\eta_1,\eta_2,\eta_3,\dots$ are $T$ periodic functions of $x$, where 
\[
\text{$T=2\pi/\kappa$ and $\kappa>0$ is the wave number,} 
\]
and $\bar{\phi}_1,\bar{\phi}_2,\bar{\phi}_3,\dots$ and $\mu_0,\mu_1,\mu_2,\mu_3,\dots$ are constants. We may assume that $\phi_1,\phi_2,\phi_3,\dots$ are odd functions of $x$ and $\eta_1,\eta_2,\eta_3,\dots$ are even functions. We shall choose $\bar{\phi}_1,\bar{\phi}_2,\bar{\phi}_3,\dots$, so that $\eta_1,\eta_2,\eta_3,\dots$ are each of mean zero over one period\footnote{Alternatively, suppose that
\begin{align*}
\phi(x,y;\eps)=&\phi_1(x,y)\eps+\phi_2(x,y)\eps^2 +\phi_3(x,y)\eps^3+O(\eps^4), \\
\eta(x;\eps)=&\bar{\eta}+\eta_1(x)\eps+\eta_2(x)\eps^2+\eta_3(x)\eps^3+O(\eps^4)
\end{align*}
as $\eps\to 0$, and one may choose $\bar{\eta}$ so that $\phi_1,\phi_2, \phi_3,\dots$ are each of mean zero over one period with respect to the $x$ variable \cite{Benjamin;BF}. See also \cite{constant-vorticity}. Here we prefer \eqref{eqn:stokes exp} because the dimensionless fluid depth $=1$. }. Notice that \eqref{eqn:stokes} does not involve $\phi$ but merely its derivatives. We pause to remark that \eqref{eqn:stokes exp} converges for $|\eps|$ sufficiently small, for instance, in $H^{s+2}(\mathbb{R}/T\mathbb{Z}\times (0,1))\times H^{s+5/2}(\mathbb{R}/T\mathbb{Z})\times\mathbb{R}$ for any $s>1$ \cite{NR;anal}. Thus $\phi(\eps)$, $\eta(\eps)$ and $\mu(\eps)$ depend real analytically on $\eps$. 

Substituting \eqref{eqn:stokes exp} into \eqref{eqn:stokes} and \eqref{eqn:u;stokes}, at the order of $\eps$, we gather
\ba \label{eqn:stokes1}
&{\phi_1}_{xx}+{\phi_1}_{yy}=0&&\text{for $0<y<1$},\\
&{\phi_1}_y=0&&\text{at $y=0$},\\
&{\eta_1}_x+{\phi_1}_y=0 &&\text{at $y=1$},\\
&\bar{\phi}_1+{\phi_1}_x-\mu_0\eta_1=0\quad&&\text{at $y=1$}.
\ea
Recall that $\phi_1$ and $\eta_1$ are $2\pi/\kappa$ periodic functions of $x$, $\kappa>0$, $\phi_1$ is an odd function of $x$, $\eta_1$ is an even function and of mean zero, and $\bar{\phi}_1$ and $\mu_0$ are constants. We solve \eqref{eqn:stokes1} by means of separation of variables, for instance, to obtain
\begin{gather}
\bar{\phi}_1=0,\qquad \phi_1(x,y)=\sin(\kappa x)\cosh(\kappa y), 
\qquad \eta_1(x)=\sinh(\kappa)\cos(\kappa x), \label{def:stokes1} 
\intertext{and}
\mu_0=\kappa \coth(\kappa).\label{def:mu0}
\end{gather}
The result agrees, for instance, with \cite[(11)-(12)]{Benjamin;BF}, after suitably redefining $\eps$. We remark that \eqref{def:mu0} or, equivalently, 
\[
c_0= \sqrt{\frac{g\tanh(\kappa h)}{\kappa}},
\]
by \eqref{def:dimensionless} and \eqref{def:mu}, makes the dispersion relation of Stokes waves. Notice that $\mu_0$ is a monotonically increasing function of $\kappa$, and $\mu_0>1$ if $\kappa>0$. A solitary wave is found in the limit as $\kappa\to0$ and, hence, $\mu_0\to1$.

We proceed likewise, substituting \eqref{eqn:stokes exp} into \eqref{eqn:stokes} and \eqref{eqn:u;stokes}, and solving at higher orders of $\eps$, to successively obtain $\bar{\phi}_2,\phi_2,\bar{\phi}_3,\phi_3,\dots$, $\eta_2,\eta_3,\dots$, and $\mu_1,\mu_2,\mu_3,\dots$. The result is in Appendix~\ref{A:stokes}.

\section{The spectral stability problem}\label{sec:spec}

Let $\phi(\eps)$, $u(\eps)$, $\eta(\eps)$ and $\mu(\eps)$, for $\eps\in\mathbb{R}$ and $|\eps|\ll1$, denote a Stokes wave of sufficiently small amplitude, whose existence follows from the previous section, and \eqref{eqn:stokes exp}, \eqref{def:stokes1}, \eqref{def:mu0}, \eqref{def:stokes2}, \eqref{def:stokes3} hold true for $|\eps|$ sufficiently small. We are interested in its spectral stability and instability. 

\subsection{The linearized problem}\label{sec:linearize}

Linearizing \eqref{eqn:ww} about $\phi(\eps)$, $u(\eps)$, $\eta(\eps)$ and evaluating the result at $\mu=\mu(\eps)$ we arrive at
\ba \label{eqn:linearize}
&\phi_x-\frac{y\eta_x(\eps)}{1+\eta(\eps)}\phi_y-\frac{y\phi_y(\eps)}{1+\eta(\eps)}\eta_x+\frac{y(\eta_x\phi_y)(\eps)}{(1+\eta(\eps))^2}\eta-u=0&&\text{for $0<y<1$,}\\
&u_x-\frac{y\eta_x(\eps)}{1+\eta(\eps)}u_y-\frac{yu_y(\eps)}{1+\eta(\eps)}\eta_x+\frac{y(\eta_xu_y)(\eps)}{(1+\eta(\eps))^2}\eta+\frac{1}{(1+\eta(\eps))^2}\phi_{yy}-\frac{2\phi_{yy}(\eps)}{(1+\eta(\eps))^3}\eta=0&&\text{for $0<y<1$,}\\
&\eta_t+(u(\eps)-1)\eta_x+\eta_x(\eps)u-\frac{1}{1+\eta(\eps)}\phi_y+\frac{\phi_y(\eps)}{(1+\eta(\eps))^2}\eta=0&&\text{at $y=1$},\\
&\phi_t-u+\frac{((u-1)\phi_y)(\eps)}{1+\eta(\eps)}\eta_x+\frac{(\phi_y\eta_x)(\eps)}{1+\eta(\eps)}u+u(\eps)u+\mu(\eps)\eta=0&&\text{at $y=1$},\\
&\phi_y=0&&\text{at $y=0$},
\ea 
where the fourth equation of \eqref{eqn:linearize} follows from the linearization of the fifth equation of \eqref{eqn:ww} and \eqref{eqn:stokes;K}. 
Seeking a solution of \eqref{eqn:linearize} of the form 
\[
\begin{pmatrix}\phi(x,y,t) \\ u(x,y,t) \\ \eta(x,t) \end{pmatrix}
=e^{\lambda t}\begin{pmatrix}\phi(x,y) \\ u(x,y) \\ \eta(x) \end{pmatrix},
\quad\lambda\in\mathbb{C},
\] 
we arrive at
\begin{subequations}\label{eqn:spec}
\begin{align}
&\phi_x-\frac{y\eta_x(\eps)}{1+\eta(\eps)}\phi_y-\frac{y\phi_y(\eps)}{1+\eta(\eps)}\eta_x+\frac{y(\eta_x\phi_y)(\eps)}{(1+\eta(\eps))^2}\eta-u=0&&\text{for $0<y<1$,} \label{eqn:spec;u}\\
&\begin{aligned}\label{eqn:spec;phi}
u_x-&\frac{y\eta_x(\eps)}{1+\eta(\eps)}u_y-\frac{yu_y(\eps)}{1+\eta(\eps)}\eta_x\\
+&\frac{1}{(1+\eta(\eps))^2}\phi_{yy}+\Big(\frac{y(\eta_xu_y)(\eps)}{(1+\eta(\eps))^2}-\frac{2\phi_{yy}(\eps)}{(1+\eta(\eps))^3}\Big)\eta=0
\end{aligned}&&\text{for $0<y<1$,}\\
&\lambda\eta+(u(\eps)-1)\eta_x+\eta_x(\eps)u-\frac{1}{1+\eta(\eps)}\phi_y+\frac{\phi_y(\eps)}{(1+\eta(\eps))^2}\eta=0&&\text{at $y=1$}\label{eqn:spec;K}
\intertext{and}
&\begin{aligned}\label{eqn:spec;D}
\lambda\phi+&(u(\eps)-1)u+\frac{\phi_y(\eps)}{(1+\eta(\eps))^2}\phi_y \\
+&\Big(\mu(\eps)-\lambda\frac{\phi_y(\eps)}{1+\eta(\eps)}-\frac{\phi_y(\eps)^2}{(1+\eta(\eps))^3}\Big)\eta=0
\end{aligned}&&\text{at $y=1$},\\
&\phi_y=0&&\text{at $y=0$},\label{eqn:spec;bdry}
\end{align}
\end{subequations}
where \eqref{eqn:spec;D} follows from \eqref{eqn:spec;K} and
\[
\lambda\phi-u+\frac{((u-1)\phi_y)(\eps)}{1+\eta(\eps)}\eta_x+\frac{(\phi_y\eta_x)(\eps)}{1+\eta(\eps)}u+u(\eps)u+\mu(\eps)\eta=0\quad\text{at $y=1$},
\] 
by the fourth equation of \eqref{eqn:linearize}. Roughly speaking, $\lambda\in\mathbb{C}$ is in the spectrum if \eqref{eqn:spec} admits a nontrivial bounded solution in some function space, and a Stokes wave of sufficiently small amplitude is spectrally stable if the spectrum does not intersect the right half plane of $\mathbb{C}$ for $\eps\in\mathbb{R}$ and $|\eps|\ll1$. We make these precise in Section~\ref{sec:L}.

Notice that \eqref{eqn:spec;D} is not autonomous in $x$. Thus we introduce
\begin{equation}\label{def:tildeu}
\begin{aligned}
\upsilon=-\left(\lambda\phi+(u(\cdot,1;\eps)-1)u+\frac{\phi_y(\cdot,1;\eps)}{(1+\eta(\eps))^2}\phi_y\right)\left(\mu(\eps)-\lambda\frac{\phi_y(\cdot,1;\eps)}{1+\eta(\eps)}-\frac{\phi_y(\cdot,1;\eps)^2}{(1+\eta(\eps))^3}\right)^{-1},
\end{aligned}
\end{equation}
so that \eqref{eqn:spec;D} becomes 
\begin{equation}\label{eqn:bdry;zeta}
\eta-\upsilon=0\quad\text{at $y=1$}.
\end{equation}
Clearly \eqref{def:tildeu} is well defined for $\eps\in\mathbb{R}$ and $|\eps|\ll1$, provided that $|{\rm Re}\,\lambda|$ is bounded. Conversely
\begin{equation} \label{def:u(tildeu)}
\begin{aligned}
u=\left(\Big(\mu(\eps)-\lambda\frac{\phi_y(\cdot,1;\eps)}{1+\eta(\eps)}-\frac{\phi_y(\cdot,1;\eps)^2}{(1+\eta(\eps))^3}\Big)\upsilon+\lambda\phi+\frac{\phi_y(\cdot,1;\eps)}{(1+\eta(\eps))^2}\phi_y\right)\big(1-u(\cdot,1;\eps)\big)^{-1}
\end{aligned}
\end{equation}
is well defined for $\eps\in\mathbb{R}$ and $|\eps|\ll1$.

We proceed to rewrite \eqref{eqn:spec;u}-\eqref{eqn:spec;K} for $\phi$, $\upsilon$ and $\eta$. We begin by replacing $u$ in \eqref{eqn:spec;u} by the right side of \eqref{def:u(tildeu)}, and $u(\cdot,1)$ in \eqref{eqn:spec;K} by the right side of \eqref{def:u(tildeu)}, evaluated at $y=1$. We replace $u_y$ in \eqref{eqn:spec;phi} by 
\begin{equation} \label{def:u(tildeu)_y}
u_y=\left(\Big(\mu(\eps)-\lambda\frac{\phi_y(\cdot,1;\eps)}{1+\eta(\eps)}-\frac{\phi_y(\cdot,1;\eps)^2}{(1+\eta(\eps))^3}\Big)\upsilon_y+\lambda\phi_y+\frac{\phi_y(\cdot,1;\eps)}{(1+\eta(\eps))^2}\phi_{yy}\right)\big(1-u(\cdot,1;\eps)\big)^{-1},
\end{equation}
by \eqref{def:u(tildeu)}, and likewise $u_x$ in \eqref{eqn:spec;phi} by 
\begin{equation} \label{def:u_x}
\begin{aligned}
u_x=&\Big(\Big(\mu(\eps)-\lambda\frac{\phi_y(\cdot,1;\eps)}{1+\eta(\eps)}-\frac{\phi_y(\cdot,1;\eps)^2}{(1+\eta(\eps))^3}\Big)\big(1-u(\cdot,1;\eps)\big)^{-1}\Big)_x\upsilon \\
&+\Big(\mu(\eps)-\lambda\frac{\phi_y(\cdot,1;\eps)}{1+\eta(\eps)}-\frac{\phi_y(\cdot,1;\eps)^2}{(1+\eta(\eps))^3}\Big)\big(1-u(\cdot,1;\eps)\big)^{-1}\upsilon_x \\
&+\lambda\big(\big(1-u(\cdot,1;\eps)\big)^{-1}\big)_x\phi
+\lambda\big(1-u(\cdot,1;\eps)\big)^{-1}\phi_x \\
&+\Big(\frac{\phi_y(\cdot,1;\eps)}{(1+\eta(\eps))^2}\big(1-u(\cdot,1;\eps)\big)^{-1}\Big)_x\phi_y
+\frac{\phi_y(\cdot,1;\eps)}{(1+\eta(\eps))^2}\big(1-u(\cdot,1;\eps)\big)^{-1}\phi_{xy}.
\end{aligned}
\end{equation}
Also we rewrite $\phi_x$ in \eqref{def:u_x} in terms of $\phi$, $\upsilon$ and $\eta$ using \eqref{eqn:spec;u} and \eqref{def:u(tildeu)} and, likewise, $\phi_{xy}$ in \eqref{def:u_x} in terms of $\phi$, $\upsilon$ and $\eta$ by differentiating \eqref{eqn:spec;u} with respect to the $y$ variable and using \eqref{def:u(tildeu)_y}.

\subsection{Spectral stability and instability}\label{sec:L}

Let $\mathbf{u}=\begin{pmatrix} \phi \\ \upsilon \\ \eta \end{pmatrix}$, 
and we write \eqref{eqn:spec} as
\be \label{eqn:LB}
\mathbf{u}_x=\mathbf{L}(\lambda)\mathbf{u}+\mathbf{B}(x;\lambda,\eps)\mathbf{u},
\ee 
where $\lambda\in\mathbb{C}$ and $\eps\in\mathbb{R}$ and $|\eps|\ll1$,
\[
\mathbf{L}(\lambda): {\rm dom}(\mathbf{L}) \subset Y \to Y\quad\text{and}\quad
\mathbf{B}(x;\lambda,\eps): 
\mathbb{R}\times {\rm dom}(\mathbf{L}) \subset \mathbb{R}\times Y \to Y,
\] 
\be \label{def:L}
\mathbf{L}(\lambda)\mathbf{u}=\begin{pmatrix} 
\lambda\phi+\mu_0 \upsilon \\ 
-\mu_0^{-1}(\phi_{yy}+\lambda^2\phi+\mu_0\lambda \upsilon)\\
[\lambda\eta-\phi_y]_{y=1}\end{pmatrix},
\ee  
\begin{equation}\label{def:Y}
Y=H^1(0,1)\times L^2(0,1)\times \mathbb{C}
\end{equation}
and
\begin{equation}\label{def:dom}
{\rm dom}(\mathbf{L})
=\{\mathbf{u}\in H^2(0,1)\times H^1(0,1)\times \mathbb{C}:\eta-\upsilon(1)=0, \phi_y(0)=0\}.
\end{equation}
Notice that $\mathbf{L}(\lambda)$ is the leading part of \eqref{eqn:spec;u}-\eqref{eqn:spec;K}, by \eqref{eqn:stokes exp}, \eqref{def:mu0}, \eqref{def:u(tildeu)}, \eqref{def:u(tildeu)_y}, \eqref{def:u_x}, and $\mathbf{B}(x;\lambda,\eps)$ is a linear operator. Notice that if $\mathbf{u}\in{\rm dom}(\mathbf{L})$ then \eqref{eqn:bdry;zeta} and \eqref{eqn:spec;bdry} hold true. Also $\mathbf{B}(x;\lambda,0)=\mathbf{0}$. Therefore when $\eps=0$, \eqref{eqn:LB} becomes $\mathbf{u}_x=\mathbf{L}(\lambda)\mathbf{u}$. 
We remark that $\mathbf{L}(\lambda)$ and $\mathbf{B}(x;\lambda,\eps)$ depend analytically on $\lambda$, and $\mathbf{B}(x;\lambda,\eps)$ depends analytically on $\eps$. Although $\phi(\eps)$ is not periodic in the $x$ variable (see the first equation of \eqref{eqn:stokes exp}), \eqref{eqn:spec} involves merely its derivatives, whence $\mathbf{B}(x;\lambda,\eps)$ is $T(=2\pi/\kappa)$ periodic in $x$. Also $\mathbf{B}(x;\lambda,\eps)$ is smooth in $x$. Our proofs do not involve all the details of $\mathbf{B}(x;\lambda,\eps)$, whence we do not include the formula here. Clearly ${\rm dom}(\mathbf{L})$ is dense in $Y$. 

Let 
\[
\mathcal{L}(\lambda,\eps):{\rm dom}(\mathcal{L})\subset X \to X,
\] 
where 
\begin{equation}\label{def:operator}
\mathcal{L}(\lambda,\eps)\mathbf{u}=\mathbf{u}_x-(\mathbf{L}(\lambda)+\mathbf{B}(x;\lambda,\eps))\mathbf{u},
\end{equation}
\begin{equation}\label{def:X}
X=L^2(\mathbb{R};Y)\quad\text{and}\quad 
{\rm dom}(\mathcal{L})=H^1(\mathbb{R};Y)\bigcap L^2(\mathbb{R};{\rm dom}(\mathbf{L}))
\end{equation}
is dense in $X$, so that \eqref{eqn:LB} becomes 
\begin{equation}\label{eqn:L}
\mathcal{L}(\lambda,\eps)\mathbf{u}=0.
\end{equation}
We regard $\mathcal{L}(\eps)$ as $\mathcal{L}(\lambda,\eps)$, parametrized by $\lambda\in\mathbb{C}$. 

\begin{definition}[The spectrum of $\mathcal{L}(\eps)$]\rm
For $\eps\in\mathbb{R}$ and $|\eps|\ll1$, 
\[
{\rm spec}(\mathcal{L}(\eps))=\{\lambda\in\mathbb{C}:
\text{$\mathcal{L}(\lambda,\eps):{\rm dom}(\mathcal{L})\subset X\to X$ is not invertible}\}.
\]
\end{definition}

We pause to remark that $\mathcal{L}(\lambda,\eps)$ makes sense provided that  $|\eps|\ll 1$. 

\begin{definition}[Spectral stability and instability]\rm
A Stokes wave of sufficiently small amplitude is said to be {\em spectrally stable} if 
\[
{\rm spec}(\mathcal{L}(\eps)) \subset \{\lambda\in\mathbb{C}: {\rm Re}\,\lambda\leq 0\}
\]
for $\eps\in\mathbb{R}$ and $|\eps|\ll1$, and {\em spectrally unstable} otherwise.
\end{definition}

Since $\mathbf{B}(x;\lambda,\eps)$ and, hence, $\mathcal{L}(\lambda,\eps)$ are $T(=2\pi/\kappa)$ periodic in $x$, by Floquet theory, the point spectrum of $\mathcal{L}(\eps)$ is empty. Moreover $\lambda$ is in the essential spectrum of $\mathcal{L}(\eps)$
if and only if \eqref{eqn:LB} admits a nontrivial solution $\mathbf{u}\in L^\infty(\mathbb{R};Y)$ satisfying
\[
\mathbf{u}(x+T)=e^{ik T}\mathbf{u}(x)
\quad \text{for some $k\in\mathbb{R}$},\quad\text{the Floquet exponent}.
\]
See, for instance, \cite{Gardner;evans1,OS;evans} for details.

In what follows, the asterisk means complex conjugation. 

\begin{lemma}[Symmetries of the spectrum]\label{lem:symm} 
If $\lambda\in{\rm spec}(\mathcal{L}(\eps))$ then $\lambda^*,-\lambda\in{\rm spec}(\mathcal{L}(\eps))$ and, hence, $-\lambda^*\in{\rm spec}(\mathcal{L}(\eps))$. 
In other words, ${\rm spec}(\mathcal{L}(\eps))$ is symmetric about the real and imaginary axes.
\end{lemma}

\begin{remark*}\rm
A Stokes wave of sufficiently small amplitude is spectrally stable if and only if ${\rm spec}(\mathcal{L}(\eps))\subset i\mathbb{R}$ for $\eps\in\mathbb{R}$ and $|\eps|\ll1$.
\end{remark*}

The proof of Lemma~\ref{lem:symm} is in Appendix~\ref{A:symm}. The symmetries of the spectrum also follow from that the water wave problem (see \eqref{eqn:ww;1}) is Hamiltonian \cite{Zakharov;ww}. 

In Section~\ref{sec:eps=0} we focus the attention on $\eps=0$ and define the eigenspace of $\mathbf{L}(\lambda): {\rm dom}(\mathbf{L})\subset Y\to Y$ associated with its finitely many and purely imaginary eigenvalues. In Section~\ref{sec:reduction} we turn to $\eps\neq0$ and take a center manifold reduction approach (see \cite{Mielke;reduction}, among others) to reduce \eqref{eqn:LB} to finite dimensions (see \eqref{eqn:LB;u1}), whereby we introduce Gardner's periodic Evans function (see, for instance, \cite{Gardner;evans1}). In Section~\ref{sec:BF} and Section~\ref{sec:high-freq} we make the power series expansion of the periodic Evans function to locate and track the spectrum of $\mathcal{L}(\eps)$ for $\eps\in\mathbb{R}$ and $|\eps|\ll1$. 

\subsection{The spectrum of \texorpdfstring{$\mathcal{L}(0)$}{Lg}. The reduced space}\label{sec:eps=0}

When $\eps=0$, 
$\lambda\in{\rm spec}(\mathcal{L}(0))$ if and only if 
\[
ik\mathbf{u}=\mathbf{L}(\lambda)\mathbf{u},
\quad\text{where}\quad \mathbf{L}(\lambda):{\rm dom}(\mathbf{L})\subset Y\to Y
\] 
is in \eqref{def:L}, \eqref{def:Y}, \eqref{def:dom}, admits a nontrivial solution for some $k\in\mathbb{R}$. In other words, $ik$ is an eigenvalue of $\mathbf{L}(\lambda)$. We pause to remark that $\mathbf{L}(\lambda)$ has compact resolvent, so that the spectrum consists of discrete eigenvalues with finite multiplicities. A straightforward calculation reveals that 
\begin{equation}\label{eqn:sigma}
\lambda=i\sigma,\quad\text{where}\quad
\sigma\in\mathbb{R}\quad\text{and}\quad(\sigma-k)^2=\mu_0k\tanh(k).
\end{equation}
Therefore ${\rm spec}(\mathcal{L}(0))=i\mathbb{R}$, implying that the Stokes wave of zero amplitude is spectrally stable. 

\begin{figure}[htbp]
\begin{center}
\includegraphics[scale=0.45]{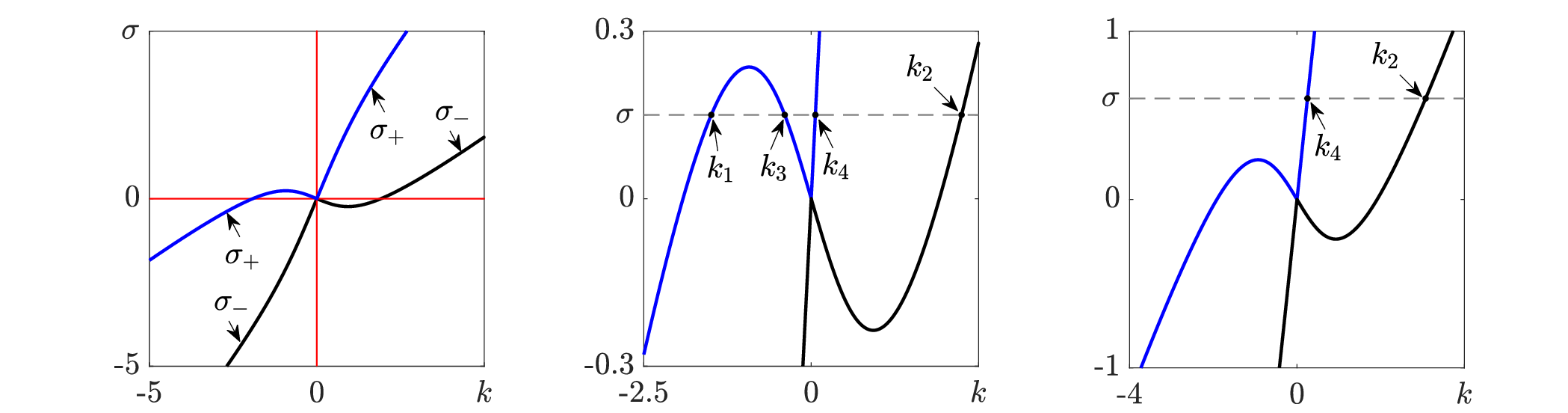}
\caption{Left: The graphs of $\sigma_+(k)$ (blue) and $\sigma_-(k)$ (black) for $\mu_0=2$. Middle: When $0<\sigma<\sigma_c$, $\sigma_\pm(k)=\sigma$ have four zeros $k_j(\sigma)$, $j=1,2,3,4$. Right: When $\sigma>\sigma_c$, $\sigma_\pm(k)=\sigma$ have two zeros $k_j(\sigma)$, $j=2,4$.}\label{fig:dispersion}
\end{center}
\end{figure}

Let
\begin{equation}\label{def:sigma}
\sigma_\pm(k)=k\pm\sqrt{\mu_0 k\tanh(k)},\quad\text{where}\quad k\in \mathbb{R}.
\end{equation}
The left panel of Figure~\ref{fig:dispersion} shows the graphs of $\sigma_\pm$ for $\mu_0=2$. Let 
\[
\sigma_c=\sigma_+(k_c), \quad\text{such that}\quad \frac{d\sigma_+}{dk}(k_c)=0. 
\]
In other words, $\sigma_c$ is a (unique) critical value of $\sigma_+$. By symmetry, $-\sigma_c=\sigma_-(-k_c)$ is a critical value of $\sigma_-$. Below it suffices to take $\sigma\geq0$. Notice that:
\begin{enumerate}
\item When $\sigma=0$, $\sigma_+(k)=0$ has one simple zero at $k=0$, and one negative simple zero at $k=-\kappa$, by \eqref{def:mu0}, and $\sigma_-(k)=0$ has one simple zero at $k=0$ and one positive simple zero at $k=\kappa$;
\item When $0<\sigma<\sigma_c$, $\sigma_+(k)=\sigma$ 
has three simple zeros, two negative and one positive, and $\sigma_-(k)=\sigma$ has one positive simple zero; 
\item When $\sigma=\sigma_c$, $\sigma_+(k)=\sigma_c$ has one negative double zero at $k=k_c$ and one positive simple zero, and $\sigma_-(k)=\sigma_c$ has one positive simple zero; 
\item When $\sigma>\sigma_c$, $\sigma_+(k)=\sigma$ has one positive simple zero, and $\sigma_-(k)=\sigma$ has one positive simple zero.
\end{enumerate} 
Let $k_2(\sigma)>0$ denote the simple zero of $\sigma_-(k)=\sigma(\geq0)$, and let $k_4(\sigma)>0$ be the simple zero of $\sigma_+(k)=\sigma(>0)$, and $k_4(0)=0$. When $0\leq\sigma\leq\sigma_c$, let $k_1(\sigma)\leq k_3(\sigma)\leq0$ be the other two zeros of $\sigma_+(k)=\sigma$. See the middle and right panels of Figure~\ref{fig:dispersion}. Therefore:
\begin{enumerate}
\item When $\sigma=0$, $k_j(0)=(-1)^j\kappa$, $j=1,2$, and $k_j(0)=0$, $j=3,4$; 
\item When $0<\sigma<\sigma_c$, $\sigma_-(k_2)=\sigma_+(k_j)=\sigma$, $j=1,3,4$, and $k_1<k_3<0<k_4<k_2$;
\item When $\sigma=\sigma_c$, $k_j(\sigma_c)=k_c$, $j=1,3$, and $\sigma_-(k_2)=\sigma_+(k_4)=\sigma_c$, $k_1=k_3<0<k_4<k_2$;
\item When $\sigma>\sigma_c$, $\sigma_-(k_2)=\sigma_+(k_4)=\sigma$ and $0<k_4<k_2$.
\end{enumerate}


\begin{lemma}[Spectrum of $\mathbf{L}({i\sigma})$]\label{lem:eps=0} 
When $\sigma=0$, $ik_j(0)=(-1)^ji \kappa$, $j=1,2$, are simple eigenvalues of $\mathbf{L}(0):{\rm dom}(\mathbf{L})\subset Y\to Y$, and
\begin{equation}\label{def:phi12}
\ker(\mathbf{L}(0)-ik_j(0)\mathbf{1})=\ker((\mathbf{L}(0)-ik_j(0)\mathbf{1})^2)={\rm span}\{\boldsymbol{\phi}_j(0)\},\quad
\boldsymbol{\phi}_j(0)=\begin{pmatrix}
\mu_0\cosh(k_j(0)y)\\ik_j(0)\cosh(k_j(0)y)\\ik_j(0)\cosh(k_j(0))\end{pmatrix},
\end{equation}
where $\mathbf{1}$ denotes the identity operator.
Also $ik_j(0)=0$, $j=3,4$, is an eigenvalue of $\mathbf{L}(0)$ with algebraic multiplicity $2$ and geometric multiplicity $1$, and 
\[
\ker(\mathbf{L}(0)^2)=\ker(\mathbf{L}(0)^3)
={\rm span}\{\boldsymbol{\phi}_3(0), \boldsymbol{\phi}_4(0)\},
\]
where
\begin{equation}\label{def:phi34}
\boldsymbol{\phi}_3(0)=\begin{pmatrix} 0 \\ 1 \\ 1 \end{pmatrix}\quad\text{and}\quad
\boldsymbol{\phi}_4(0)=\begin{pmatrix} \mu_0 \\ 0 \\ 0 \end{pmatrix}.
\end{equation}

When $0<\sigma<\sigma_c$, $ik_j(\sigma)$, $j=1,2,3,4$, are simple eigenvalues of $\mathbf{L}(i\sigma)$, and
\begin{equation}\label{def:phi1-4}
\ker(\mathbf{L}(i\sigma)-ik_j(\sigma)\mathbf{1})=\ker((\mathbf{L}(i\sigma)-ik_j(\sigma)\mathbf{1})^2)
={\rm span}\{\boldsymbol{\phi}_j(\sigma)\},\quad
\boldsymbol{\phi}_j(\sigma)=\begin{pmatrix}
\mu_0\cosh(k_j(\sigma)y)\\i(k_j(\sigma)-\sigma)\cosh(k_j(\sigma)y)\\i(k_j(\sigma)-\sigma)\cosh(k_j(\sigma))\end{pmatrix}.
\end{equation}

When $\sigma>\sigma_c$, $ik_j(\sigma)$, $j=2,4$, are simple eigenvalues of $\mathbf{L}(i\sigma)$, and \eqref{def:phi1-4} holds true. 

When $\sigma=\sigma_c$, $ik_j(\sigma_c)$, $j=2,4$, are simple eigenvalues of $\mathbf{L}(i\sigma_c)$, and \eqref{def:phi1-4} holds true. Also $ik_c$ is an eigenvalue with algebraic multiplicity $2$ and geometric multiplicity $1$, and
\[
\ker((\mathbf{L}(i\sigma_c)-ik_c\mathbf{1})^2)=\ker((\mathbf{L}(i\sigma_c)-ik_c\mathbf{1})^3)
={\rm span}\{\boldsymbol{\phi}_1(\sigma_c), \boldsymbol{\phi}_3(\sigma_c)\},
\]
where
\begin{equation}\label{def:phi13}
\boldsymbol{\phi}_1(\sigma_c)=\begin{pmatrix}
\mu_0\cosh(k_cy) \\ i(k_c-\sigma_c)\cosh(k_cy) \\ i(k_c-\sigma_c)\cosh(k_c)\end{pmatrix}
\quad\text{and}\quad 
\boldsymbol{\phi}_3(\sigma_c)=\begin{pmatrix}
-i\mu_0y\sinh(k_cy) \\ \cosh(k_cy)+(k_c-\sigma_c)y\sinh(k_cy) \\ \cosh(k_c)+(k_c-\sigma_c)\sinh(k_c)\end{pmatrix}.
\end{equation}
\end{lemma}

The proof is in Appendix~\ref{A:eps=0}. 

\begin{definition}[Eigenspace and projection]\label{def:proj}\rm
Let $\sigma\geq0$, and $Y(\sigma)$ denote the ({\em generalized}) {\em eigenspace} of $\mathbf{L}(i\sigma): {\rm dom}(\mathbf{L})\subset Y\to Y$ associated with its finitely many and purely imaginary eigenvalues. Let 
\[
\boldsymbol{\Pi}(\sigma): {\rm dom}(\mathbf{L})\subset Y \to Y(\sigma)
\]
be the {\em projection} of ${\rm dom}(\mathbf{L})$ onto $Y(\sigma)$, which commutes with $\mathbf{L}(i\sigma)$. 
\end{definition}

Lemma~\ref{lem:eps=0} says that:
\begin{enumerate}
\item When $\sigma=0$, $Y(0)={\rm span}\{\boldsymbol{\phi}_j(0):j=1,2,3,4\}$, where $\boldsymbol{\phi}_j(0)$ is in \eqref{def:phi12} and \eqref{def:phi34};
\item When $0<\sigma<\sigma_c$, $Y(\sigma)={\rm span}\{\boldsymbol{\phi}_j(\sigma):j=1,2,3,4\}$, where $\boldsymbol{\phi}_j(\sigma)$ is in \eqref{def:phi1-4};
\item When $\sigma=\sigma_c$, $Y(\sigma_c)={\rm span}\{\boldsymbol{\phi}_j(\sigma_c):j=1,2,3,4\}$, where $\boldsymbol{\phi}_j(\sigma_c)$ is in \eqref{def:phi1-4} and \eqref{def:phi13}; 
\item When $\sigma>\sigma_c$, $Y(\sigma)={\rm span}\{\boldsymbol{\phi}_j(\sigma):j=2,4\}$, where $\boldsymbol{\phi}_j(\sigma)$ is in \eqref{def:phi1-4}.
\end{enumerate}
The formulae of $\boldsymbol{\Pi}(\sigma)$ are in Section~\ref{sec:proj}. 

By symmetry, $ik_j(\sigma)=-ik_j(-\sigma)$, whereby Lemma~\ref{lem:eps=0} and Definition~\ref{def:proj} extend to $\sigma<0$. 

\subsection{Reduction of the spectral problem. The periodic Evans function}\label{sec:reduction}

We turn the atten-tion to $|\eps|\neq0,\ll1$. Let
\[
\lambda=i\sigma+\delta,\quad\text{where}\quad \sigma\in\mathbb{R}, 
\quad\delta\in \mathbb{C}\quad\text{and}\quad |\delta|\ll1,
\] 
and we rewrite \eqref{eqn:LB} as 
\be \label{eqn:LB;delta}
\mathbf{u}_x=
\mathbf{L}(i\sigma)\mathbf{u}+(\mathbf{L}(i\sigma+\delta)-\mathbf{L}(i\sigma))\mathbf{u}
+\mathbf{B}(x;i\sigma+\delta,\eps)\mathbf{u}(x)
=:\mathbf{L}(i\sigma)\mathbf{u}+\mathbf{B}(x;\sigma,\delta,\eps)\mathbf{u}(x),
\ee 
where $\mathbf{L}(i\sigma):{\rm dom}(\mathbf{L})\subset Y\to Y$ is in \eqref{def:L}, \eqref{def:Y}, \eqref{def:dom}, and 
\[
\mathbf{B}(x;\sigma,\delta,\eps):\mathbb{R}\times {\rm dom}(\mathbf{L})\subset \mathbb{R}\times Y\to Y.
\]
Notice that $\mathbf{B}(x;\sigma,\delta,\eps)$ is smooth and $T(=2\pi/\kappa)$ periodic in $x$, and it depends analytically on $\sigma$, $\delta$ and $\eps$. Also $\mathbf{B}(x;\sigma,0,0)=\mathbf{0}$. Our proofs do not involve all the details of $\mathbf{B}(x;\sigma,\delta,\eps)$, and rather its leading order terms as $\delta,\eps\to 0$, whence we do not include the formula here. But see, for instance, \eqref{def:B;exp0}, \eqref{def:B;expH} and Appendix~\ref{A:Bexp}.

For $\mathbf{u}(x)\in {\rm dom}(\mathbf{L})$, $x\in\mathbb{R}$, 
let 
\begin{equation*}
\mathbf{v}(x)=\boldsymbol{\Pi}(\sigma)\mathbf{u}(x)\quad\text{and}\quad 
\mathbf{w}(x)=(\mathbf{1}-\boldsymbol{\Pi}(\sigma))\mathbf{u}(x),
\end{equation*}
where $\mathbf{1}$ denotes the identity operator, 
and \eqref{eqn:LB;delta} becomes
\ba\label{eqn:LB;u12}
{\mathbf{v}}_x=&\mathbf{L}(i\sigma)\mathbf{v}
+\boldsymbol{\Pi}(\sigma)\mathbf{B}(x;\sigma,\delta,\eps)(\mathbf{v}(x)+\mathbf{w}(x)),\\
{\mathbf{w}}_x=&\mathbf{L}(i\sigma)\mathbf{w}
+(\mathbf{1}-\boldsymbol{\Pi}(\sigma))\mathbf{B}(x;\sigma,\delta,\eps)(\mathbf{v}(x)+\mathbf{w}(x)).
\ea 
Recall from the previous subsection that
\[
Y(\sigma)=\begin{cases}{\rm span}\{\boldsymbol{\phi}_j(\sigma):j=1,2,3,4\} 
\quad&\text{if $|\sigma|\leq\sigma_c$}\\
{\rm span}\{\boldsymbol{\phi}_j(\sigma):j=2,4\} &\text{if $|\sigma|>\sigma_c$,}
\end{cases}
\] 
where $\boldsymbol{\phi}_j(\sigma)$ is in Lemma~\ref{lem:eps=0}, particularly, $Y(\sigma)$ is finite dimensional. 
We appeal to, for instance, \cite[Theorem~1]{Mielke;reduction}, and for $\sigma\in\mathbb{R}$, $\delta\in\mathbb{C}$ and $|\delta|\ll1$, for $\eps\in\mathbb{R}$ and $|\eps|\ll1$, there exists 
\begin{equation}\label{def:u2}
\mathbf{w}(x,\mathbf{v}(x);\sigma,\delta,\eps): \mathbb{R}\times Y(\sigma) \to {\rm dom}(\mathbf{L})
\end{equation}
such that $\mathbf{v}+\mathbf{w}$ makes a bounded solution of \eqref{eqn:LB;delta} and, hence, \eqref{eqn:LB}, provided that
\be \label{eqn:LB;u1}
{\mathbf{v}}_x=\mathbf{L}(i\sigma)\mathbf{v}
+\boldsymbol{\Pi}(\sigma)\mathbf{B}(x;\sigma,\delta,\eps)
(\mathbf{v}(x)+\mathbf{w}(x,\mathbf{v}(x);\sigma,\delta,\eps)).
\ee 
Therefore we turn \eqref{eqn:LB;delta}, for which $\mathbf{u}(x)\in{\rm dom}(\mathbf{L})$, $x\in\mathbb{R}$, and $\dim({\rm dom}(\mathbf{L}))=\infty$, into \eqref{eqn:LB;u1}, for which $\mathbf{v}(x)\in Y(\sigma)$ and $\dim(Y(\sigma))<\infty$. 

For $\mathbf{v}(x)\in Y(\sigma)$, $x\in\mathbb{R}$, let 
\begin{equation*}
\mathbf{v}(x)=\sum_j a_j(x)\boldsymbol{\phi}_j(\sigma)\quad\text{and}\quad \mathbf{a}(x)=(a_j(x)),
\end{equation*}
and we may further rewrite \eqref{eqn:LB;u1} as
\be\label{eqn:A}
\mathbf{a}_x=\mathbf{A}(x;\sigma,\delta,\eps)\mathbf{a},
\ee 
where $\mathbf{A}(x;\sigma,\delta,\eps)$ is a square matrix of order $4$ if $|\sigma|\leq \sigma_c$, and order $2$ if $|\sigma|>\sigma_c$. Notice that $\mathbf{A}(x;\sigma,\delta,\eps)$ is smooth and $T(=2\pi/k)$ periodic in $x$, and it depends analytically on $\sigma$, $\delta$ and $\eps$. Our proofs do not involve all the details of $\mathbf{A}(x;\sigma,\delta,\eps)$ and, rather, its leading order terms as $\delta,\eps\to0$, whence we do not include the formula here. But see, for instance, \eqref{eqn:a}, \eqref{def:f0} and \eqref{eqn:aH}, \eqref{def:fH}. By Floquet theory, if $\mathbf{a}$ is a bounded solution of \eqref{eqn:A} then, necessarily, 
\[
\mathbf{a}(x+T)=e^{ik T}\mathbf{a}(x)\quad\text{for some $k\in\mathbb{R}$},
\quad\text{where $T=2\pi/\kappa$}
\]
is the period of the periodic wave. 
 
Following \cite{Gardner;evans1,Gardner;evans2} and others we take a periodic Evans function approach. 

\begin{definition}[The periodic Evans function]\label{def:Evans}\rm
For $\lambda=i\sigma+\delta$, $\sigma\in \mathbb{R}$, $\delta\in\mathbb{C}$ and $|\delta|\ll1$, for $\eps\in\mathbb{R}$ and $|\eps|\ll1$, let $\mathbf{X}(x;\sigma,\delta,\eps)$ denote the fundamental matrix of \eqref{eqn:A} such that $\mathbf{X}(0;\sigma,\delta,\eps)=\mathbf{I}$, where $\mathbf{I}$ is the identity matrix. Let $\mathbf{X}(T;\sigma,\delta,\eps)$ be the {\em monodromy matrix} for \eqref{eqn:A}, and for $k\in\mathbb{R}$,
\be\label{def:Delta}
\Delta(\lambda,k;\eps)=\det(e^{ikT}\mathbf{I}-\mathbf{X}(T;\sigma,\delta,\eps))
\ee 
the {\em periodic Evans function}, where $T=2\pi/\kappa$ is the period of the Stokes wave of sufficiently small amplitude.
\end{definition}

Since $\mathbf{A}(x;\sigma,\delta,\eps)$ depends analytically on $\sigma$, $\delta$ and $\epsilon$ for any $x\in\mathbb{R}$, 
so do $\mathbf{X}(T;\sigma,\delta,\eps)$ and, hence, $\Delta(\lambda,k;\eps)$ depends analytically on $\lambda$, $k$ and $\eps$, where $\lambda=i\sigma+\delta$ and $k\in\mathbb{R}$.  
By Floquet theory and, for instance, \cite[Theorem~1]{Mielke;reduction}, for $\eps\in\mathbb{R}$ and $|\eps|\ll1$, $\lambda\in{\rm spec}(\mathcal{L}(\eps))$ if and only if 
\[
\Delta(\lambda,k;\eps)=0\quad\text{for some $k\in\mathbb{R}$}.
\]
See, for instance, \cite{Gardner;evans1} for more details. 

\begin{remark*}\rm
A Stokes wave of sufficiently small amplitude is spectrally stable if and only if 
\begin{equation*}
{\rm spec}(\mathcal{L}(\eps))=\{\lambda\in\mathbb{C}: \Delta(\lambda,k;\eps)=0\quad\text{for some $k\in\mathbb{R}$}\}\subset i\mathbb{R}
\end{equation*}
for $\eps\in\mathbb{R}$ and $|\eps|\ll1$.
\end{remark*}

In what follows, we identify ${\rm spec}(\mathcal{L}(\eps))$ with the zeros of the periodic Evans function.

One should not expect to be able to evaluate the periodic Evans function except for few cases, for instance, completely integrable PDEs. When $\eps\in\mathbb{R}$ and $|\eps|\ll1$, on the other hand, we shall use the result of Section~\ref{sec:Stokes} and determine \eqref{def:Delta}. 

\subsection{Computation of \texorpdfstring{$\boldsymbol{\Pi}(\sigma)$}{Lg}}\label{sec:proj}

We begin by constructing the adjoint of $\mathbf{L}(\lambda):{\rm dom}(\mathbf{L})\subset Y\to Y$. For $\mathbf{u}_1:=\begin{pmatrix}\phi_1\\ \upsilon_1 \\ \eta_1\end{pmatrix}, 
\mathbf{u}_2:=\begin{pmatrix}\phi_2\\ \upsilon_2\\ \eta_2 \end{pmatrix} \in Y$,  
we define the inner product as
\begin{equation}\label{def:inner}
\langle \mathbf{u}_1, \mathbf{u}_2\rangle
=\int^1_0(\phi_1\phi_2^*+{\phi_1}_y{\phi_2}_y^*)~dy+\int^1_0 \upsilon_1\upsilon_2^*~dy+\eta_1\eta_2^*,
\end{equation}
where the asterisk means complex conjugation. For $\mathbf{u}_1\in {\rm dom}(\mathbf{L})\subset Y$ and $\mathbf{u}_2\in Y$, 
\begin{align*}
\langle \mathbf{L}(\lambda)\mathbf{u}_1, \mathbf{u}_2\rangle
=&\left\langle\begin{pmatrix}\lambda\phi_1+\mu_0\upsilon_1\\
 -\mu_0^{-1}({\phi_1}_{yy}+\lambda^2\phi_1+\mu_0\lambda \upsilon_1) \\
\lambda\eta_1-{\phi_1}_y(1)\end{pmatrix},
\begin{pmatrix} \phi_2 \\ \upsilon_2 \\ \eta_2 \end{pmatrix}\right\rangle \\
=&\int^1_0(\phi_1(\lambda^*\phi_2)^*+{\phi_1}_y(\lambda^*{\phi_2}_y)^*
+\mu_0 \upsilon_1\phi_2^*-\mu_0 \upsilon_1{\phi_2}_{yy}^*)~dy\\
&+\int^1_0 (\mu_0^{-1}{\phi_1}_y{\upsilon_2}_y^*
-\phi_1(\mu_0^{-1}{\lambda^*}^2\upsilon_2)^*+\upsilon_1(\lambda^*\upsilon_2)^*)~dy\\
&+\mu_0\eta_1{\phi_2}_y^*(1)-\mu_0 \upsilon_1(0){\phi_2}_y^*(0)
+\lambda\eta_1\eta_2^*-{\phi_1}_y(1)(\mu_0^{-1}\upsilon_2(1)+\eta_2)^*\\
=&\int^1_0(\phi_1(\lambda^*\phi_2)^*+{\phi_1}_y(\lambda^*{\phi_2}_y)^*
-\phi_1(\mu_0^{-1}{\lambda^*}^2\upsilon_2)^*+\mu_0^{-1}{\phi_1}_y{\upsilon_2}_y^*)~dy\\
&+\int^1_0(\upsilon_1(\mu_0\phi_2-\mu_0{\phi_2}_{yy})^*+\upsilon_1(\lambda^*\upsilon_2)^*)~dy
+\lambda\eta_1\eta_2^*+\mu_0\eta_1{\phi_2}_y^*(1)\\
=&\left\langle \begin{pmatrix} \phi_1 \\ \upsilon_1 \\ \eta_1 \end{pmatrix}, 
\begin{pmatrix}\lambda^*\phi_2+\mu_0^{-1}\upsilon_2\\ \mu_0\phi_2-\mu_0{\phi_2}_{yy}-\lambda^*\upsilon_2\\\mu_0{\phi_2}_y(1)+\lambda^*\eta_2\end{pmatrix}\right\rangle
-\int^1_0\mu_0^{-1}\phi_1(1+{\lambda^*}^2)\upsilon_2^*~dy \\
=&\left\langle \begin{pmatrix} \phi_1 \\ \upsilon_1 \\ \eta_1 \end{pmatrix}, 
\begin{pmatrix}\lambda^*\phi_2+\mu_0^{-1}\upsilon_2\\ \mu_0\phi_2-\mu_0{\phi_2}_{yy}-\lambda^*\upsilon_2\\\mu_0{\phi_2}_y(1)+\lambda^*\eta_2\end{pmatrix}\right\rangle 
+\left\langle \begin{pmatrix} \phi_1 \\ \upsilon_1 \\ \eta_1 \end{pmatrix},
\begin{pmatrix}\phi_p\\ \upsilon_p\\ \eta_p \end{pmatrix}\right\rangle
=:\langle \mathbf{u}_1,\mathbf{L}(\lambda)^\dag\mathbf{u}_2\rangle,
\end{align*}
where $\mathbf{L}(\lambda)^\dag$ denotes the adjoint of $\mathbf{L}(\lambda)$. Here the first equality uses \eqref{def:L}, and the second equality uses \eqref{def:inner} and follows after integration by parts because if $\mathbf{u}_1\in {\rm dom}(\mathbf{L})$ then $\eta_1=\upsilon_1(1)$ and ${\phi_1}_y(0)=0$ (see \eqref{def:dom}). The third equality follows provided that 
\begin{equation}\label{eqn:bdry;adj}
\upsilon_2(1)+\mu_0\eta_2=0\quad\text{and}\quad {\phi_2}_y(0)=0,
\end{equation}
so that the inner product is continuous with respect to $\phi_1\in H^1(0,1)$ and $\upsilon_1\in L^2(0,1)$ (see \eqref{def:Y}),
and the fourth equality uses \eqref{def:inner}. The fifth equality follows provided that 
\begin{align*}
-\int^1_0\mu_0^{-1}\phi_1(1+{\lambda^*}^2)\upsilon_2^*~dy &=
\left\langle \begin{pmatrix} \phi_1 \\ \upsilon_1 \\ \eta_1 \end{pmatrix}, 
\begin{pmatrix}\phi_p\\ \upsilon_p\\ \eta_p \end{pmatrix}\right\rangle \\
&=\int^1_0 (\phi_1\phi_p^*+{\phi_1}_y{\phi_p}_y^*)~dy+\int^1_0 \upsilon_1\upsilon_p^*~dy+\eta_1\eta_p^*\\
&=\int^1_0\phi_1(\phi_p-{\phi_p}_{yy})^*~dy+\phi_1(1){\phi_p}_y^*(1)-\phi_1(0){\phi_p}_y^*(0),
\end{align*}
where the last equality assumes that $\upsilon_p=0$ and $\eta_p=0$, because the left side does not depend on $\upsilon_1$ or $\eta_1$, and it follows after integration by parts. This works provided that 
\begin{equation}\label{eqn:up}
\left\{\begin{aligned}
&{\phi_p}_{yy}-\phi_p=\mu_0^{-1}(1+{\lambda^*}^2)\upsilon_2\quad\text{for $0<y<1$},\\
&{\phi_p}_y(1)=0,\\
&{\phi_p}_y(0)=0.
\end{aligned}\right.
\end{equation}

To recapitulate, 
\[
\mathbf{L}(\lambda)^\dag:{\rm dom}(\mathbf{L}^\dag) \subset Y\to Y,
\] 
where 
\[
L(\lambda)^\dag\mathbf{u}
=\begin{pmatrix} \lambda^*\phi+\mu_0^{-1}\upsilon+\phi_p \\ \mu_0\phi-\mu_0\phi_{yy}-\lambda^* \upsilon \\ \mu_0\phi_y(1)+\lambda^*\eta \end{pmatrix},
\]
\begin{equation}\label{def:up}
\phi_p(y)=-\left(\int_0^1\cosh(1-y)\upsilon(y)~dy\right)\frac{(1+{\lambda^*}^2)\cosh(y)}{\mu_0\sinh(1)}+\frac{1+{\lambda^*}^2}{\mu_0}\int_0^y\sinh(y-y')\upsilon(y')~dy',
\ee 
and
\[
{\rm dom}(\mathbf{L}^\dag)=\{\mathbf{u}\in H^2(0,1)\times H^1(0,1)\times \mathbb{C}:
\upsilon(1)+\mu_0\eta=0, \phi_y(0)=0\}.
\]
Indeed we solve \eqref{eqn:up}, for instance, by the method of variation of parameters and evaluate the result at $\upsilon_2=\upsilon$, to obtain \eqref{def:up}. If $\mathbf{u}\in {\rm dom}(\mathbf{L}^\dag)$ then \eqref{eqn:bdry;adj} holds true. Clearly ${\rm dom}(\mathbf{L}^\dag)$ is dense in $Y$.

The spectrum of $\mathbf{L}(i\sigma)^\dag:{\rm dom}(\mathbf{L}^\dag)\subset Y\to Y$ consists of discrete eigenvalues with finite multiplicities, say, $\gamma \in\mathbb{C}$ and a straightforward calculation reveals that
\[
(\sigma-i\gamma)^2=\mu_0(i\gamma)\tanh(i\gamma).
\]
Therefore $\gamma=-ik$, where $ik$ is an eigenvalue of $\mathbf{L}(i\sigma):{\rm dom}(\mathbf{L})\subset Y\to Y$. Compare \eqref{eqn:sigma}. Also the corresponding eigenfunction is 
\[
\begin{pmatrix}
{\displaystyle -\frac{(1-\sigma^2)(\gamma+i\sigma)}{\mu_0^2(\gamma^2+1)}\upsilon(1)\frac{\cosh(y)}{\sinh(1)}+\frac{\gamma-i\sigma}{\mu_0(\gamma^2+1)}\upsilon(y)} \\ \upsilon(y) \\ -\frac{1}{\mu_0}\upsilon(1) \end{pmatrix},\]
where $\upsilon$ satisfies 
\begin{align*}
\upsilon(y)=\left(-\frac{(\gamma+i\sigma)^2}{\mu_0}\upsilon(1)+(\gamma^2+1)\int_0^1\cosh(1-y)\upsilon(y)~dy\right)&\frac{\cosh(y)}{\sinh(1)}\\-(\gamma^2+1)&\int_0^y\sinh(y-y')\upsilon(y')~dy'.
\end{align*}

When $\sigma=0$, we infer from Lemma~\ref{lem:eps=0} that $-ik_j(0)=i(-1)^{j+1}\kappa$, $j=1,2$, are simple eigenvalues of $\mathbf{L}(0)^\dag:{\rm dom}(\mathbf{L}^\dag) \subset Y\to Y$, and a straightforward calculation reveals that the corresponding eigenfunctions are 
\begin{equation}\label{def:psi12}
\boldsymbol{\psi}_j(0)=\begin{pmatrix}
{\displaystyle \frac{ik_j(0)p_j}{\mu_0^2(1-k_j(0)^2)}
\Big(\cosh(k_j(0))\frac{\cosh(y)}{\sinh(1)}-\mu_0\cosh(k_j(0) y)\Big)}\\
p_j\cosh(k_j(0) y)\\{\displaystyle -\frac{p_j}{\mu_0}\cosh(k_j(0))}\end{pmatrix},
\end{equation}
where
\begin{equation}\label{def:cj12}
p_j=-\frac{i\cosh(k_j(0))}{{\cosh(k_j(0))}^2\sinh(k_j(0))-\mu _0\sinh(k_j(0))},
\end{equation}
so that $\langle \boldsymbol{\phi}_{j}(0),\boldsymbol{\psi}_{j'}(0)\rangle=\delta_{jj'}$, $j,j'=1,2$, where $\boldsymbol{\phi}_j(0)$, $j=1,2$, are in \eqref{def:phi12}. 

Also, when $\sigma=0$, $-ik_j(0)=0$, $j=3,4$, is an eigenvalue of $\mathbf{L}(0)^\dag$ with algebraic multiplicity $2$ and geometric multiplicity $1$, by Lemma~\ref{lem:eps=0}, and the corresponding eigenfunctions are
\begin{equation}\label{def:psi34}
\boldsymbol{\psi}_3(0)=\begin{pmatrix}
0\\ {\displaystyle \frac{\mu_0}{\mu_0-1}}\\-{\displaystyle \frac{1}{\mu_0-1}}\end{pmatrix}
\quad\text{and}\quad
\boldsymbol{\psi}_4(0)=\begin{pmatrix}
{\displaystyle \frac{1}{(\mu_0-1)\mu_0}\Big(\mu_0-\frac{\cosh(y)}{\sinh(1)}\Big)}\\0\\0\end{pmatrix},
\end{equation}
whence $\langle \boldsymbol{\phi}_{j}(0),\boldsymbol{\psi}_{j'}(0)\rangle=\delta_{jj'}$, $j,j'=3,4$, where $\boldsymbol{\phi}_j(0)$, $j=3,4$, are in \eqref{def:phi34}. Notice that 
\[
\langle \boldsymbol{\phi}_{j}(0),\boldsymbol{\psi}_{j'}(0)\rangle=\delta_{jj'}, \quad j,j'=1,2,3,4.
\] 
Therefore
\be\label{def:Pi0} 
\boldsymbol{\Pi}(0)\mathbf{u}=\langle \mathbf{u},\boldsymbol{\psi}_1(0)\rangle\boldsymbol{\phi}_1(0)
+\langle \mathbf{u},\boldsymbol{\psi}_2(0)\rangle\boldsymbol{\phi}_2(0)
+\langle \mathbf{u},\boldsymbol{\psi}_3(0)\rangle\boldsymbol{\phi}_3(0)
+\langle \mathbf{u},\boldsymbol{\psi}_4(0)\rangle\boldsymbol{\phi}_4(0).
\ee 
Clearly $\boldsymbol{\Pi}(0)$ commutes with $\mathbf{L}(0)$. 

\begin{remark*}
The first entry of \eqref{def:psi12} appears to be not defined when $\kappa=1$. On the other hand, recall \eqref{def:mu0} and a straightforward calculation leads to that
\begin{align*}
\lim_{\kappa\to1}&\frac{\displaystyle \cosh(\kappa)\frac{\cosh(y)}{\sinh(1)}-\mu_0(\kappa)\cosh(\kappa y)}{1-\kappa}
=\lim_{\kappa\to1}\frac{\displaystyle \cosh(\kappa)\frac{\cosh(y)}{\sinh(1)}-\kappa\frac{\cosh(\kappa)}{\sinh(\kappa)}\cosh(\kappa y)}{1-\kappa}\\
=&\frac{\cosh(1)}{\sinh(1)^2}(\sinh(1)\cosh(y) - \cosh(1)\cosh(y) + y\sinh(1)\sinh(y))
\end{align*}
is well defined. Therefore we may define $\boldsymbol{\psi}_j(0)={\displaystyle \lim_{\kappa\to1}\boldsymbol{\psi}_j(0)}$, $j=1,2$, when $\kappa=1$, and verify that $\langle \boldsymbol{\phi}_{j}(0),\boldsymbol{\psi}_{j'}(0)\rangle=\delta_{jj'}$, $j,j'=1,2,3,4$.
\end{remark*}

When $\sigma>\sigma_c$, we infer from Lemma~\ref{lem:eps=0} that $-ik_j(\sigma)$, $j=2,4$, are simple eigenvalues of $\mathbf{L}(i\sigma)^\dag$, and a straightforward calculation reveals that the corresponding eigenfunctions are
\be \label{def:psi24}
\boldsymbol{\psi}_j(\sigma)=\begin{pmatrix} p_{1,j}\cosh(y)+p_{2,j} \cosh(k_j(\sigma)y)\\
{\displaystyle -\frac{i\mu_0 p_{2,j}(k_j(\sigma)^2 - 1)}{k_j(\sigma) + \sigma}\cosh(k_{j}(\sigma)y)}\\
{\displaystyle \frac{i p_{2,j}(k_j(\sigma)^2 - 1)}{k_j(\sigma) + \sigma}\cosh(k_{j}(\sigma))}\end{pmatrix},
\ee 
where 
\begin{equation}\label{def:cj24}
\begin{aligned}
&p_{1,j}=\frac{2\cosh(k_j(\sigma))(\sigma^2 - 1)(k_j(\sigma) - \sigma)^2}{\mu_0^2\sinh(1)(k_j(\sigma)^2 - 1)(k_j(\sigma)\sinh(2k_j(\sigma)) + \sigma\sinh(2k_j(\sigma)) + 2k_j(\sigma)\sigma - 2k_j(\sigma)^2)},\\
&p_{2,j}=\frac{2k_j(\sigma)^2 - 2\sigma^2}{\mu_0(k_j(\sigma)^2 - 1)(k_j(\sigma)\sinh(2k_j(\sigma)) + \sigma\sinh(2k_j(\sigma)) + 2k_j(\sigma)\sigma - 2k_j(\sigma)^2)},
\end{aligned}
\end{equation}
so that  $\langle \boldsymbol{\phi}_{j}(\sigma),\boldsymbol{\psi}_{j'}(\sigma)\rangle=\delta_{jj'}$, $j,j'=2,4$, where $\boldsymbol{\phi}_j(\sigma)$, $j=2,4$, are in \eqref{def:phi1-4}. Therefore
\begin{equation}\label{def:Pi;high}
\boldsymbol{\Pi}(\sigma)\mathbf{u}=\langle \mathbf{u},\boldsymbol{\psi}_2(\sigma)\rangle\boldsymbol{\phi}_2(\sigma)
+\langle \mathbf{u},\boldsymbol{\psi}_4(\sigma)\rangle\boldsymbol{\phi}_4(\sigma).
\end{equation}

When $0<\sigma\leq\sigma_c$, we proceed likewise to define $\boldsymbol{\Pi}(\sigma)$. We do not include the formulae here.

\section{The Benjamin--Feir instability}\label{sec:BF}

Recall the notation of the previous section. When $\sigma,\delta=0$ and $\eps=0$, \eqref{eqn:LB;u1} becomes ${\mathbf{v}}_x=\mathbf{L}(0)\mathbf{v}$, and Lemma~\ref{lem:eps=0} says that 
\begin{align}
&\mathbf{L}(0)\boldsymbol{\phi}_j(0)=ik_j(0)\boldsymbol{\phi}_j(0),
\quad k_j(0)=(-1)^j\kappa,\quad j=1,2,\label{eqn:L12}
\intertext{and}
&\mathbf{L}(0)\boldsymbol{\phi}_3(0)=\boldsymbol{\phi}_4(0),\qquad
\mathbf{L}(0)\boldsymbol{\phi}_4(0)=0, \label{eqn:L34}
\end{align}
where $\boldsymbol{\phi}_j(0)$, $j=1,2,3,4$, are in \eqref{def:phi12} and \eqref{def:phi34}. 
Therefore 
the monodromy matrix (see Definition~\ref{def:Evans}) becomes
\[
\mathbf{X}(T;0,0,0)
=\begin{pmatrix}e^{-i\kappa T}&0&0&0 \\ 0&e^{i\kappa T}&0&0 \\ 0&0&1&0 \\ 0&0&T&1\end{pmatrix}
=\begin{pmatrix}1&0&0&0\\0&1&0&0\\0&0&1&0\\0&0&T&1\end{pmatrix},
\]
where $T=2\pi/\kappa$ is the period of the wave. Correspondingly the periodic Evans function (see \eqref{def:Delta}) becomes 
\[
\Delta(0,K\kappa;0)=0\quad\text{for all $K\in\mathbb{Z}$}.
\] 
We shall examine the zeros of $\Delta(\lambda,k;\epsilon)$ for $(\lambda,k,\eps)$ in the vicinity of $(0,K\kappa,0)$, $K\in\mathbb{Z}$, whereby reproducing the celebrate Benjamin--Feir instability for a Stokes wave of sufficiently small amplitude \cite{BM;BF} (see also \cite{Benjamin;BF,Whitham;BF}). Since $\Delta(\lambda,k;\eps)$ depends analytically on $\lambda, k$ and $\eps$, let
\[
\text{$\lambda\in\mathbb{C}$ and $|\lambda|\ll1$},\quad 
\text{$k=K\kappa+\gamma$, $K\in\mathbb{Z}$, $\gamma\in\mathbb{R}$ and $|\gamma|\ll1$},\quad
\text{$\eps\in\mathbb{R}$ and $|\eps|\ll1$},
\] 
and let
\begin{equation}\label{eqn:evans exp}
\Delta(\lambda,K\kappa+\gamma;\eps)=\sum_{\ell,m,n=0}^\infty 
d^{(\ell,m,n)}\lambda^\ell\gamma^m\eps^n
\end{equation}
for $|\lambda|,|\gamma|,|\eps|\ll 1$. Our effort goes into determining $d^{(\ell,m,n)}$, $\ell,m,n=0,1,2,\dots$.

\subsection{Expansion of the monodromy matrix}\label{sec:a0}

Throughout the subsection, $\sigma=0$. Let 
\[
\begin{pmatrix}\mathbf{v}_1&\mathbf{v}_2&\mathbf{v}_3&\mathbf{v}_4\end{pmatrix}(x;\delta,\eps)
=\begin{pmatrix}\boldsymbol{\phi}_1(0)&\boldsymbol{\phi}_2(0)&\boldsymbol{\phi}_3(0)&\boldsymbol{\phi}_4(0)\end{pmatrix}\mathbf{X}(x;0,\delta,\eps)
\] 
denote a fundamental matrix of \eqref{eqn:A}, where $\boldsymbol{\phi}_j(0)$, $j=1,2,3,4$, are in \eqref{def:phi12} and \eqref{def:phi34}, and $\mathbf{X}(x;0,\delta,\eps)$ in Definition~\ref{def:Evans}. We write
\be \label{def:a;exp0}
\mathbf{v}_k(x;\delta,\eps)=\sum_{j=1}^4\Big(\sum_{m+n=0}^{\infty}a_{jk}^{(m,n)}(x)\delta^m\eps^n\Big)\boldsymbol{\phi}_j(0)
\ee 
for $|\delta|,|\eps|\ll 1$, where $a_{jk}^{(m,n)}(x)$, $j,k=1,2,3,4$ and $m,n=0,1,2,\dots$ are to be determined. We pause to remark that $\mathbf{X}(x;0,\delta,\eps)$ depends analytically on $\delta$ and $\eps$ for any $x\in[0,T]$, for $\delta\in\mathbb{C}$ and $|\delta|\ll1$ for $\eps\in\mathbb{R}$ and $|\eps|\ll1$, whence the series \eqref{def:a;exp0} converges for any $x\in[0,T]$. 
Let
\[
\mathbf{a}^{(m,n)}(x)=(a_{jk}^{(m,n)}(x))_{j,k=1,\dots,4},
\] 
and we may assume that
\be \label{cond:a}
\mathbf{a}^{(0,0)}(0)=\mathbf{I}\quad\text{and}\quad 
\mathbf{a}^{(m,n)}(0)=\mathbf{0}\quad \text{for $m+n\geq 1$}.
\ee
Our task is to evaluate $\mathbf{a}^{(m,n)}(T)$, $m,n=0,1,2,\dots$. We write \eqref{def:u2} as
\be\label{def:b;exp0}
\mathbf{w}(x,\mathbf{v}_k(x);0,\delta,\eps)=\sum_{m+n=1}^\infty \mathbf{w}_{k}^{(m,n)}(x;0)\delta^m\eps^n
\ee
for $|\delta|,|\eps|\ll 1$, where $\mathbf{w}_{k}^{(0,0)}(x;0)=\mathbf{0}$, $k=1,2,3,4$, and $\mathbf{w}_{k}^{(m,n)}(x;0)$, $k=1,2,3,4$ and $m+n\geq1$, are to be determined. Recall \eqref{eqn:LB;delta} and we write 
\be\label{def:B;exp0}
\mathbf{B}(x;0,\delta,\eps)=\sum_{m+n=1}^\infty\mathbf{B}^{(m,n)}(x;0)\delta^m\eps^n
\ee
for $|\delta|,|\eps|\ll 1$, where $\mathbf{B}^{(0,0)}(x;0)=\mathbf{0}$, and $\mathbf{B}^{(m,n)}(x;\sigma)$, $1\leq m+n\leq2$, are in Appendix~\ref{A:Bexp}. Notice that $\mathbf{B}^{(m,0)}(x;0)$, $m\geq1$, do not involve $x$. 

Inserting \eqref{def:a;exp0}, \eqref{def:b;exp0} and \eqref{def:B;exp0} into the former equation of \eqref{eqn:LB;u12} we recall Lemma~\ref{lem:eps=0} or, equivalently, \eqref{eqn:L12} and \eqref{eqn:L34}, and make a straightforward calculation to obtain
\be \label{eqn:a00}
\mathbf{a}^{(0,0)}(x)=\begin{pmatrix} 
e^{-i\kappa x}&0&0&0 \\ 0&e^{i\kappa x}&0&0\\0&0&1&0\\0&0&x&1\end{pmatrix},
\ee 
and for $m+n\geq1$ we arrive at
\be \label{eqn:a}
\sum_{j=1}^4\Big(\frac{d}{dx}a_{jk}^{(m,n)}\Big)\boldsymbol{\phi}_j(0)=
-i\kappa\,a_{1k}^{(m,n)}\boldsymbol{\phi}_1(0)
+i\kappa\,a_{2k}^{(m,n)}\boldsymbol{\phi}_2(0)
+a_{3k}^{(m,n)}\boldsymbol{\phi}_4(0)+\boldsymbol{\Pi}(0)\mathbf{f}_{k}^{(m,n)}(x;0),
\ee  
where $\boldsymbol{\Pi}(0)$ is in \eqref{def:Pi0} and
\be \label{def:f0}
\mathbf{f}_{k}^{(m,n)}(x;0)=\sum_{\substack{0\leq m'\leq m\\ 0\leq n' \leq n}}
\mathbf{B}^{(m',n')}(x;0)\big(\mathbf{w}_{k}^{(m-m',n-n')}(x;0)
+\sum_{j=1}^4a_{jk}^{(m-m',n-n')}\boldsymbol{\phi}_j(0)\big).
\ee 

Inserting \eqref{def:a;exp0}, \eqref{def:b;exp0} and \eqref{def:B;exp0} into the latter equation of \eqref{eqn:LB;u12}, at the order of $\delta^m\eps^n$, $m+n\geq 1$, let $\mathbf{w}_{k}^{(m,n)}(x;0)=\begin{pmatrix}\phi\\ \upsilon \\ \eta\end{pmatrix}$, by abuse of notation, and we arrive at
\ba \label{eqn:red0}
&\phi_{xx}+\phi_{yy}={((\mathbf{1}-\boldsymbol{\Pi}(0))\mathbf{f}_{k}^{(m,n)}(x;0))_1}_x+\mu_0((\mathbf{1}-\boldsymbol{\Pi}(0))\mathbf{f}_{k}^{(m,n)}(x;0))_2&&\text{for $0<y<1$},\\
&\upsilon=\mu_0^{-1}\phi_x-\mu_0^{-1}((\mathbf{1}-\boldsymbol{\Pi}(0))\mathbf{f}_{k}^{(m,n)}(x;0))_1&&\text{for $0<y<1$},\\
&\eta_x=-\phi_y+((\mathbf{1}-\boldsymbol{\Pi}(0))\mathbf{f}_{k}^{(m,n)}(x;0))_3&&\text{at $y=1$},\\
&\eta=\upsilon&&\text{at $y=1$}, \\
&\phi_y=0&&\text{at $y=0$}.
\ea
Notice that since $\mathbf{B}^{(0,0)}(x;0)=\mathbf{0}$, the right side of \eqref{def:f0} does not involve $\mathbf{w}_{k}^{(m,n)}(x;0)$, and it is made up of lower order terms. Also notice that the fourth and fifth equations of \eqref{eqn:red0} ensure that $\mathbf{w}_{k}^{(m,n)}(x;0)\in {\rm dom}(\mathbf{L})$ (see \eqref{def:dom}). Recall \eqref{def:phi12}, \eqref{def:phi34} and \eqref{def:psi12}, \eqref{def:cj12}, \eqref{def:psi34}. We use the result of Appendix~\ref{A:Bexp} and solve the first and the last equations of \eqref{eqn:red0}, for instance, by the method of undetermined coefficients, subject to that $\mathbf{w}_{k}^{(m,n)}(x;0)\in(\mathbf{1}-\boldsymbol{\Pi}(0))Y$, so that $\boldsymbol{\Pi}(0)\mathbf{w}_{k}^{(m,n)}(x;0)=\mathbf{0}$, to determine $\phi$, and we determine $u$ and $\eta$ by the second and fourth equations of \eqref{eqn:red0}. The result is in Appendix~\ref{A:reduction0}. 


\begin{lemma}\label{lem:a4=0}
We have $a_{j4}^{(0,n)}(x)=0$ for $j=1,2,3,4$ for all $n\geq1$.
\end{lemma}

\begin{proof}
Recall \eqref{eqn:spec} and \eqref{def:tildeu}, and $\mathbf{B}(x;\sigma,\delta,\eps)\boldsymbol{\phi}_4(0)=\mathbf{0}$ for $\sigma=\delta=0$ (see \eqref{eqn:LB;delta}), whence $\mathbf{B}^{(0,n)}(x;0)\boldsymbol{\phi}_4(0)=\boldsymbol{0}$ for all $n\geq1$, by \eqref{def:B;exp0}. When $n=1$, \eqref{def:f0} leads to
\[
\mathbf{f}_{4}^{(0,1)}(x;0)=\mathbf{B}^{(0,1)}(x;0)\boldsymbol{\phi}_4(0)=\mathbf{0},
\]
by \eqref{def:b;exp0} and \eqref{eqn:a00}, and we solve \eqref{eqn:a} and \eqref{cond:a} to obtain $a_{j4}^{(0,1)}(x)=0$, $1\leq j\leq4$. We in turn solve \eqref{eqn:red0} to obtain $\mathbf{w}_{4}^{(0,1)}(x;0)=\mathbf{0}$. The assertion then follows by the induction on $n$ because $\mathbf{f}_{4}^{(0,n)}(x;0)=\mathbf{0}$ for all $n\geq1$. 
\end{proof}

\begin{corollary}\label{translation-invariance}
For $\eps\in\mathbb{R}$ and $|\eps|\ll1$, $\Delta(0,K\kappa;\eps)=0$, $K\in\mathbb{Z}$.
\end{corollary}

\begin{proof}
Notice that $\Delta(0,K\kappa;\eps)=\det(\mathbf{I}-\mathbf{X}(T;0,0,\eps))$, and Lemma~\ref{lem:a4=0} and \eqref{cond:a} assert that the fourth column of $\mathbf{I}-\mathbf{X}(T;0,0,\eps)$ vanishes.
\end{proof}

For $m+n\geq1$ we rearrange \eqref{eqn:a} as
\be\label{firstorderii}
\frac{d}{dx}\begin{pmatrix}a_{1k}^{(m,n)}\\a_{2k}^{(m,n)}\\a_{3k}^{(m,n)}\\a_{4k}^{(m,n)}\end{pmatrix}
=\mathbf{A}(0)\begin{pmatrix}a_{1k}^{(m,n)}\\a_{2k}^{(m,n)}\\a_{3k}^{(m,n)}\\a_{4k}^{(m,n)}\end{pmatrix}+\mathbf{F}_k^{(m,n)}(x;0),\quad \mathbf{F}_k^{(m,n)}(x;0):=\begin{pmatrix}\langle\mathbf{f}_k^{(m,n)}(x;0),\boldsymbol{\psi}_1(0)\rangle\\\langle\mathbf{f}_k^{(m,n)}(x;0),\boldsymbol{\psi}_2(0)\rangle\\\langle\mathbf{f}_k^{(m,n)}(x;0),\boldsymbol{\psi}_3(0)\rangle\\\langle\mathbf{f}_k^{(m,n)}(x;0),\boldsymbol{\psi}_4(0)\rangle\end{pmatrix},
\ee  
where $\mathbf{A}(0)$ collects the coefficients of $a_{jk}^{(m,n)}$ on the right side, $\boldsymbol{\psi}_j(0)$, $j=1,2,3,4$, are in \eqref{def:psi12}, \eqref{def:cj12} and \eqref{def:psi34}, and $\langle ,\rangle$ in \eqref{def:inner}. Recall from Section~\ref{sec:proj} that $\langle \boldsymbol{\phi}_j(0),\boldsymbol{\psi}_{j'}(0)\rangle=\delta_{jj'}$, $j=1,2,3,4$. We write the solution of \eqref{firstorderii} and \eqref{cond:a} symbolically as
\[
\mathbf{a}_{k}^{(m,n)}(x)=e^{\mathbf{A}(0)x}\int_0^x e^{-\mathbf{A}(0)x'}\mathbf{F}_k^{(m,n)}(x';0)~dx'.
\] 
Recalling the result of Appendix~\ref{A:Bexp} and Appendix~\ref{A:reduction0} we make a straightforward calculation to show that: 
\begin{equation}\label{eqn:a10}
\mathbf{a}^{(1,0)}(T)=\begin{pmatrix} \dfrac{4\pi c^3}{\mu_0s(s^2 - \mu_0 + 1)}
 & 0 & 0 & 0 \\ 0 & \dfrac{4\pi c^3}{\mu_0s(s^2 - \mu_0 + 1)} & 0 & 0 \\ 0 & 0 & \dfrac{2\pi c(\mu_0 + 1)}{\mu_0s(1 - \mu_0)} & 0\\ \dfrac{4\pi(s^2 + 1)}{\mu_0(\mu_0 - 1)s}&\dfrac{4\pi(s^2 + 1)}{\mu_0(\mu_0 - 1)s}&\dfrac{4\pi^2(s^2 + 1)}{\mu_0^2s^2(1-\mu_0)}&\dfrac{2\pi}{\kappa} \end{pmatrix},
\end{equation}
\begin{equation}\label{eqn:a01}
\mathbf{a}^{(0,1)}(T)=\begin{pmatrix} 0 & 0 & \dfrac{\pi(2s^2 + 3)}{s^2 - \mu_0 + 1}&0\\ 0 & 0 & \dfrac{\pi(2s^2 + 3)}{s^2 - \mu_0 + 1}&0\\0&0&0&0\\\dfrac{4\pi e^{2\kappa}(\mu_0 - 4c^2)(c^2 - 1)i}{(e^{4\kappa} - 1)(\mu_0 - 1)}&-\dfrac{4\pi e^{2\kappa}(\mu_0 - 4c^2)(c^2 - 1)i}{(e^{4\kappa} - 1)(\mu_0 - 1)}&\dfrac{2\pi c(\mu_0^2 + \mu_0 + 1)}{\mu_0(\mu_0 - 1)}&0\end{pmatrix},
\end{equation}
and
\begin{equation}\label{eqn:a20,11}
\mathbf{a}^{(2,0)}(T)=\begin{pmatrix} a_{11}^{(2,0)}& 0 & * &0  \\ 0 & (a_{11}^{(2,0)})^*& * &0 \\ 
a_{31}^{(2,0)}& a_{31}^{(2,0)} & * & a_{34}^{(2,0)} \\ * & * & * & a_{44}^{(2,0)} \end{pmatrix},
\qquad
\mathbf{a}^{(1,1)}(T)=\begin{pmatrix} a_{11}^{(1,1)} & a_{11}^{(1,1)} & * & a_{14}^{(1,1)}\\ 
a_{11}^{(1,1)} & a_{11}^{(1,1)} & * & a_{14}^{(1,1)}\\
a_{31}^{(1,1)} & * & * & 0 \\ * & * & * & * \end{pmatrix},
\end{equation}
where
\begin{align}
&\begin{aligned}\label{def:a20}
a_{11}^{(2,0)}
=&\frac{8\pi^2c^6}{\mu_0^2(\mu_0 - c^2)^2(c^2 - 1)}+\frac{2\pi c^4(c^4 + 4\mu_0^2c^2 - 3\mu_0^2 - 2\mu_0 c^2)i}{\mu_0^2(\mu_0 - c^2)^3(c^2 - 1)},\\
a_{31}^{(2,0)}=&\frac{4\pi(s^2 + 1)(s^2 - 3\mu_0s^2 - 2\mu_0 + \mu_0^2 + 1)}{\mu_0s(\mu_0 - 1)^2(s^2 - \mu_0 + 1)}
,\\
a_{34}^{(2,0)}=&\frac{2\pi c(c + s)}{\mu_0 + cs + c^2 - \mu_0c^2 - \mu_0cs - 1},\qquad 
a_{44}^{(2,0)}=-\frac{2\pi^2(s^2 + 1)}{\mu_0^2s^2(\mu_0 - 1)}
\end{aligned}
\intertext{and}
&\begin{aligned}\label{def:a11}
a_{11}^{(1,1)}=&-\frac{2\pi(c + 2c^3)}{(\mu_0 - c^2)(\mu_0 - 1)},\qquad
a_{14}^{(1,1)}=\frac{2\pi(s^2 + 1)}{s^2 - \mu_0 + 1},\\
a_{31}^{(1,1)}=&-\frac{2\pi i(4c^4 + 4sc^3 - 5c^2 - 3sc + 1)(- 2c^4\mu_0^2 - 4c^4\mu_0 + 2c^4 + 3c^2\mu_0^2 + 2c^2\mu_0 - \mu_0^3)}{c(\mu_0 - 1)^2(- c^2 + \mu_0)(- 4c^3 - 4sc^2 + 3c + s)}.
\end{aligned}
\end{align}
Here and elsewhere we employ the notation
\[
s=\sinh(\kappa)\quad\text{and}\quad c=\cosh(\kappa)
\] 
wherever it is convenient to do so. 
Our proof does not involve other entries of $\mathbf{a}^{(2,0)}(T)$ and $\mathbf{a}^{(1,1)}(T)$, whence we do not include the formulae here. We emphasize that $a_{22}^{(2,0)}$ in \eqref{eqn:a20,11} is the complex conjugate of $a_{11}^{(2,0)}$. Additionally we calculate that
\begin{equation}\label{def:a02}
\begin{aligned}
a_{11}^{(0,2)}(T)=-i\mu_0 \pi (&24 s^2 - 21 \mu_0 s^2 - 20 \mu_0 s^4 - 8 \mu_0 s^6 - 9 \mu_0 + 40 s^4 \\&+ 16 s^6 + 15 \mu_0^2 s^2 + 16 \mu_0^2 s^4 + 8 \mu_0^2 s^6 + 9 \mu_0^2)(4 (s^2 + 1) (\mu_0 - 1) (s^2 - \mu_0 + 1))^{-1}
\end{aligned}
\end{equation}
and 
\begin{equation}\label{def:a02;0}
a_{j4}^{(0,2)}(T)=0, \quad j=1,2,3,4.
\end{equation} 
We remark that $a_{jk}^{(m,n)}(T)$, $1\leq m+n\leq 2$, suffices for a proof of the Benjamin--Feir instability.

\subsection{The modulational instability index}\label{sec:ind1}

Let
\[
\text{$\lambda\in\mathbb{C}$ and $|\lambda|\ll1$},\quad 
\text{$k=K\kappa+\gamma$, $K\in\mathbb{Z}$, $\gamma\in\mathbb{R}$ and $|\gamma|\ll1$},\quad
\text{$\eps\in\mathbb{R}$ and $|\eps|\ll1$},
\]
and we turn the attention to \eqref{eqn:evans exp}. Putting together  \eqref{eqn:a00}, \eqref{eqn:a10},  \eqref{eqn:a01},  \eqref{eqn:a20,11}, \eqref{def:a20}, \eqref{def:a11}, \eqref{def:a02}, \eqref{def:a02;0}, we arrive at 
\ba \label{eqn:evans0}
\Delta(\lambda,K\kappa+\gamma;\eps)=&d^{(4,0,0)}\lambda^4+d^{(3,1,0)}\lambda^3 \gamma+d^{(2,2,0)}\lambda^2\gamma^2+d^{(1,3,0)}\lambda \gamma^3+d^{(0,4,0)}\gamma^4\\
&+d^{(3,0,2)}\lambda^3\eps^2+d^{(2,0,3)}\lambda^2\eps^3+d^{(0,3,2)}\gamma^3\eps^2+d^{(0,2,3)}\gamma^2\eps^3\\
&+d^{(2,1,2)}\lambda^2\gamma\eps^2+d^{(1,2,2)}\lambda\gamma^2\eps^2+d^{(1,1,3)}\lambda\gamma \eps^3
+d^{(1,0,5)}\lambda\eps^5+d^{(0,1,5)}\gamma\eps^5\\
&\begin{aligned}
+o(&(|\lambda|+|\gamma|)^4+|\lambda|^3|\eps|^2+|\lambda|^2|\eps|^3+|\gamma|^3|\eps|^2+|\gamma|^2|\eps|^3\\
&+|\lambda\gamma| |\eps|^2(|\lambda|+|\gamma|+|\eps|)+|\lambda| |\eps|^5+|\gamma| |\eps|^5)\end{aligned}
\ea
as $\lambda,\gamma,\eps\to 0$, where $d^{(\ell,m,n)}$, $\ell,m,n=0,1,2,\dots$ can be determined in terms of $a_{jk}^{(m,n)}(T)$, $j,k=1,2,3,4$ and $m,n=0,1,2,\dots$, and $T=2\pi/\kappa$ is the period of the wave. We may regard $a_{jk}^{(m,n)}(T)$ and, hence, $d^{(\ell,m,n)}$ as functions of $T$ or, equivalently, $\kappa$. Below we suppress $T$ for the simplicity of notation. We remark that $d^{(\ell,m,n)}=0$ when $\lambda^\ell\gamma^m\eps^n$ is of lower order than
\[
(\lambda+\gamma)^4, \lambda^3\eps^2, \lambda^2\eps^3, \gamma^3\eps^2, \gamma^2\eps^3, \lambda\gamma\eps^2(\lambda+\gamma+\eps), \lambda \eps^5, \gamma \eps^5.
\]

The Weierstrass preparation theorem asserts that \eqref{eqn:evans0} becomes 
\[
\Delta(\lambda,K\kappa+\gamma;\eps)=W(\lambda,\gamma,\eps)h(\lambda,\gamma,\eps)
\]
for $|\lambda|,|\gamma|,|\eps|\ll 1$, where 
\begin{equation}\label{def:W(0,0,0)}
W(\lambda,\gamma,\eps)=\lambda^4+g_3(\gamma,\eps)\lambda^3+g_2(\gamma,\eps)\lambda^2+g_1(\gamma,\eps)\lambda+g_0(\gamma,\eps)
\end{equation}
is a Weierstrass polynomial, $g_j(\gamma,\eps)$, $j=0,1,2,3$, are analytic and $g_j(0,0)=0$, and $h(\lambda,\gamma,\eps)$ is analytic at $(0,0,0)$ and $h(0,0,0)=d^{(4,0,0)}\neq0$. Therefore the zeros of $\Delta(\lambda,K\kappa+\gamma;\eps)$ for $(\lambda,\gamma,\eps)$ in the vicinity of $(0,0,0)$ are the four zeros of $W(\lambda,\gamma,\eps)$. One may examine the leading terms of the zeros of \eqref{def:W(0,0,0)} as $\gamma,\eps\to 0$, to determine the asymptotics of $\lambda(\gamma,\eps)$ for $\lambda$ in the vicinity of the origin of the complex plane. 
Recently, Berti, Maspero, and Ventura \cite{BMV;finite} (see also \cite{BMV} in the infinite depth) used Kato's similarity transformation and a KAM theory approach, to analytically confirm a ``figure-8'' loop, as numerically predicted in \cite{DO} and others. In Section~\ref{sec:high-freq}, we successfully solve a quadratic Weierstrass polynomial (see \eqref{weierstrass_high}), to analytically confirm ``bubbles'' of unstable spectrum away from $0\in\mathbb{C}$. See \eqref{delta_high}.

Our goal here is to determine spectral stability and instability near the origin of the complex plane, rather than the asymptotics of the spectrum. Let 
\begin{equation}\label{def:lambda(gamma,eps)}
\lambda_j(k_j(0)+\gamma,\eps)=
\alpha^{(1,0)}_j\gamma+\alpha^{(2,0)}_j\gamma^2+\alpha^{(1,1)}_j\gamma\eps
+o(|\gamma|^2+ |\gamma||\eps|),\quad j=1,2,3,4,
\end{equation}
as $\gamma,\eps\to 0$, where $\alpha^{(1,0)}_j$, $\alpha^{(2,0)}_j$, $\alpha^{(1,1)}_j$, $j=1,2,3,4$, are to be determined in terms of $d^{(\ell,m,n)}$ and, hence, $a_{jk}^{(m,n)}$, $j,k=1,2,3,4$ and $m,n=0,1,2,\dots$. We pause to remark about the form of \eqref{def:lambda(gamma,eps)}. When $\eps=0$ we deduce from Section~\ref{sec:eps=0} that $\Delta(i\sigma(k),k;0)=0$ for any $k\in\mathbb{R}$, where $\sigma$ is in \eqref{eqn:sigma}. Particularly $\Delta(\lambda_j(k_j(0),0),k_j(0);0)=0$, $j=1,2,3,4$, where 
\[
k_j(0)=(-1)^j\kappa\quad\text{for $j=1,2$},\quad k_j(0)=0\quad \text{for $j=3,4$}.
\]
In other words, $\lambda=0$ and $k=k_j(0)$, $j=1,2,3,4$, are the four zeros of $\Delta(\cdot,\cdot;0)$. Also $\sigma(k_j(0)+\gamma)$ 
admits power series expansions about $\gamma=0$, and they must agree with 
\[
\lambda_j(k_j(0)+\gamma,0)=\alpha^{(1,0)}_j\gamma+\alpha^{(2,0)}_j\gamma^2+\cdots
\]
as $\gamma\to0$. Moreover, \cite[Theorem~1.1]{BMV;finite} justifies \eqref{def:lambda(gamma,eps)} for $|\gamma|\ll|\eps|\ll 1$. 

Substituting \eqref{def:lambda(gamma,eps)} into \eqref{eqn:evans0}, 
after a straightforward calculation, we learn that $\gamma^4$ is the leading order whose coefficient reads
\begin{equation}\label{eqn:alpha10}
\begin{aligned}
&d^{(4,0,0)}(\alpha^{(1,0)}_j)^4+d^{(3,1,0)}(\alpha^{(1,0)}_j)^3+d^{(2,2,0)}(\alpha^{(1,0)}_j)^2+d^{(1,3,0)}\alpha^{(1,0)}_j+d^{(0,4,0)}\\
&=-\big(-a_{11}^{(1,0)}\alpha^{(1,0)}_j+iT\big)^2
\big((Ta_{34}^{(2,0)}-a_{33}^{(1,0)}a_{44}^{(1,0)})(\alpha^{(1,0)}_j)^2+iT(a_{33}^{1,0}+a_{44}^{(1,0)})\alpha^{(1,0)}_j+T^2\big),
\end{aligned}
\end{equation}
where $a_{11}^{(1,0)}$, $a_{33}^{(1,0)}$, $a_{44}^{(1,0)}$, $a_{34}^{(2,0)}$ are in \eqref{eqn:a10} and \eqref{def:a20}. At the order of $\gamma^4$, \eqref{eqn:alpha10} must vanish, whence
\ba\label{def:alpha10'}
\alpha^{(1,0)}_j=&\frac{iT}{a_{11}^{(1,0)}}\quad \text{or}\\
\alpha^{(1,0)}_j=&\frac{-iT(a_{33}^{(1,0)} + a_{44}^{(1,0)}) \pm T\sqrt{- {a_{33}^{(1,0)}}^2 + 2a_{33}^{(1,0)}a_{44}^{(1,0)} - ({a_{44}^{(1,0)}})^2 - 4Ta_{34}^{(2,0)}}}{2(Ta_{34}^{(2,0)} - a_{33}^{(1,0)}a_{44}^{(1,0)})}.
\ea
Notice that $\alpha^{(1,0)}_j$, $j=1,2,3,4$, are purely imaginary by \eqref{eqn:a10} and \eqref{def:a20}. On the other hand, \eqref{def:alpha10'} must agree with power series expansions of \eqref{def:sigma} about $k_j(0)$, $j=1,2,3,4$. Therefore
\begin{equation}\label{def:alpha10}
\alpha^{(1,0)}_j=\frac{iT}{a_{11}^{(1,0)}} \quad\text{for}\quad j=1,2,
\end{equation}
and the latter equation of \eqref{def:alpha10'} holds true for $j=3,4$.

Substituting \eqref{def:lambda(gamma,eps)} and \eqref{def:alpha10} into \eqref{eqn:evans0}, after a straightforward calculation, we verify that the $\gamma^5$ term vanishes, and solving at the order of $\gamma^6$ we arrive at 
\be\label{def:alpha20}
\alpha^{(2,0)}_j=\pm\frac{T^2(-(a_{11}^{(1,0)})^2+2a_{11}^{(2,0)})}{2(a_{11}^{(1,0)})^3},\quad j=1,2,
\ee
where $a_{11}^{(1,0)}$ and $a_{11}^{(2,0)}$ are in \eqref{eqn:a10} and \eqref{def:a20}. We remark that $\alpha^{(2,0)}_j$, $j=1,2$, are purely imaginary because $(a_{11}^{(1,0)})^2$ offsets the real part of $2a_{11}^{(2,0)}$. The $\pm$ signs explain the oppositeness of the convexity of the curves \eqref{def:sigma} at $k_j(0)$, $j=1,2$. See Figure~\ref{fig:dispersion}.

To proceed, substituting \eqref{def:lambda(gamma,eps)} and \eqref{def:alpha10}, \eqref{def:alpha20} into \eqref{eqn:evans0}, after a straightforward calculation, we verify that the $\gamma^3\eps^2$ term vanishes, and the coefficient of $\gamma^4\eps^2$ reads
\ba \label{eqn:f}
T^2f_1(\alpha_j^{(1,1)})^2-\frac{T^4(2a_{11}^{(2,0)}-(a_{11}^{(1,0)})^2)}{(a_{11}^{(1,0)})^4} f_2,
\ea 
where
\ba \label{f1f2}
f_1=&Ta_{34}^{(2,0)} + a_{11}^{(1,0)}a_{33}^{(1,0)} + a_{11}^{(1,0)}a_{44}^{(1,0)} - a_{33}^{(1,0)}a_{44}^{(1,0)} - ({a_{11}^{(1,0)}})^2,\\
f_2=&a_{11}^{(0,2)}(a_{11}^{(1,0)})^2 - Ta_{11}^{(0,2)}a_{34}^{(2,0)} + Ta_{14}^{(1,1)}a_{31}^{(1,1)} - a_{11}^{(0,2)}a_{11}^{(1,0)}a_{33}^{(1,0)} \\
&+ a_{11}^{(1,0)}a_{13}^{(0,1)}a_{31}^{(1,1)} -a_{11}^{(0,2)}a_{11}^{(1,0)}a_{44}^{(1,0)} 
+ a_{11}^{(1,0)}a_{14}^{(1,1)}a_{41}^{(0,1)} + a_{11}^{(0,2)}a_{33}^{(1,0)}a_{44}^{(1,0)}\\& - a_{13}^{(0,1)}a_{31}^{(1,1)}a_{44}^{(1,0)} + a_{13}^{(0,1)}a_{34}^{(2,0)}a_{41}^{(0,1)} - a_{14}^{(1,1)}a_{33}^{(1,0)}a_{41}^{(0,1)},
\ea 
and $a^{(m,n)}_{jk}$ is in the previous subsection. 

\begin{theorem}[Spectral instability near $0\in\mathbb{C}$]\label{thm:BF}
A $2\pi/\kappa$ periodic Stokes wave of sufficiently small amplitude in water of unit depth is spectrally unstable in the vicinity of $0\in\mathbb{C}$ provided that ${\rm ind}_1(\kappa)>0$, where
\ba \label{def:ind1}
{\rm ind}_1(\kappa)=&8\cosh(2\kappa) + 24\kappa\sinh(2\kappa) + 2\kappa\sinh(4\kappa) + 19\cosh(2\kappa)^2 - 8\cosh(2\kappa)^3\\
&- 10\cosh(2\kappa)^4 - 8\kappa^2\cosh(2\kappa)^2 - 28\kappa^2 + 8\kappa\cosh(2\kappa)^3\sinh(2\kappa) - 9.
\ea
\end{theorem}

\begin{proof}
At the order of $\gamma^4\eps^2$, \eqref{eqn:f} must vanish, whence
\begin{equation}\label{def:alpha11}
\alpha^{(1,1)}_j=\pm\sqrt{\frac{T^2(2a_{11}^{(2,0)}-(a_{11}^{(1,0)})^2)}{(a_{11}^{(1,0)})^4}\frac{f_2}{f_1} }, 
\end{equation}
where $f_1$ and $f_2$ are in \eqref{f1f2}. We recall the result of the previous subsection and make a straightforward calculation to arrive at
\begin{align*}
&2a_{11}^{(2,0)}-(a_{11}^{(1,0)})^2=-\frac{4i\pi {(s^2+1)}^2({(s^2-\mu _{0}+1)}^2+4{\mu_0}^2s^2)}{{\mu _{0}}^2s^2{(s^2-\mu _{0}+1)}^3}
\intertext{and}
&f_1=-\frac{4\pi ^2(s^2+1)(4\mu_0s^2(s^2 + 1)-(s^2-\mu _{0}+1)^2)}{{\mu _{0}}^2s^2(\mu _{0}-1){(s^2-\mu _{0}+1)}^2},\\
&f_2=-\frac{i\pi^3(s^2 + 1)^2}{4s^4(\mu_0 - 1)(s^2 - \mu_0 + 1)^3}{\rm ind}_1(\kappa),
\end{align*}
where ${\rm ind}_1(\kappa)$ is in \eqref{def:ind1}. Therefore $(\alpha^{(1,1)}_j)^2\in \mathbb{R}$ and $(\alpha^{(1,1)}_j)^2>0$ implies spectral instability. We shall show that  
$$
a_{11}^{(1,0)}>0,\quad s^2-\mu_0+1>0,\quad \text{and}\quad 
4\mu_0s^2(s^2 + 1)-(s^2-\mu _{0}+1)^2>0\quad\text{for all $\kappa>0$},
$$
so that $(\alpha^{(1,1)}_j)^2>0$ if and only if ${\rm ind}_1(\kappa)>0$. 

Since $\sinh(x)-x=(\cosh(x_0)-1)x$ for any $x>0$ for some $0<x_0<x$ by the mean value theorem, 
\[
s^2-\mu_0+1=\frac{\cosh(\kappa)(\sinh(2\kappa)-2\kappa)}{2\sinh(\kappa)}>0\quad\text{for all $\kappa>0$},
\]
and, hence, $a^{(1,0)}_{11}>0$ by \eqref{eqn:a10}. Also since
\[
4\kappa\sinh(4\kappa)-\cosh(4\kappa) - 8\kappa^2 + 1=16\kappa_0(\cosh(4\kappa_0) - 1)\kappa>0\quad\text{for any $\kappa>0$}\quad \text{for some $0<\kappa_0<\kappa$}
\]
by the mean value theorem, 
\[
4\mu_0s^2(s^2 + 1)-(s^2-\mu _{0}+1)^2=\frac{\cosh(\kappa)^2(4\kappa\sinh(4\kappa)-\cosh(4\kappa)  - 8\kappa^2 + 1)}{8(\cosh(\kappa)^2 - 1)}>0\quad\text{for all $\kappa>0$}.
\]
This completes the proof. 
\end{proof}

Numerical evaluation reveals that ${\rm ind}_1(\kappa)$ has exactly one zero at $\kappa=\kappa_1:=1.3627827\dots$ and ${\rm ind}_1(\kappa)>0$ if $\kappa >\kappa_1$ and ${\rm ind}_1(\kappa)<0$ if $0<\kappa<\kappa_1$. We shall verify the numerical findings by means of rigorous analysis and validated numerics, with mathematically strict error control including rounding error, whereby giving a computer-assisted proof.

\begin{corollary}[The Benjamin--Feir instability]\label{cor:BF}
A $2\pi/\kappa$ periodic Stokes wave of sufficiently small amplitude in water of unit depth is spectrally unstable for $\kappa>\kappa_1$, where $\kappa_1$ is the unique zero of \eqref{def:ind1}, and $1.3627827<\kappa_1<1.3627828$. 
\end{corollary}

\begin{proof}
Throughout we carry out interval arithmetic computations using the MATLAB based package INTLAB \cite{Ru99a}. We begin by focusing on $\kappa\ll1$, say, $\kappa\in(0,0.2]$. An explicit calculation reveals that
\[
\operatorname{ind}_1(0),\frac{d\operatorname{ind}_1}{d\kappa}(0),
\frac{d^2\operatorname{ind}_1}{d\kappa^2}(0),
\frac{d^3\operatorname{ind}_1}{d\kappa^3}(0)=0
\quad\text{and}\quad \frac{d^4\operatorname{ind}_1}{d\kappa^4}(0)=-3456, 
\]
whence it follows from the mean value theorem that 
\[
\operatorname{ind}_1(\kappa)=\kappa^{(1)}\kappa^{(2)}\kappa^{(3)}\frac{d^4\operatorname{ind}_1}{d\kappa^4}(\kappa^{(4)})\kappa
\quad\text{for some $\kappa^{(n)}$, $n=1,2,3,4$,}
\] 
such that $0<\kappa^{(4)}<\kappa^{(3)}<\kappa^{(2)}<\kappa^{(1)}<\kappa$. We make an INTLAB computation to learn that  
\[
\frac{d^4\operatorname{ind}_1}{d\kappa^4}(\text{midrad}(0,0.2))
=[-13248.1677692142,  -125.8373503340],
\]
where $\text{midrad}(0,0.2)$ denotes an interval rigorously enclosing $[-0.2,0.2]$, accounting for rounding error, and the right side rigorously encloses the range of $\frac{d^4\operatorname{ind}_1}{d\kappa^4}(\kappa)$ for $\kappa\in[-0.2,0.2]$. Clearly $\frac{d^4\operatorname{ind}_1}{d\kappa^4}(\kappa)<0$ for $\kappa\in(0,0.2]$ and, hence, $\operatorname{ind}_1(\kappa)<0$ for $\kappa\in(0,0.2]$.

We turn to $\kappa\gg1$, say, $\kappa\in(2,\infty)$. We rearrange \eqref{def:ind1} as 
\begin{align*}
{\rm ind}_1(\kappa)=\tfrac18e^{-8\kappa}(&4\kappa e^{16\kappa}+ 16\kappa e^{12\kappa}+ 18e^{12\kappa} + 96\kappa e^{10\kappa}+ 8e^{10\kappa}+ 8e^{6\kappa}+18e^{4\kappa}\\
&-5e^{16\kappa}-8e^{14\kappa}- 16\kappa^2e^{12\kappa}- 256\kappa^2e^{8\kappa}-26e^{8\kappa}
-96\kappa e^{6\kappa} \\
&-16\kappa^2e^{4\kappa}-16\kappa e^{4\kappa}- 8e^{2\kappa} - 4\kappa - 5)\\
=\tfrac18e^{-8\kappa}(&(4\kappa e^{4\kappa}-5e^{4\kappa}-8e^{2\kappa}-16\kappa^2) e^{12\kappa} \\
&+16\kappa ( e^{4\kappa}-16 \kappa) e^{8\kappa}+ (18e^{4\kappa}-26)e^{8\kappa}+ 96\kappa(e^{4\kappa}-1) e^{6\kappa} \\
&+ (8e^{6\kappa}-16\kappa^2-16\kappa)e^{4\kappa}+ 8(e^{4\kappa}-1)e^{2\kappa}+18e^{4\kappa} - 4\kappa - 5).
\end{align*}
A straightforward calculation reveals that 
\begin{align*}
&4\kappa e^{4\kappa}-5e^{4\kappa}-8e^{2\kappa}-16\kappa^2>8e^{4\kappa}-5e^{4\kappa}-24e^{2\kappa}=(3e^{2\kappa}-24)e^{2\kappa}>(3e^{4}-24)e^{2\kappa}>0,\\
&e^{4\kappa}-16 \kappa>e^{4\kappa}-16e^{\kappa}=(e^{3\kappa}-16)e^{\kappa}>(e^{6}-16)e^{\kappa}>0,\\
&18e^{4\kappa}-26>18e^{8}-26>0,\\
&8e^{6\kappa}-16\kappa^2-16\kappa>8e^{6\kappa}-16e^{2\kappa}-16e^{\kappa}>8e^{6\kappa}-32e^{2\kappa}=8(e^{4\kappa}-4)e^{2\kappa}>8(e^{8}-4)e^{2\kappa}>0,\\
&18e^{4\kappa} - 4\kappa - 5>18e^{4\kappa} - 4e^{\kappa} - 5>0
\end{align*}
for $\kappa\in(2,\infty)$ and, hence, $\operatorname{ind}_1(\kappa)>0$ for $\kappa\in(2,\infty)$.

It remains to treat $\kappa \in [0.2, 2]$. We divide the interval $[0.2,1.3627827]$ into finitely many subintervals $I_n$ and verify by means of validated numerics that $\sup(\operatorname{ind}_1(I_n))<0$ for each subinterval. We likewise divide the interval $[1.3627828,2]$ into finitely many subintervals $J_n$ and verify that $\inf(\operatorname{ind}_1(J_n))>0$ for each subinterval. Therefore it follows from the intermediate value theorem that $\operatorname{ind}_1$ has a zero in the interval $(1.3627827,1.3627828)$. We make an INTLAB computation to learn that
\[
\frac{d\operatorname{ind}_1}{d\kappa}(\operatorname{infsup}(1.3627827,1.3627828))=[31301.1666863430, 31301.6401032990],
\]
where $\operatorname{infsup}(1.3627827,1.3627828)$ is an interval rigorously enclosing $[1.3627827,1.3627828]$ and the right side likewise rigorously encloses the range of $\frac{d\operatorname{ind}_1}{d\kappa}(\kappa)$ for $\kappa\in [1.3627827,1.3627828]$. Therefore $\frac{d\operatorname{ind}_1}{d\kappa}(\kappa)>0$ for $\kappa \in[1.3627827,1.3627828]$ and, hence, $\operatorname{ind}_1$ has a unique zero in the interval $(1.3627827,1.3627828)$. This completes the proof. MATLAB scripts can be made available upon request.
\end{proof}

Since $\mu_0$ is a monotonically increasing function of $\kappa$ (see \eqref{def:mu0}), we may regard ${\rm ind}_1$ as a function of $\mu_0$. Numerical evaluation reveals that ${\rm ind}_1(\mu_0)=0$ when $\mu_0= 1.553848798953821\dots$. Also we may regard ${\rm ind}_1$ as a function of $F:=1/\sqrt{\mu_0}$, where $F$ is the Froude number in the linear limit (see \eqref{def:mu}). Numerical evaluation reveals that ${\rm ind}_1(F)=0$ when $F= 0.802223946850146\dots$. 

Moreover, comparing \eqref{def:ind1} with the index $\nu(F)$ of \cite[(6.17)]{BM;BF}, we verify that
\[
\nu(F)=-\frac{\mu_0{\rm ind_1}(\kappa)}{32 s^4(2s^2 - 6\mu_0s^2 - 4\mu_0s^4 - 2\mu_0 + s^4 + \mu_0^2 + 1)},
\]
where $s=\sinh(\kappa)$. Therefore, Corollary~\ref{cor:BF} agrees with \cite[Theorem~2]{BM;BF}. We remark that $\nu(F)$ is the same as those functions in \cite[pp.68]{Benjamin;BF} and \cite[(57),(58)]{Whitham;BF}, among others.


\section{The spectrum away from the origin}\label{sec:high-freq}

We turn the attention to the spectrum away from $0\in\mathbb{C}$. 

Recall from Section~\ref{sec:eps=0} that 
\[
{\rm spec}(\mathcal{L}(0))=\{i\sigma: \text{$\sigma\in\mathbb{R}$ and $(\sigma-k)^2=\mu_0k\tanh(k)$ for some $k\in\mathbb{R}$}\}.
\] 
Recall that $k_2(\sigma)\geq \kappa>0$ is a simple zero of $\sigma_-(k)=\sigma\geq0$ and $k_4(\sigma)\geq0$ the simple zero of $\sigma_+(k)=\sigma>0$, that is,  
\[
k_2(\sigma)-\sqrt{\mu_0 k_2(\sigma)\tanh(k_2(\sigma))}=k_4(\sigma)+\sqrt{\mu_0 k_4(\sigma)\tanh(k_4(\sigma))}=\sigma.
\]
Numerical evidence (see \cite{McLean;finite-depth, FK, DO}, among others) suggests spectral instability in the vicinity of $i\sigma\in\mathbb{C}$ provided that 
\begin{equation}\label{def:resonance}
k_2(\sigma)-k_4(\sigma)=N\kappa \quad\text{for some $N>0,\in\mathbb{Z}$},
\end{equation}
where $\kappa>0$ is the wave number of the Stokes wave. 
Also, in view of Hamiltonian systems, MacKay and Saffman \cite{MS} argued that \eqref{def:resonance} is a necessary condition of spectral instability. We shall make rigorous analysis to elucidate the numerical findings and, in the process, give a first proof of spectral instability of a Stokes wave of sufficiently small amplitude away from $0\in\mathbb{C}$. 

When $\sigma=0$, recall from Section~\ref{sec:eps=0} that $k_2(0)=\kappa$ and $k_4(0)=0$, whence \eqref{def:resonance} holds true for $N=1$. Also $k_1(0)=-\kappa$ and $k_3(0)=0$. Theorem~\ref{thm:BF} and Corollary~\ref{cor:BF} address spectral instability near $i0\in\mathbb{C}$.  

\begin{lemma}\label{lemma:k2minusk4;monotonity} 
It follows that $k_2(\sigma)-k_4(\sigma)$ is a strictly increasing function of $\sigma\geq0$ and $k_2(\sigma)-k_4(\sigma)\to\infty$ as $\sigma\to\infty$. 
\end{lemma}


\begin{proof}
Assume for now that $\frac{dk_2}{d\sigma}(\sigma), \frac{dk_4}{d\sigma}(\sigma)>0$ for all $\sigma\geq0$. We shall demonstrate that $\frac{d\sigma_-}{dk}(k_2(\sigma))<\frac{d\sigma_+}{dk}(k_4(\sigma))$ and, hence, $\frac{dk_2}{d\sigma}(\sigma)>\frac{dk_4}{d\sigma}(\sigma)$ for all $\sigma\geq0$. Indeed, since $\frac{d\sqrt{\mu_0k\tanh(k)}}{dk}>0$ for all $k\geq0$, 
\[
\frac{d\sigma_-}{dk}(k_2(\sigma))=1-\Big[\frac{d\sqrt{\mu_0k\tanh(k)}}{dk}\Big]_{k=k_2(\sigma)}<1<1+\Big[\frac{d\sqrt{\mu_0k\tanh(k)}}{dk}\Big]_{k=k_4(\sigma)}=\frac{d\sigma_+}{dk}(k_4(\sigma)).
\]
Notice that $\frac{dk_4}{d\sigma}(\sigma)=\frac{1}{\frac{d\sigma_+}{dk}(k_4(\sigma)) }>0$ for all $\sigma\geq0$. 

We calculate 
\begin{align*}
\frac{d^2\sigma_-}{dk^2}(k)=&\frac14\sqrt{\frac{\mu_0}{k^3\tanh(k)^3}}
(-3k^2{\tanh(k)}^4+2k^2{\tanh(k)}^2+k^2+2k{\tanh(k)}^3-2k\tanh(k)+{\tanh(k)}^2)\\
=:&\frac14\sqrt{\frac{\mu_0}{k^3\tanh(k)^3}}f(k)
\end{align*}
and we proceed as in the proof of Corollary~\ref{cor:BF} to show that $f(k)>0$ and, hence, $\frac{d^2\sigma_-}{dk^2}(k)>0$ for all $k>0$. An explicit calculation reveals that 
\[
f(0),\frac{df}{dk}(0),\frac{d^2f}{dk^2}(0),\frac{d^3f}{dk^3}(0)=0\quad \text{and}\quad \frac{d^4f}{dk^4}(0)=96.
\]
We make an INTLAB computation to learn that 
\[\frac{d^4f}{dk^4}({\rm midrad}(0,0.2))=[0.04357723543460\times 10^{2},1.46072895729964\times 10^{2}],
\]
whence $f(k)>0$ for $k\in(0,0.2]$. Also we can rearrange 
\[
f(k)=(e^{2k}+1)^{-4}((e^{2k}- 8k)e^{6k}+( 16k^2(e^{2k}-1)-2)e^{4k} + 16k^2e^{2k} + 8ke^{2k} + 1)
\]
so that, clearly, $f(k)>0$, say, for $k\in(2,\infty)$. We can divide $[0.2,2]$ into finitely many subintervals and rigorously enclose $f$ for each subinterval. 

Since $\sigma_-(0),\sigma_-(\kappa)=0$ and since $\frac{d^2\sigma_-}{dk^2}(k)>0$ for all $k>0$, $\sigma_-$ has a unique critical point, denoted $-\kappa_c$, in the interval $(0,\kappa)$ (see Section~\ref{sec:eps=0}), and 
\[
\frac{d\sigma_-}{dk}(k_2(\sigma))\geq \frac{d\sigma_-}{dk}(\kappa)>\frac{ d\sigma_-}{dk}(-\kappa_c)=0\quad \text{for $k_2(\sigma)\geq \kappa>-\kappa_c>0$},
\]
whence $\frac{dk_2}{d\sigma}(\sigma)=\frac{1}{\frac{d\sigma_-}{dk}(k_2(\sigma)) }>0$ for all $\sigma\geq0$.

Since $k_2(\sigma), k_4(\sigma) \to \infty$ as $\sigma\to \infty$, 
\[
k_2(\sigma)-k_4(\sigma)=\sqrt{\mu_0 k_2(\sigma)\tanh(k_2(\sigma))}+\sqrt{\mu_0 k_4(\sigma)\tanh(k_4(\sigma))}\to \infty\quad\text{as $\sigma\to \infty$.}
\]
This completes the proof.
\end{proof}

\begin{lemma}\label{lemma:k2minusk4;upper_estimate}
We have $\kappa<k_2(\sigma_c)-k_4(\sigma_c)<2\kappa$, where $\sigma_c>0$ is the critical value of $\sigma_+$.
\end{lemma}

\begin{proof}
Recall from the proof of Lemma~\ref{lemma:k2minusk4;monotonity} that $\frac{d\sigma_-}{dk}(k)$ is a strictly increasing function of $k>0$. Recall from Section~\ref{sec:eps=0} that $\sigma_-(\kappa)=0$ and $-\sigma_c=\sigma_-(-k_c)$ is the critical value of $\sigma_-$. Therefore
\be\label{useofconcavity}
\sigma_c=\int_{-k_c}^\kappa \frac{d\sigma_-}{dk}(k)~dk<\int_{\kappa}^{2\kappa+k_c} \frac{d\sigma_-}{dk}(k)~dk=\sigma_-(2\kappa+k_c).
\ee 

Lemma~\ref{lemma:k2minusk4;monotonity} asserts that there exists a unique $\sigma_2>0$ such that \eqref{def:resonance} holds for $N=2$, that is, $k_2(\sigma_2)-k_4(\sigma_2)=2\kappa$. Also
\be 
\label{kkappacompare}
\kappa<2\kappa+k_c<2\kappa<2\kappa+k_4(\sigma_2)=k_2(\sigma_2).
\ee 
Here the first two inequalities use $-\kappa<k_c<0$ and the third inequality uses $k_4(\sigma_2)>0$. Since $\sigma_-(k)$ is a strictly increasing function for $k>-k_c$, it follows from \eqref{useofconcavity} and \eqref{kkappacompare} that
\[
\sigma_c<\sigma_-(2\kappa+k_c)<\sigma_-(k_2(\sigma_2))=\sigma_2.
\]
Therefore it follows from Lemma~\ref{lemma:k2minusk4;monotonity} that $k_2(\sigma_c)-k_4(\sigma_c)<2\kappa$. Also it follows from Lemma~\ref{lemma:k2minusk4;monotonity} that $\kappa=k_2(0)-k_4(0)<k_2(\sigma_c)-k_4(\sigma_c)$. This completes the proof.
\end{proof}

When $0<\sigma\leq\sigma_c$, it follows from Lemmas~\ref{lemma:k2minusk4;monotonity} and \ref{lemma:k2minusk4;upper_estimate} that \eqref{def:resonance} does not hold true for any $N>0,\in\mathbb{Z}$. Also $0\leq k_3(\sigma)-k_1(\sigma)<\kappa$. 
When $\sigma>\sigma_c$, on the other hand, 
there are 
\[
\sigma_c<\sigma_2<\cdots <\sigma_N<\cdots \to\infty \quad \text{as $N\to\infty$},
\] 
for which $k_2(\sigma_N)-k_4(\sigma_N)=N\kappa$. 

In what follows, let $\sigma>\sigma_c$. 
The result of Lemma~\ref{lem:eps=0} leads to that 
\begin{equation}\label{eqn:delta=0;high}
\Delta(i\sigma,k_j(\sigma)+K\kappa;0)=0, \quad j=2,4,\quad \text{for all $K\in\mathbb{Z}$}.
\end{equation}
We shall proceed as in the previous section and examine the zeros of $\Delta(\lambda,k;\eps)$ for $(\lambda,k,\eps)$ in the vicinity of $(i\sigma, k_j(\sigma)+K\kappa,0)$, $j=2,4$ and $K\in\mathbb{Z}$. We begin by determining the monodromy matrix. 

\subsection{Expansion of the monodromy matrix}\label{sec:aH}

Throughout the subsection, $\sigma>\sigma_c$. Let 
\[
\begin{pmatrix}\mathbf{v}_1&\mathbf{v}_2\end{pmatrix}(x;\sigma,\delta,\eps)
=\begin{pmatrix}\boldsymbol{\phi}_2(\sigma)&\boldsymbol{\phi}_4(\sigma)\end{pmatrix}
\mathbf{X}(x;\sigma,\delta,\eps)
\] 
denote a fundamental matrix of \eqref{eqn:A}, where $\boldsymbol{\phi}_j(\sigma)$, $j=2,4$, are in \eqref{def:phi1-4} and $\mathbf{X}(x;\sigma,\delta,\eps)$ in Definition~\ref{def:Evans}. We write
\be \label{def:a;expH}
\mathbf{v}_k(x;\sigma,\delta,\eps)=\sum_{j=1}^2\Big(\sum_{m+n=0}^{\infty}a_{jk}^{(m,n)}(x)\delta^m\eps^n\Big)\boldsymbol{\phi}_{2j}(\sigma)
\ee 
for $|\delta|,|\eps|\ll 1$, where $a_{jk}^{(m,n)}(x)$, $j,k=1,2$ and $m,n=0,1,2,\dots$, are to be determined. Let
\[
\mathbf{a}^{(m,n)}(x)=(a_{jk}^{(m,n)}(x))_{j,k=1,2},
\] 
and we may assume that
\be \label{cond:aH}
\mathbf{a}^{(0,0)}(0)=\mathbf{I}\quad\text{and}\quad 
\mathbf{a}^{(m,n)}(0)=\mathbf{0}\quad \text{for $m+n\geq 1$}.
\ee
Our task is to evaluate $\mathbf{a}^{(m,n)}(T)$, $m,n=0,1,2,\dots$. We write \eqref{def:u2} as
\be\label{def:b;expH}
\mathbf{w}(x,\mathbf{v}_k(x);\sigma,\delta,\eps)=\sum_{m+n=1}^\infty \mathbf{w}_{k}^{(m,n)}(x;\sigma)\delta^m\eps^n
\ee
for $|\delta|,|\eps|\ll 1$, where $\mathbf{w}_{k}^{(0,0)}(x;\sigma)=\mathbf{0}$, $k=1,2$, and $\mathbf{w}_{k}^{(m,n)}(x;\sigma)$, $k=1,2$ and $m+n\geq1$, are to be determined. Recall \eqref{eqn:LB;delta} and we write 
\be\label{def:B;expH}
\mathbf{B}(x;\sigma,\delta,\eps)=\sum_{m+n=1}^\infty\mathbf{B}^{(m,n)}(x;\sigma)\delta^m\eps^n
\ee
for $|\delta|,|\eps|\ll 1$, where $\mathbf{B}^{(0,0)}(x;\sigma)=\mathbf{0}$, and $\mathbf{B}^{(m,n)}(x;\sigma)$, $1\leq m+n\leq2$, are in Appendix~\ref{A:Bexp}. Notice that $\mathbf{B}^{(m,0)}(x;\sigma)$, $m\geq1$, do not involve $x$. 

Inserting \eqref{def:a;expH}, \eqref{def:b;expH} and \eqref{def:B;expH} into the former equation of \eqref{eqn:LB;u12}, we recall Lemma~\ref{lem:eps=0} and make a straightforward calculation to obtain
\be \label{eqn:a00H}
\mathbf{a}^{(0,0)}(x)=\begin{pmatrix} 
e^{ik_2(\sigma)x}&0 \\ 0&e^{ik_4(\sigma)x}\end{pmatrix},
\ee 
and for $m+n\geq1$ we arrive at 
\be \label{eqn:aH}
\sum_{j=1}^2\Big(\frac{d}{dx}a_{jk}^{(m,n)}\Big)\boldsymbol{\phi}_{2j}(\sigma)=
ik_2(\sigma)a_{1k}^{(m,n)}\boldsymbol{\phi}_2(\sigma)
+ik_4(\sigma)a_{2k}^{(m,n)}\boldsymbol{\phi}_4(\sigma)+\boldsymbol{\Pi}(\sigma)\mathbf{f}_{k}^{(m,n)}(x;\sigma),
\ee  
where $\boldsymbol{\boldsymbol{\Pi}}(\sigma)$ is in \eqref{def:Pi;high} and
\be \label{def:fH}
\mathbf{f}_{k}^{(m,n)}(x;\sigma)=\sum_{\substack{0\leq m'\leq m\\0\leq n' \leq n}}
\mathbf{B}^{(m',n')}(x;\sigma)\big(\mathbf{w}_{k}^{(m-m',n-n')}(x;\sigma)
+\sum_{j=1}^2a_{jk}^{(m-m',n-n')}\boldsymbol{\phi}_{2j}(\sigma)\big).
\ee 

Inserting \eqref{def:a;expH}, \eqref{def:b;expH} and \eqref{def:B;expH} into the latter equation of \eqref{eqn:LB;u12}, at the order of $\delta^m\eps^n$, $m+n\geq1$, we arrive at \eqref{eqn:red0}, where $\boldsymbol{\Pi}(\sigma)$ (see \eqref{def:Pi;high}) replaces $\boldsymbol{\Pi}(0)$, and \eqref{def:fH} replaces $\mathbf{f}^{(m,n)}_k(x;0)$. Notice that since $\mathbf{B}^{(0,0)}(x;\sigma)=\mathbf{0}$, the right side of \eqref{def:fH} does not involve $\mathbf{w}^{(m.n)}_k(x;\sigma)$, and it is made up of lower order terms. We solve this, as we do \eqref{eqn:red0}, to determine $\mathbf{w}^{(m,n)}_k(x;\sigma)$. The result is in Appendix~\ref{A:reduction;high}. We pause to remark that we use the Symbolic Math Toolbox in MATLAB for extremely long and tedious algebraic manipulations.

For $m+n\geq1$, we may write the solution of \eqref{eqn:aH} and \eqref{cond:aH} as
\begin{equation}\label{def:a>0H}
\begin{aligned}
a_{jk}^{(m,n)}(x)&=e^{ik_{2j}(\sigma)x}\int_0^xe^{-ik_{2j}(\sigma)x'}
\big\langle\mathbf{f}^{(m,n)}_k(x';\sigma),\boldsymbol{\psi}_{2j}(\sigma)\big\rangle~dx'\\
&\begin{aligned}=e^{ik_{2j}(\sigma)x}\Big\langle
\sum_{\substack{0\leq m'\leq m\\0\leq n' \leq n}}\int_0^xe^{-ik_{2j}(\sigma)x'}\mathbf{B}^{(m',n')}(x';\sigma)
\big(&\mathbf{w}_{k}^{(m-m',n-n')}(x';\sigma)\\[-11pt]
+&\sum_{j'=1}^2a_{j'k}^{(m-m',n-n')}(x')\boldsymbol{\phi}_{2j'}(\sigma)\big)~dx',
\boldsymbol{\psi}_{2j}(\sigma)\Big\rangle,\end{aligned}
\end{aligned}
\end{equation}
where $\boldsymbol{\phi}_{2j}(\sigma)$ and $\boldsymbol{\psi}_{2j}(\sigma)$, $j=1,2$, are in \eqref{def:phi1-4} and \eqref{def:psi24}, \eqref{def:cj24}, and $\langle\,,\rangle$ in \eqref{def:inner}. 
When $m'=m$ and $n'=n$, for instance, we recall \eqref{def:b;expH} and \eqref{eqn:a00H}, and the integral on the right side of \eqref{def:a>0H} becomes
\begin{equation}\label{eqn:m'=m,n'=n}
\int_{0}^xe^{-ik_{2j}(\sigma)x'}\mathbf{B}^{(m,n)}(x';\sigma)e^{ik_{2j'}(\sigma)x'}\boldsymbol{\phi}_{2j'}(\sigma)~dx'
=\int_{0}^xe^{i(k_{2j}(\sigma)-k_{2j'}(\sigma))x'}\mathbf{B}^{(m,n)}(x';\sigma)\boldsymbol{\phi}_{2j'}(\sigma)~dx'.
\end{equation}
We deduce from the result of Appendices~\ref{A:stokes} and \ref{A:Bexp}, and \eqref{def:stokes1}, \eqref{def:phi1-4} that $\mathbf{B}^{(m,n)}(x;\sigma)\boldsymbol{\phi}_{2j'}(\sigma)$, $j'=1,2$, are made up of $\sin(N\kappa x)$ or $\cos(N\kappa x)$ with respect to $x$, $N>0,\in\mathbb{Z}$, and, for instance,
\begin{equation}\label{piecewise}
\begin{aligned}
\int_{0}^x &e^{i(k_{2j}(\sigma)-k_{2j'}(\sigma))x'}\sin(N\kappa x')~dx'\\
&=\left\{\begin{aligned}&\frac{N\kappa-N\kappa\cos(N\kappa x)e^{i(k_{2j}-k_{2j'})x}+i(k_{2j}-k_{2j'})\sin(N\kappa x)e^{i(k_{2j}-k_{2j'})x}}{N^2\kappa ^2-(k_{2j}-k_{2j'})^2}\quad
&&\text{if $|k_{2j}-k_{2j'}|\neq N\kappa$},\\
&\pm\frac{i}{2}x+\frac{1-e^{\pm 2iN\kappa x}}{4N\kappa}&&\text{if $k_{2j}-k_{2j'}=\pm N\kappa$}.
\end{aligned}\right.
\end{aligned}
\end{equation}
Therefore the result is different, depending on whether \eqref{def:resonance} holds true or not. When $0\leq m'\leq m$, $0\leq n'\leq n$ and $m'+n'\neq m+n$, likewise, we deduce from the result of Appendices~\ref{A:stokes}, \ref{A:Bexp} and \ref{A:reduction;high} that $\mathbf{B}^{(m',n')}(x;\sigma)\mathbf{w}_{k}^{(m-m',n-n')}(x;\sigma)$ are made up of $\sin(N\kappa x)$ or $\cos(N\kappa x)$ with respect to $x$, $N>0,\in\mathbb{Z}$, and the integral on the right side of \eqref{def:a>0H} can be treated as we do for \eqref{eqn:m'=m,n'=n} and \eqref{piecewise}. Therefore, the result of \eqref{def:a>0H} is different, depending on whether \eqref{def:resonance} holds true or not.  

In what follows, we focus the attention to $k_2(\sigma)-k_4(\sigma)=N\kappa$ for some $N\geq2,\in\mathbb{Z}$, where $\kappa>0$ is the wave number, for which it turns out that spectral instability is possible in the vicinity of $i\sigma\in\mathbb{C}$. We write $k_2=k_2(\sigma)$ and $k_4=k_4(\sigma)$ for the simplicity of notation.

\begin{lemma}\label{coefficients_highn1}
If \eqref{def:resonance} holds true then
\be 
\label{coeff_g1}
\mathbf{a}^{(0,0)}(T)=e^{ik_4T}\begin{pmatrix}1&0\\0&1\end{pmatrix},\quad 
\mathbf{a}^{(1,0)}(T)=\begin{pmatrix}a_{11}^{(1,0)}&0\\0&a_{22}^{(1,0)}\end{pmatrix},\quad
\mathbf{a}^{(0,1)}(T)=\begin{pmatrix}0&0\\0&0\end{pmatrix},
\ee
where $a_{jj}^{(1,0)}\neq0$ and $a_{jj}^{(1,0)}\in e^{ik_4T}\mathbb{R}$, for $j=1,2$. 
If \eqref{def:resonance} does not hold true for any $N\geq2,\in\mathbb{Z}$ then 
\[
\mathbf{a}^{(0,0)}(T)=\begin{pmatrix}e^{ik_2T}&0\\0&e^{ik_4T}\end{pmatrix},\quad 
\mathbf{a}^{(1,0)}(T)=\begin{pmatrix} * & * \\ * & * \end{pmatrix},\quad
\mathbf{a}^{(0,1)}(T)=\begin{pmatrix}0&*\\ *&0\end{pmatrix}.
\]
Particularly, the off-diagonal entries of $ \mathbf{a}^{(1,0)}(T)$ and $\mathbf{a}^{(0,1)}(T)$ are not zero. 
\end{lemma}
\begin{proof}
If \eqref{def:resonance} holds true then we evaluate \eqref{eqn:a00H} at $x=T$ and verify the first equation of \eqref{coeff_g1}. 

A straightforward calculation reveals that 
\begin{equation}\label{def:a10H}
\begin{aligned} 
a_{11}^{(1,0)}(x)=&\frac{2k_2s_2(2)}{k_2s_2(2) + \sigma s_2(2) + 2k_2\sigma - 2k_2^2} x  e^{ik_2x} ,\\
a_{12}^{(1,0)}(x)=&a_{12,c}^{(1,0)} ( e^{ik_2x} -  e^{ik_4x}) ,\\
a_{21}^{(1,0)}(x)=&a_{21,c}^{(1,0)}( e^{ik_4x} -  e^{ik_2x}) ,\\
a_{22}^{(1,0)}(x)=&\frac{2k_4s_4(2)}{k_4s_4(2) + \sigma s_4(2) + 2k_4\sigma - 2k_4^2} x  e^{ik_4x} ,\\
\end{aligned}
\end{equation}
where 
$$
\begin{aligned}
a_{12,c}^{(1,0)}=&-\frac{k_2k_4^2c_4s_2 + k_2^2k_4c_2s_4 + 2k_2^2k_4c_4s_2 - k_2\sigma^2c_4s_2 + k_4\sigma^2c_2s_4 - 2k_2k_4\sigma c_2s_4 - 2k_2k_4\sigma c_4s_2}{\kappa(k_2 + k_4)(k_2 - \sigma)(k_2s_2(2) + \sigma s_2(2) + 2k_2\sigma - 2k_2^2)}i\\
a_{21,c}^{(1,0)}=&\frac{2 k_2 k_4^2  c_2 s_4 + k_2 k_4^2 c_4 s_2 + k_2^2 k_4  c_2 s_4 + k_2\sigma^2 c_4 s_2 - k_4\sigma^2  c_2 s_4 - 2 k_2 k_4\sigma  c_2 s_4 - 2 k_2 k_4\sigma c_4 s_2}{\kappa (k_2 + k_4) (k_4 -\sigma) (k_4 s_4(2) +\sigma s_4(2) + 2 k_4\sigma - 2 k_4^2)}i.
\end{aligned}
$$
Here and elsewhere, we use the notation
\begin{align*}
 &s_2=\sinh(k_2),& &c_2=\cosh(k_2),& &s_2(2)=\sinh(2k_2),\\
 &s_4=\sinh(k_4), & &c_4=\cosh(k_4), & &s_4(2)=\sinh(2k_4).
\end{align*}

If \eqref{def:resonance} holds true then we evaluate \eqref{def:a10H} at $x=T$ and obtain the second equation of \eqref{coeff_g1}. If \eqref{def:resonance} does not hold true for any $N\geq2,\in\mathbb{Z}$ then $a_{jk}^{(1,0)}(T)\neq0$, $j,k=1,2$. 

When $m=0$ and $n=1$, we rewrite \eqref{def:a>0H} as
\begin{align*}
a_{jj}^{(0,1)}(T)=&e^{ik_{2j} T}\Big\langle \int_0^Te^{-ik_{2j}x}\mathbf{B}^{(0,1)}(x;\sigma)e^{ik_{2j}x}\boldsymbol{\phi}_{2j}(\sigma)~dx, \boldsymbol{\psi}_{2j}(\sigma)\Big\rangle\\
=&e^{ik_{2j} T}\Big\langle \int_0^T\mathbf{B}^{(0,1)}(x;\sigma)\boldsymbol{\phi}_{2j}(\sigma)~dx,\boldsymbol{\psi}_{2j}(\sigma)\Big\rangle.
\end{align*}
We deduce from \eqref{eqn:B01}, \eqref{def:stokes1} and \eqref{def:phi1-4} that $\mathbf{B}^{(0,1)}(x;\sigma)\boldsymbol{\phi}_{2j}(\sigma)$, $j=1,2$, depends linearly on $\sin(\kappa x)$ and $\cos(\kappa x)$, whence it vanishes after the integration over one period. Therefore $a_{jj}^{(0,1)}(T)=0$, $j=1,2$. On the other hand,
\[
a_{12}^{(0,1)}(T)
=e^{ik_{4} T}\Big\langle \int_0^Te^{i(k_2-k_4)x}\mathbf{B}^{(0,1)}(x;\sigma)\boldsymbol{\phi}_2(\sigma)~dx,\boldsymbol{\psi}_4(\sigma)\Big\rangle.
\]
If \eqref{def:resonance} holds true then we deduce that $e^{i(k_2-k_4)x}\mathbf{B}^{(0,1)}(x;\sigma)\boldsymbol{\phi}_2(\sigma)$ depends linearly on $e^{iN\kappa x}\sin(\kappa x)$ and $e^{iN\kappa x}\cos(\kappa x)$, whence we proceed as in the former of \eqref{piecewise}, and it vanishes after the integration over one period. Likewise, $a_{21}^{(0,1)}(T)=0$. Indeed, a straightforward calculation leads to 
\begin{equation}\label{eqn:a01H}
a_{jk}^{(0,1)}(x)=c_{jk,1}e^{i(k_{2k}+\kappa)x}+c_{jk,2}e^{i(k_{2k}-\kappa)x}-(c_{jk,1}+c_{jk,2})e^{ik_{2j}x}
\end{equation}
for nonzero constants $c_{jk,1}$ and $c_{jk,2}$, $j,k=1,2$. Evaluating \eqref{eqn:a01H} at $x=T$, we obtain  $a_{jk}^{(0,1)}(T)=0$, $j,k=1,2$. If \eqref{def:resonance} does not hold true for any $N\geq2,\in\mathbb{Z}$, then we evaluate \eqref{eqn:a01H} at $x=T$, and $a_{12}^{(0,1)}(T), a_{21}^{(0,1)}(T)\neq 0$. This completes the proof.
\end{proof}

\begin{lemma}\label{coefficients_highn2}
If \eqref{def:resonance} holds true for some $N\geq 3,\in\mathbb{Z}$ then $a_{jk}^{(0,2)}(T)=0$, $j\neq k$, and $a_{jj}^{(0,2)}(T)$, $j=1,2$, do not necessarily vanish. If $\eqref{def:resonance}$ holds true for $N=2$ or \eqref{def:resonance} does not hold true for any $N\geq2,\in\mathbb{Z}$, on the other hand, then none of $a_{jk}^{(0,2)}(T)$ has to vanish. In particular, if \eqref{def:resonance} holds true for some $N\geq 2,\in\mathbb{Z}$, then $a_{i,j}^{(0,2)}\in ie^{ik_4T}\mathbb{R}$, for $i,j=1,2$.
\end{lemma}

\begin{proof}

When $m=0$ and $n=2$, we rewrite \eqref{def:a>0H} as
\ba \label{eqn:a02H}
a_{jk}^{(0,2)}(T)
=&e^{ik_{2j}T}\Big\langle\int_0^Te^{i(k_{2k}-k_{2j})x}\mathbf{B}^{(0,2)}(x;\sigma)\boldsymbol{\phi}_{2k}(\sigma)~dx,\boldsymbol{\psi}_{2j}(\sigma)\Big\rangle\\
&+e^{ik_{2j}T}\Big\langle\int_0^T\mathbf{B}^{(0,1)}(x;\sigma)e^{-ik_{2j}x}(\mathbf{w}^{(0,1)}_k(x;\sigma)+\sum_{j'=1}^2a_{j'k}^{(0,1)}(x)\boldsymbol{\phi}_{2j'}(\sigma))~dx,\boldsymbol{\psi}_{2j}(\sigma)\Big\rangle,
\ea
where $\mathbf{B}^{(0,1)}(x;\sigma)$ and $\mathbf{B}^{(0,2)}(x;\sigma)$ are in Appendix~\ref{A:Bexp},  $\mathbf{w}^{(0,1)}_k(x;\sigma)$, $k=1,2$, are in Appendix~\ref{A:reduction;high}, and \eqref{eqn:a01H} holds true. 

If \eqref{def:resonance} holds true for some $N\geq 2,\in\mathbb{Z}$ then we deduce from the result of Appendix~\ref{A:Bexp}, \eqref{def:stokes1}, \eqref{def:stokes2} and \eqref{def:phi1-4} that $\mathbf{B}^{(0,2)}(x;\sigma)\boldsymbol{\phi}_{2k}(\sigma)$ depends linearly on 
\begin{gather*}
\sin(2\kappa x), \quad \cos(2\kappa x),\quad \sin(\kappa x), \quad \cos(\kappa x),\\
\sin^2(\kappa x)=\tfrac12(1-\cos(2\kappa x)), \quad \cos^2(\kappa x)=\tfrac12(\cos(2\kappa x)+1),\quad 
\sin(\kappa x)\cos(\kappa x)=\tfrac12\sin(2\kappa x)
\end{gather*}
with respect to $x$. Notice that $k_{2k}-k_{2j}=\pm N\kappa$, $j\neq k$. Therefore for $N\geq3$, we proceed as in the former of \eqref{piecewise}, and 
\[
\int_0^Te^{i(k_{2j}-k_{2k})x}\mathbf{B}^{(0,2)}(x;\sigma)\boldsymbol{\phi}_{2k}(\sigma)~dx=0,
\] 
whereas for $N=2$, we proceed as in the latter of \eqref{piecewise} and the integral does not vanish. 
When $j=k$, so that $k_{2j}-k_{2k}=0$, those terms of $\mathbf{B}^{(0,2)}(x;\sigma)\boldsymbol{\phi}_{2k}(\sigma)$ with a factor of $\sin^2(\kappa x)$ or $\cos^2(\kappa x)$ do not necessarily vanish after the integration over one period for any $N\geq2,\in\mathbb{Z}$. 

It remains to show that if \eqref{def:resonance} holds true for some $N\geq 3,\in\mathbb{Z}$ then the second term on the right side of \eqref{eqn:a02H} vanishes when $j\neq k$. Multiplying \eqref{eqn:a01H} by $e^{-ik_4x}$, we arrive at
\[
e^{-ik_{4}x}a_{11}^{(0,1)}(x)=c_{11,1}e^{i(N+1)\kappa x}+c_{11,2}e^{i(N-1)\kappa x}-(c_{11,1}+c_{11,2})e^{iN\kappa x},
\]
while we deduce from \eqref{eqn:B01} and \eqref{def:stokes1} that $\mathbf{B}^{(0,1)}(x;\sigma)$ depends linearly on $\sin(\kappa x)$ or $\cos(\kappa x)$. 
Therefore we proceed as in the former of \eqref{piecewise}, and ${\displaystyle \int^T_0\mathbf{B}^{(0,1)}(x;\sigma)e^{-ik_{4}x}a_{11}^{(0,1)}(x)~dx=0}$. Likewise,
\[
e^{-ik_4x}a_{21}^{(0,1)}(x)=c_{21,1}e^{i(N+1)\kappa x}+c_{21,2}e^{i(N-1)\kappa x}-(c_{21,1}+c_{21,2}),
\]
and $\mathbf{B}^{(0,1)}(x;\sigma)$ depends linearly on $\sin(\kappa x)$ or $\cos(\kappa x)$. Therefore it becomes zero after the integration over one period. Also we verify that
\[
\int^T_0\mathbf{B}^{(0,1)}(x;\sigma)e^{-ik_{2}x}(a_{12}^{(0,1)}(x)\boldsymbol{\phi}_{2}(\sigma)+a_{22}^{(0,1)}(x)\boldsymbol{\phi}_{4}(\sigma))~dx=0.
\] 
Lastly, since each term of $\mathbf{w}^{(0,1)}_k(x;\sigma)$ is made up of terms each having the factor of $e^{ik_{2k}x}$ (see \eqref{def:w(0,1)1}), 
\[
\int^T_0\mathbf{B}^{(0,1)}(x;\sigma)e^{-ik_{2j}x}\mathbf{w}^{(0,1)}_k(x;\sigma)~dx=0.
\]
This completes the proof. 
\end{proof}

\subsection{Spectral instability indices}\label{sec:ind23}

Let 
\[
\sigma>\sigma_c\quad\text{and}\quad k_2(\sigma)-k_4(\sigma)=N\kappa\quad\text{for some $N\geq2,\in\mathbb{Z}$},
\] 
where $\kappa>0$ is the wave number. 
Let
\[
\text{$\lambda=i\sigma+\delta$, $\delta\in\mathbb{C}$ and $|\delta|\ll1$},\quad
\text{$k=k_j(\sigma)+K\kappa+\gamma$, $j=2,4$, $K\in\mathbb{Z}$, $\gamma\in\mathbb{R}$ and $|\gamma|\ll1$},
\]
$\eps\in\mathbb{R}$ and $|\eps|\ll1$, and putting the result of Lemmas~\ref{coefficients_highn1} and \ref{coefficients_highn2}, we arrive at 
\ba  \label{expanddelta2}
\Delta(i\sigma&+\delta,k_j(\sigma)+K\kappa+\gamma;\eps)\\
=&d^{(2,0,0)}\delta^2+d^{(0,2,0)}\gamma^2+d^{(0,0,4)}\eps^4+d^{(1,1,0)}\delta\gamma+d^{(1,0,2)}\delta\eps^2+d^{(0,1,2)}\gamma\eps^2\\
&+o((|\delta|+|\gamma|+|\eps|^2)^2)
\ea 
as $\delta,\gamma,\eps\to 0$, where 
\ba \label{d_high}
&d^{(2,0,0)}=a_{11}^{(1,0)}a_{22}^{(1,0)},& d^{(0,2,0)}&=-T^2e^{2ik_4(\sigma)T},\\
&d^{(0,0,4)}=\det(\mathbf{a}^{(0,2)}),&d^{(1,1,0)}&=-iTe^{ik_4(\sigma)T}(a_{11}^{(1,0)}+a_{22}^{(1,0)}),\\
&d^{(1,0,2)}=a_{11}^{(0,2)}a_{22}^{(1,0)}+a_{22}^{(0,2)}a_{11}^{(1,0)},\quad&d^{(0,1,2)}&=-iTe^{ik_4(\sigma)T}(a_{11}^{(0,2)}+a_{22}^{(0,2)}),
\ea 
and $a^{(m,n)}_{jk}(T)$, $j,k=1,2$ and $m,n=0,1,2,\dots$ are in the previous subsection. Recall $T=2\pi/\kappa$ and we suppress $T$ for the simplicity of notation. 
The formulae of $a_{jk}^{(0,2)}$, $j,k=1,2$, are too bulky to be displayed here\footnote{It takes $552$kB to save the formulae of $a_{jk}^{(0,2)}$ while it takes about $17$kB for $b^{0,1}_{1,6}$ and $b^{(0,1)}_{1,7}$ (see, for instance, Appendix \ref{A:reduction;high}).}. Recall from Lemma~\ref{coefficients_highn1} that $d^{(2,0,0)}\neq 0$. 
We pause to remark that if $k_2(\sigma)-k_4(\sigma)\neq N\kappa$ for any $N\geq2,\in\mathbb{Z}$ then 
\[
\Delta(i\sigma+\delta,k_j(\sigma)+K\kappa+\gamma;\eps)=O(|\delta|+|\gamma|+|\eps|^2)
\quad\text{as $\delta,\gamma,\eps\to0$}
\]
instead.

The Weierstrass preparation theorem asserts that \eqref{expanddelta2} becomes 
\[
\Delta(i\sigma+\delta,k_j(\sigma)+K\kappa+\gamma;\eps)=W(\delta,\gamma,\eps)h(\delta,\gamma,\eps)
\]
for $|\delta|, |\gamma|, |\eps|\ll 1$, where $W(\delta,\gamma,\eps)$ is a Weierstrass polynomial and, for $|\delta|\ll1$,
\ba
\label{weierstrass_high}
W(\delta,\gamma,\eps)=\delta^2&+\frac{d^{(1,1,0)}\gamma+d^{(1,0,2)}\eps^2+o(|\gamma|+|\eps|^2)}{d^{(2,0,0)}}\delta\\
&+\frac{d^{(0,2,0)}\gamma^2+d^{(0,1,2)}\gamma\eps^2+d^{(0,0,4)}\eps^4+o((|\gamma|+|\eps|^2)^2)}{d^{(2,0,0)}}
\ea
as $\gamma,\eps\to0$, $h(\delta,\gamma,\eps)$ is analytic at $(0,0,0)$ and $h(0,0,0)=d^{(2,0,0)}\neq0$. Therefore the zeros of $\Delta(i\sigma+\delta,k_j(\sigma)+p\kappa+\gamma;\eps)$ for $(\delta,\gamma,\eps)$ in the vicinity of $(0,0,0)$ are the two zeros of $W(\delta,\gamma,\eps)$. Indeed,
\ba  \label{delta_high}
\delta(\gamma,\eps)=&-\frac{d^{(1,1,0)}\gamma+d^{(1,0,2)}\eps^2+o(|\gamma|+|\eps|^2)}{2d^{(2,0,0)}}\\
&\pm\Big(\frac{(d^{(1,1,0)})^2-4d^{(2,0,0)}d^{(0,2,0)}}{4(d^{(2,0,0)})^2}\gamma^2+\frac{d^{(1,1,0)}d^{(1,0,2)}-2d^{(2,0,0)}d^{(0,1,2)}}{2(d^{(2,0,0)})^2}\gamma\eps^2 \\
&\qquad +\frac{(d^{(1,0,2)})^2-4d^{(2,0,0)}d^{(0,0,4)}}{4(d^{(2,0,0)})^2}\eps^4+o((|\gamma|+|\eps|^2)^2)\Big)^{1/2} \\
=:&\alpha^{(1,0)}\gamma+\alpha^{(0,2)}\eps^2+o(|\gamma|+|\eps|^2)
\pm\sqrt{\alpha^{(2,0)}\gamma^2+\alpha^{(1,2)}\gamma\eps^2+\alpha^{(0,4)}\eps^4+o((|\gamma|+|\eps|^2)^2)} \\
=:&\alpha^{(1,0)}\gamma+\alpha^{(0,2)}\eps^2+o(|\gamma|+|\eps|^2)
\pm\sqrt{Q(\gamma;\eps)+o((|\gamma|+|\eps|^2)^2)} 
\ea 
as $\gamma,\eps\to0$.


\begin{theorem}[Spectral stability and instability away from $0\in\mathbb{C}$]\label{thm:unstable2}
A $2\pi/\kappa$ periodic Stokes wave of sufficiently small amplitude in water of unit depth is spectrally unstable near $i\sigma\in \mathbb{C}$, $\sigma\in\mathbb{R}$, for which $k_2(\sigma)-k_4(\sigma)=2\kappa$, provided that 
\be \label{def:ind2}
{\rm ind}_2(\kappa):=\frac{a_{12}^{(0,2)}a_{21}^{(0,2)}}{a_{11}^{(1,0)}a_{22}^{(1,0)}}(T)>0,
\ee 
where $a^{(m,n)}_{jk}(T)$ is in \eqref{eqn:a00H} and \eqref{def:a>0H}, and $T=2\pi/\kappa$. It is spectrally stable at the order of $\eps^2$ as $\eps\to0$ otherwise, where $\eps$ is the dimensionless amplitude parameter. It is spectrally stable near $i\sigma\in \mathbb{C}$, $\sigma\in\mathbb{R}$, for which $k_2(\sigma)-k_4(\sigma)=N\kappa$ for $N\geq3,\in\mathbb{Z}$ at the order of $\eps^2$ as $\eps\to0$.
\end{theorem}

\begin{proof}
Recall from \eqref{d_high} and Lemmas~\ref{coefficients_highn1} and \ref{coefficients_highn2} that 
\[
\text{$\alpha^{(1,0)}$ and $\alpha^{(0,2)}$ are purely imaginary}
\] and 
\[
\text{$\alpha^{(2,0)}$, $\alpha^{(1,2)}$ and $\alpha^{(0,4)}$ are real}.
\] 
Moreover since
\[
\alpha^{(2,0)}=-\frac{T^2e^{2ik_4T}(a^{(1,0)}_{11}-a^{(1,0)}_{22})^2}{4(a^{(1,0)}_{11}a^{(1,0)}_{22})^2}<0
\]  
for $|\eps|\ll1$, $Q(\gamma)$ (see \eqref{delta_high}) takes its maximum $\mathrm{ind}_2(\kappa)\eps^4$ at ${\displaystyle \gamma=-\frac{\alpha^{(1,2)}}{2\alpha^{(2,0)}}\eps^2}$. 

If $k_2(\sigma)-k_4(\sigma)=N\kappa$ for some $N\geq 3,\in \mathbb{Z}$ then we infer from Lemma~\ref{coefficients_highn2} that $\mathrm{ind}_2(\kappa)=0$. This completes the proof. 
\end{proof}

\begin{remark*}\rm
To investigate stability and instability at the order higher than $\eps^2$, one has to expand the periodic Evans function to higher order and study the spectrum at 
\[
\gamma=-\frac{\alpha^{(1,2)}}{2\alpha^{(2,0)}}\eps^2=\frac{ie^{-ik_4T}}{T}\frac{a_{11}^{(0,2)}a_{22}^{(1,0)} - a_{11}^{(1,0)}a_{22}^{(0,2)}}{a_{11}^{(1,0)} - a_{22}^{(1,0)}}\eps^2.
\] 
\end{remark*}

\begin{remark*}\rm
McLean \cite{McLean;finite-depth} (see also \cite{McLean;infinite1,McLean;infinite2} for the infinite depth) numerically investigated the spectral stability and instability of a Stokes wave, for a range of $\eps$ for $\kappa$ smaller and greater than $\kappa_1$ (permitting transversal perturbations), and reported instability in the vicinity of $i\sigma\in\mathbb{C}$, $\sigma\in\mathbb{R}$, whenever $k_2(\sigma)-k_4(\sigma)=N\kappa$ for some $N\geq2,\in\mathbb{Z}$. See, for instance, \cite[Figures~2,3,4]{McLean;finite-depth}. See also \cite{FK, DO}, among others, for further numerical investigations. Theorem~\ref{thm:unstable2} explains the numerical finding for $N=2$. 
\end{remark*}

\begin{remark*}\rm
If we proceed as in Section~\ref{sec:ind1} instead, and let
\[
\lambda_{2j}(k_{2j}(\sigma)+\gamma,\eps)=i\sigma+
\alpha^{(1,0)}_{2j}\gamma+\alpha^{(0,2)}_{2j}\eps^2+o(|\gamma|+|\eps|^2),\quad j=1,2,
\]
as $\gamma, \eps\to0$, for some $\alpha^{(1,0)}_{2j}$ and $\alpha^{(0,2)}_{2j}$, then we arrive at
\[
{\rm ind}_3(\kappa)=\Big(\frac{(a_{11}^{(0,2)}a_{22}^{(1,0)}-a_{11}^{(1,0)}a_{22}^{(0,2)})^2
+4a_{12}^{(0,2)}a_{21}^{(0,2)}a_{11}^{(1,0)}a_{22}^{(1,0)}}{(a_{11}^{(1,0)}a_{22}^{(1,0)})^2}\Big)(T)>0,
\]
rather than \eqref{def:ind2}, which holds true if $\kappa_L<\kappa<\kappa_R$, where $\kappa_L\approx 0.86430$ and $\kappa_R\approx 1.00804$. 
\end{remark*}

\begin{figure}[htbp]
\begin{center}
\includegraphics[scale=0.38]{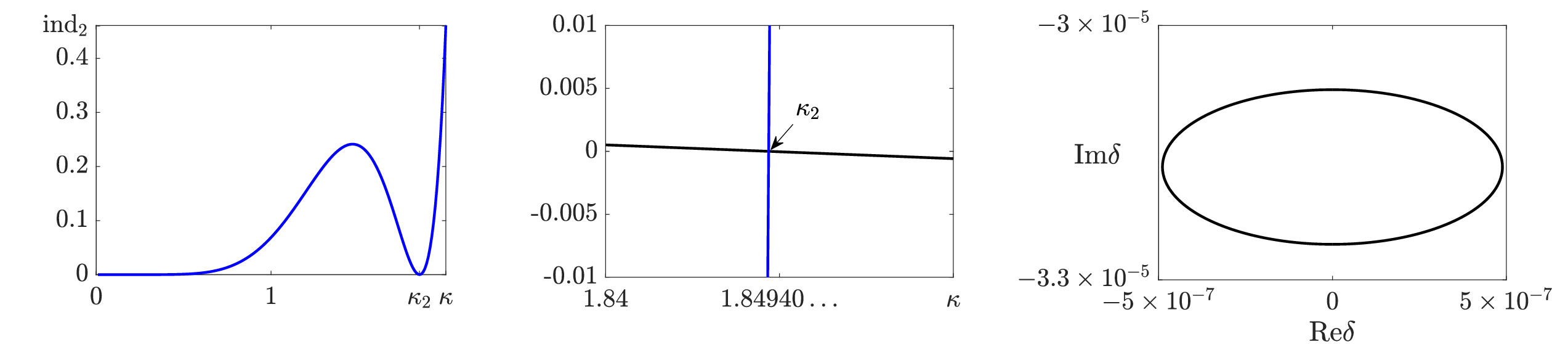}
\caption{Left: A plot of ${\rm ind}_2$ versus $\kappa$; ${\rm ind}_2(\kappa)>0$ unless $\kappa=\kappa_2$. Middle: A plot of the imaginary parts of  $\frac{a_{12}^{(0,2)}}{a_{11}^{(1,0)}}$ (black) and $\frac{a_{21}^{(0,2)}}{a_{22}^{(1,0)}}$ (blue) versus $\kappa$; both change their signs at $\kappa=\kappa_2$. Right: A plot of the leading part of $\delta(\gamma,\eps)$ for $\kappa=1.5$, $\eps=0.001$ and $\gamma$ in the interval bounded by the two zeros of $Q(\gamma)$, showing a ``bubble" (an ellipse) of unstable spectrum.}
\label{high_instability}
\end{center}
\end{figure}

Numerical evaluation reveals that ${\rm ind}_2(\kappa)>0$ unless $\kappa= \kappa_2:=1.849404083750\dots$ and ${\rm ind}_2(\kappa_2)=0$. See the left panel of Figure~\ref{high_instability}. Recall from Lemmas~\ref{coefficients_highn1} and \ref{coefficients_highn2} that $\frac{a_{12}^{(0,2)}}{a_{11}^{(1,0)}}$ and $\frac{a_{21}^{(0,2)}}{a_{22}^{(1,0)}}$ are purely imaginary. Numerical evaluation reveals that $\operatorname{Im}\Big(\frac{a_{12}^{(0,2)}}{a_{11}^{(1,0)}}\Big)$ and $\operatorname{Im}\Big(\frac{a_{21}^{(0,2)}}{a_{22}^{(1,0)}}\Big)$ change their signs at $\kappa=\kappa_2$. 
See the middle panel of Figure~\ref{high_instability}. 

The right panel of Figure~\ref{high_instability} provides an example of the leading part of \eqref{delta_high}, showing a ``bubble" (an ellipse) of unstable spectrum whose center ``drifts" from the origin by a distance of order $\eps^2$. Here the origin means $\lambda=i\sigma$ where \eqref{def:resonance} holds for $N=2$. Our result agrees with that from a formal perturbation method of \cite{creedon2021highfrequency}.

Therefore we conjecture that a $2\pi/\kappa$ periodic Stokes wave of sufficiently small amplitude in water of unit depth is spectrally unstable near $i\sigma\in\mathbb{C}$ , for which $k_2(\sigma)-k_4(\sigma)=2\kappa$, unless $\kappa=\kappa_2$, where $\kappa_2$ is the unique zero of \eqref{def:ind2} and $1.84940<\kappa_2<1.84941$. Recall from Corollary~\ref{cor:BF} that the Benjamin--Feir instability takes place for $\kappa>\kappa_1$ and $1.3627827<\kappa_1<1.3672828$. Therefore we conjecture that {\em all} Stokes waves of sufficiently small amplitude are spectrally unstable.

We wish to verify the numerical findings by means of rigorous analysis and validated numerics, as we do for Corollary~\ref{cor:BF}, whereby giving a computer-assisted proof of the conjecture. Specifically we wish to verify that $\operatorname{ind}_2(\kappa)>0$ for $0<\kappa\ll1$ and, also, $\kappa\gg1$. This seems unwieldy, though, because \eqref{def:ind2} involves an excessive number of symbolic expressions. Moreover the symbolic expressions depend on $\kappa$ and, also, $\sigma$ and $k_2(\sigma)$, $k_4(\sigma)$, subject to $k_2(\sigma)-k_4(\sigma)=2\kappa$. It would be helpful if one could simplify the symbolic expressions using 
\[
k_2(\sigma)-k_4(\sigma)=2\kappa\quad\text{and}\quad 
k_2(\sigma)-\sqrt{\mu_0 k_2(\sigma)\tanh(k_2(\sigma))}=k_4(\sigma)+\sqrt{\mu_0 k_4\tanh(k_4(\sigma))}=\sigma.
\]
This is a subject of future investigation.

Also we wish to establish a rigorous enclosure of \eqref{def:ind2} for finite intervals. We apply a Newton--Kantorovich theorem (see, for instance, \cite{Castelli2018,newton2,CHURCH2022133072} for applications to computer-assisted proofs) to
\be 
\label{F_func}
F(k_4,\sigma;\kappa):=((k_4+2\kappa)-\sqrt{\mu_0 (k_4+2\kappa)\tanh((k_4+2\kappa))}-\sigma,k_4+\sqrt{\mu_0 k_4\tanh(k_4)}-\sigma)
\ee
and rigorously enclose $k_2$, $k_4$ and $\sigma$ as a function of $\kappa$, suitably represented by an interval. See Appendix~\ref{rigor_enclosure} for details. Our INTLAB scripts require that the intervals of $\kappa$ be sufficiently small and, otherwise, suffer from huge wrapping effects because \eqref{def:ind2} involves an excessive number of operations of symbolic expressions. Also carrying out interval arithmetic computations of all the operations is a considerable amount of work. Nevertheless, rigorously validating the sign of \eqref{def:ind2} for sufficiently small intervals of $\kappa$ is possible. For instance, we make an INTLAB computation to learn that 
\begin{align*}
&{\rm Im}\left(\frac{a_{12}^{(0,2)}}{a_{11}^{(1,0)}}({\rm intval(1.84940)})\right)=[   0.22290970792399\times 10^{-6},   0.22302918033170\times 10^{-6}], \\
&{\rm Im}\left(\frac{a_{12}^{(0,2)}}{a_{11}^{(1,0)}}({\rm intval(1.84941)})\right)=[  -0.32307477032718\times 10^{-6},  -0.32297173040054\times 10^{-6}],\\
&{\rm Im}\left(\frac{a_{21}^{(0,2)}}{a_{22}^{(1,0)}}({\rm intval(1.84940)})\right)=[-0.00151096685232,  -0.00005120430605],\\
&{\rm Im}\left(\frac{a_{21}^{(0,2)}}{a_{22}^{(1,0)}}({\rm intval(1.84941)})\right)=[0.00049975168425,   0.00176360845785], 
\end{align*}
so that 
\begin{align*}
&{\rm Im}\left(\frac{a_{12}^{(0,2)}}{a_{11}^{(1,0)}}(1.84940)\right)>0,\quad
&&{\rm Im}\left(\frac{a_{12}^{(0,2)}}{a_{11}^{(1,0)}}(1.84941)\right)<0, \\
&{\rm Im}\left(\frac{a_{21}^{(0,2)}}{a_{22}^{(1,0)}}(1.84940)\right)<0, \quad
&&{\rm Im}\left(\frac{a_{21}^{(0,2)}}{a_{22}^{(1,0)}}(1.84941)\right)>0.
\end{align*}
Since ${\rm Im}\Big(\frac{a_{12}^{(0,2)}}{a_{11}^{(1,0)}}\Big)$ and ${\rm Im}\Big(\frac{a_{21}^{(0,2)}}{a_{22}^{(1,0)}}\Big)$ are continuous functions of $\kappa$, it follows from \eqref{def:ind2} that $\operatorname{ind}_2(\kappa)$ has a zero over the interval $(1.84940,1.84941)$.

\begin{corollary}\label{cor:high-freq}
There are Stokes wave of sufficiently small amplitude that are spectrally unstable away from $0\in\mathbb{C}$ but they are insusceptible to the Benjamin--Feir instability near $0\in\mathbb{C}$.
\end{corollary}

\begin{proof}
Recall from Corollary~\ref{cor:BF} that a $2\pi/\kappa$ periodic Stokes wave of sufficiently small amplitude in water of unit depth is Benjamin--Feir unstable if $\kappa>\kappa_1\in(1.3627827,1.3627828)$. We make an INTLAB computation to learn that
\[
{\rm ind}_2({\rm midrad(1.3,10^{-12})})=[   0.19949873745408,   0.20008799694687],
\]
where ${\rm midrad(1.3,10^{-12})}$ is an interval rigorously enclosing $[1.3-10^{-12},1.3+10^{-12}]$ and the right side rigorously encloses the range of $\operatorname{ind}_2$ for $\kappa \in [1.3-10^{-12},1.3+10^{-12}]$. Therefore ${\rm ind}_2(\kappa)>0$ for $\kappa\in[1.3-10^{-12},1.3+10^{-12}]$. This completes the proof.
\end{proof}

\begin{remark*}
In view of Corollary~\ref{cor:BF}, in order to prove that {\em all} Stokes waves of sufficiently small amplitude are spectrally unstable, one must demonstrate that $\operatorname{ind}_2(\kappa)>0$ for $\kappa>0$ sufficiently small, say, for $\kappa\in (0,\kappa_0]$ and rigorously compute that $\operatorname{ind}_2(\kappa)>0$ for finitely many subintervals of sufficiently small length which together cover, say, $[\kappa_0, 1.37]$.
\end{remark*}

\noindent{\bf Data Availibility Statement.}~Data sharing is not applicable to this article as no datasets were generated or analysed during the current study.

\newpage

\begin{appendix}

\section{Stokes expansion}\label{A:stokes}

Substituting \eqref{eqn:stokes exp} into \eqref{eqn:stokes} and \eqref{eqn:u;stokes}, at the order of $\eps^2$, we gather
\begin{equation}\label{eqn:stokes2}
\begin{aligned}
&{\phi_2}_{xx}+{\phi _{2}}_{yy}=2\eta_1{\phi_1}_{yy}+y{\eta_1}_{xx}{\phi_1}_y+2y{\eta_1}_x{\phi_1}_{xy}&&\text{for $0<y<1$}, \\
&{\phi_2}_y=0&&\text{at $y=0$}, \\
&{\eta_2}_x+{\phi_2}_y={\eta_1}_x{\phi_1}_x+\eta_1{\phi_1}_y &&\text{at $y=1$},\\
&\bar{\phi}_{2}+{\phi_2}_x-\mu_0\eta_2={\eta_1}_x{\phi_1}_y+\frac12({\phi_1}_x^2+{\phi_1}_y^2)+\mu_1\eta_1&&\text{at $y=1$},
\end{aligned}
\end{equation}
where $\phi_1,\eta_1$ and $\mu_0$ are in \eqref{def:stokes1} and \eqref{def:mu0}. Recall that $\phi_2$ and $\eta_2$ are $2\pi/\kappa$ periodic functions of $x$, $\phi_2$ is an odd function of $x$, $\eta_2$ is an even function and of mean zero over one period, and $\bar{\phi}_2$ and $\mu_1$ are constants. We solve \eqref{eqn:stokes2}, for instance, by the method of undetermined coefficients, to obtain
\ba\label{def:stokes2}
\bar{\phi}_{2}=&\frac{\mu_0^2}{4}\tanh(\kappa)^2,\\
\phi_2(x,y)=&\frac{3\mu _{0}\sin(2\kappa x)\cosh(2\kappa y)}{8\sinh(\kappa)\cosh(\kappa)}
+\frac{\mu _{0}\sinh(\kappa)^2}{2\cosh(\kappa)}y\sin(2\kappa x)\sinh(\kappa y),\\
\eta_2(x)=&\frac{\mu _{0}}{4}(2\sinh(\kappa)^2+3)\cos(2\kappa x),
\ea
and $\mu_1=0$. 

To proceed, at the order of $\eps^3$, we gather
\begin{align*}
&\begin{aligned}{\phi_3}_{xx}+{\phi_3}_{yy}
=&2y{\eta_1}_x{\phi_2}_{xy}+2\eta_1{\phi_2}_{yy}+y{\eta_1}_{xx}{\phi_2}_y\\
&+2y({\eta_2}_x-\eta_1{\eta_1}_x){\phi_1}_{xy}+(2\eta_2-3\eta_1^2-y^2{\eta_1}_x^2){\phi_1}_{yy}+y({\eta_2}_{xx}-2{\eta_1}_x^2-\eta_1{\eta_1}_{xx}){\phi_1}_{y}\end{aligned}
\intertext{for $0<y<1$, ${\phi_3}_y=0$ at $y=0$, and }
&\hspace*{-5pt}\left\{\begin{aligned}
&{\eta_3}_x+{\phi_3}_y={\eta_1}_x({\phi_2}_x+\bar{\phi}_{2})+\eta_1{\phi_2}_y+{\eta_2}_x{\phi_1}_x+(\eta_2-{\eta_1}_x^2-\eta_1^2){\phi_1}_y\\
&\begin{aligned}\bar{\phi}_{3}+{\phi_3}_x-\mu_0\eta_3=&{\phi_1}_x({\phi_2}_x+\bar{\phi}_{2})+{\eta_1}_x{\phi_2}_y+{\phi_1}_y({\phi_2}_y+{\eta_2}_x)\\
&-(\eta_1{\eta_1}_x+{\phi_1}_x{\eta_1}_x){\phi_1}_y-\eta_1{\phi_1}_y^2+\mu_2\eta_1\end{aligned}
\end{aligned}\right.
\end{align*}
at $y=1$, where $\phi_1,\eta_1,\mu_0$ are in \eqref{def:stokes1} and \eqref{def:mu0}, $\bar{\phi}_2,\phi_2,\eta_2$ are in \eqref{def:stokes2}. We likewise solve this by the method of undetermined coefficients to obtain $\bar{\phi}_3=0$,
\begin{subequations}\label{def:stokes3}
\begin{align}
&\begin{aligned}
\phi_3(x,y)=&-\frac{\mu_0^2(4s^2-9)}{16s(2)^2}\sin(3\kappa x)\cosh(3\kappa y) 
+\frac{3\mu_0^2s}{8(s^2+1)}y \sin(3\kappa x)\sinh(2\kappa y)\\
&+\frac{\mu_0^2s(2 s^2 + 3)}{8c}y\sin(3\kappa x)\sinh(\kappa y)
+\frac{\mu_0^2s^4}{8(s^2+1)} y^2\sin(3\kappa x)\cosh(\kappa y)\\
&+\frac{3\mu_0^2s}{8(s^2+1)}y\sin(\kappa x)\sinh(2\kappa y)
-\frac{\mu_0^2s(2s^2+3)}{8c}y\sin(\kappa x)\sinh(\kappa y) \\
&+\frac{\mu_0^2s^4}{8(s^2+1)}y^2\sin(\kappa x) \cosh(\kappa y),
\end{aligned}\\
&\hspace*{15pt}\begin{aligned}
\eta_3(x)=&\frac{\mu_0^2(24s^6+72s^4+72s^2+27)}{64(s^3+s)}\cos(3\kappa x)
+\frac{\mu_0^2s(5s^4+13s^2+6)}{8(s^2+1)}\cos(\kappa x),
\end{aligned}
\intertext{and}
&\hspace*{30pt}\mu_2=-\frac{\mu_0^3(8s^4+12s^2+9)}{8(s^2+1)},
\end{align}
\end{subequations}
where
\[
c=\cosh(\kappa),\quad s=\sinh(\kappa)\quad\text{and}\quad s(2)=\sinh(2\kappa).
\]

We do not include the formulae of $\bar{\phi}_4,\phi_4,\bar{\phi}_5,\phi_5,\dots$, $\eta_4,\eta_5,\dots$, $\mu_3,\mu_4,\dots$.

\section{Proof of Lemma~\ref{lem:symm}}\label{A:symm}

Let $\lambda\in{\rm spec}(\mathcal{L}(\eps))$, and suppose that $\mathbf{u}=\begin{pmatrix} \phi \\ u \\ \eta \end{pmatrix}\in L^\infty(\mathbb{R};Y)$, by abuse of notation, is a nontrivial solution of \eqref{eqn:spec} satisfying $\mathbf{u}(x+T)=e^{ik T}\mathbf{u}(x)$ for all $x\in\mathbb{R}$ for some $k\in\mathbb{R}$. Notice that \eqref{eqn:spec} remains invariant under
\[
\lambda\mapsto \lambda^*\quad\text{and}\quad \mathbf{u}\mapsto \mathbf{u}^*,
\]
and $\mathbf{u}^*(x+T)=e^{-ik T}\mathbf{u}^*(x)$ for all $x\in\mathbb{R}$, where the asterisk denotes complex conjugation. Thus $\lambda^*\in{\rm spec}(\mathcal{L}(\eps))$. Also notice that \eqref{eqn:spec} remains invariant under
\[
\lambda\mapsto -\lambda\quad\text{and}\quad
\begin{pmatrix} \phi(x) \\ u(x) \\ \eta(x) \end{pmatrix}\mapsto
\begin{pmatrix} -\phi(-x) \\ u(-x) \\ \eta(-x) \end{pmatrix}=:\mathbf{u}_-(x),
\]
and $\mathbf{u}_-(x+T)=e^{-ik T}\mathbf{u}_-(x)$ for all $x\in\mathbb{R}$. Thus $-\lambda\in{\rm spec}(\mathcal{L}(\eps))$. This completes the proof.

\section{Proof of Lemma~\ref{lem:eps=0}}\label{A:eps=0}

The proof of \eqref{def:phi12}, \eqref{def:phi34} and \eqref{def:phi1-4} is rudimentary. When $\sigma=\sigma_c$, clearly,
\[
(\mathbf{L}(i\sigma_c)-ik_c)\boldsymbol{\phi}_3(\sigma_c)=(\mathbf{L}(i\sigma_c)-ik_c)\begin{pmatrix} \mu_0\cosh(k_cy) \\ i(k_c-\sigma_c)\cosh(k_cy)\\ i(k_c-\sigma_c)\cosh(k_c)\end{pmatrix}=\mathbf{0}.
\]
It remains to solve
\[
(\mathbf{L}(i\sigma_c)-ik_c)\mathbf{u}
:=(\mathbf{L}(i\sigma_c)-ik_c)\begin{pmatrix} \phi \\ \upsilon \\ \eta \end{pmatrix}
=\begin{pmatrix} \mu_0\cosh(k_cy) \\ i(k_c-\sigma_c)\cosh(k_cy)\\ i(k_c-\sigma_c)\cosh(k_c)\end{pmatrix},
\]
where $\mathbf{u}\in {\rm dom}(\mathbf{L})$.
Solving the first and second equations, subject to $\phi_y(0)=0$ (see \eqref{def:dom}), we obtain
\[
\phi(y)=c\mu_0\cosh(k_cy)-i\mu_0y\sinh(k_cy)\quad\text{and}\quad
\upsilon(y)=(1+ic(k_c-\sigma_c))\cosh(k_cy)+(k_c-\sigma_c)y\sinh(k_cy)\]
for some constant $c$, whence $\eta=(1+ic(k_c-\sigma_c))\cosh(k_c)+(k_c-\sigma_c)\sinh(k_c)$ by \eqref{def:dom}. Substituting into the third equation, we arrive at
\[
c((\sigma-k_c)^2-\mu_0k_c\tanh(k_c))+i((\mu_0-(\sigma-k_c)^2)\tanh(k_c)+2\sigma - 2k_c + \mu_0k_c)=0.
\]
In view of \eqref{eqn:sigma}, this is solvable, provided that the second term on the left side vanishes. Differentiating \eqref{eqn:sigma} and evaluating $k=k_c$, we learn that
\[
2\sigma_+(k_c)-2k_c+\mu_0k_c=\mu_0k_c\tanh(k_c)^2-\mu_0\tanh(k_c),
\]
whence
$$
\begin{aligned}
&(\mu_0-(\sigma_+(k_c)-k_c)^2)\tanh(k_c)+2\sigma_+(k_c)- 2k_c+ \mu_0k_c\\
&=(\mu_0-(\sigma_+(k_c)-k_c)^2)\tanh(k_c)+\mu_0k_c\tanh(k_c)^2-\mu_0\tanh(k_c)=0.
\end{aligned}
$$
Therefore $c$ is arbitrary. Also we verify that
\[
\ker((\mathbf{L}(i\sigma_c)-ik_c\mathbf{1})^3)=\ker((\mathbf{L}(i\sigma_c)-ik_c\mathbf{1})^2),
\]
where $\mathbf{1}$ is the identity operator. This completes the proof.

\newpage

\section{Expansion of \texorpdfstring{$\mathbf{B}(x;\sigma,\delta,\epsilon)$}{Lg}}\label{A:Bexp}

Throughout the section, let $\mathbf{u}=\begin{pmatrix} \phi\\ \upsilon \\ \eta\end{pmatrix}$. Recalling \eqref{def:B;exp0} and \eqref{def:B;expH}, we use \eqref{eqn:spec}, \eqref{def:tildeu}, \eqref{def:u(tildeu)}, \eqref{def:u(tildeu)_y}, \eqref{def:u_x}, \eqref{def:L}, \eqref{def:stokes1}, \eqref{def:stokes2}, \eqref{def:stokes3}, and make straightforward calculations to show that
\be \label{eqn:B10B20}
\mathbf{B}^{(1,0)}(x;\sigma)\mathbf{u}=\begin{pmatrix}\phi\\-\upsilon-2i\sigma\mu_0^{-1}\phi \\ \eta\end{pmatrix}
\quad\text{and}\quad
\mathbf{B}^{(2,0)}(x;\sigma)\mathbf{u}=\begin{pmatrix}0\\-\mu_0^{-1}\phi\\0\end{pmatrix}.
\ee
Notice that $\mathbf{B}^{(m,0)}$, $m\geq1$, do not involve $x$. 

Let $\mathbf{B}^{(0,1)}(x;\sigma)\mathbf{u}=\begin{pmatrix}
(\mathbf{B}^{(0,1)}(x;\sigma)\mathbf{u})_1\\
(\mathbf{B}^{(0,1)}(x;\sigma)\mathbf{u})_2 \\
(\mathbf{B}^{(0,1)}(x;\sigma)\mathbf{u})_3\end{pmatrix}$, and we make straightforward calculations to show that
\begin{align}
&\begin{aligned}\label{eqn:B01}
(\mathbf{B}^{(0,1)}(x;\sigma)\mathbf{u})_1
=&i\sigma{\phi_1}_x(1)\phi+(y{\eta_1}_x+{\phi_1}_y(1))\phi_y-y{\phi_1}_y\phi_y(1)\\
&+(\mu_0{\phi_1}_x(1)-i\sigma{\phi_1}_y(1))\upsilon+i\sigma y{\phi_1}_y\eta,\\
(\mathbf{B}^{(0,1)}(x;\sigma)\mathbf{u})_2
=&\left(i\sigma^3 {\phi_1}_y(1)+ \mu_0 \sigma^2 {\phi_1}_x(1) - i\mu_0 \sigma {\phi_1}_{xx}(1) \right)\mu_0^{-2}\phi-\left({\phi_1}_{xy}(1)+ 2i \sigma {\phi_1}_y(1)\right)\mu_0^{-1}\phi_y\\
&+\left(\mu_0 {\phi_1}_x(1) + 2 \mu_0 \eta_1 - i\sigma {\phi_1}_y(1) \right)\mu_0^{-2}\phi_{yy}+\left(i\sigma y {\phi_1}_y  - y {\phi_1}_{xy}\right)\mu_0^{-1}\phi_y(1)\\
&+\left(i\sigma {\phi_1}_{xy}(1)  - \mu_0 {\phi_1}_{xx}(1) - i\mu_0 \sigma {\phi_1}_x(1)\right)\mu_0^{-1}\upsilon
+\left(y {\eta_1}_x - {\phi_1}_y(1)\right)\upsilon_y \\
&+\left(2 {\phi_1}_{yy} + \sigma^2 y {\phi_1}_y + i\sigma y {\phi_1}_{xy} \right)\mu_0^{-1}\eta,\\
(\mathbf{B}^{(0,1)}(x;\sigma)\mathbf{u})_3
=&i\sigma {\eta_1}_x\phi(1)+\left(\eta_1 - {\phi_1}_x(1)\right)\phi_y(1)+\mu_0 {\eta_1}_x\upsilon(1)+\left({\phi_1}_y(1) + i\sigma {\phi_1}_x(1) \right)\eta,
\end{aligned}
\intertext{
where $\phi_1$ and $\eta_1$ are in \eqref{def:stokes1}, and, likewise,}
&\begin{aligned}
(\mathbf{B}^{(1,1)}(x;\sigma)\mathbf{u})_1=&{\phi_1}_x(1)\phi-{\phi_1}_y(1)\upsilon+y {\phi_1}_y\eta,\\
(\mathbf{B}^{(1,1)}(x;\sigma)\mathbf{u})_2=&\left(3 \sigma^2 {\phi_1}_y(1) - \mu_0 {\phi_1}_{xx}(1) - 2i\mu_0 \sigma {\phi_1}_x(1) \right)\mu_0^{-2}\phi-2 {\phi_1}_y(1)\mu_0^{-1}\phi_y\\
&-{\phi_1}_y(1)\mu_0^{-2}\phi_{yy}+y {\phi_1}_y\mu_0^{-1}\phi_y(1)\\
&+\left({\phi_1}_{xy}(1) - \mu_0 {\phi_1}_x(1)\right)\mu_0^{-1}\upsilon
+\left(y {\phi_1}_{xy} - 2i\sigma y {\phi_1}_y\right)\mu_0^{-1}\eta,\\
(\mathbf{B}^{(1,1)}(x;\sigma)\mathbf{u})_3=&{\eta_1}_x\phi(1)+{\phi_1}_x(1)\eta.
\end{aligned}\notag
\end{align}
Additionally, we calculate that
\begin{align*}\label{eqn:B02}
(&\mathbf{B}^{(0,2)}(x;\sigma)\mathbf{u})_1\\
&=i\sigma\left( {\phi_1}_x(1)^2 + \bar{\phi}_2  +  {\phi_2}_x(1) - {\phi_1}_y(1) {\eta_1}_x\right)\phi\\
&\quad+\Big({\phi_2}_y(1) + {\phi_1}_x(1) {\phi_1}_y(1) - 2 {\phi_1}_y(1) \eta_1 
+ y {\eta_2}_x - y \eta_1 {\eta_1}_x\Big)\phi_y\\
&\quad+i\sigma y {\phi_1}_y {\eta_1}_x\phi(1)+\left(2 y \eta_1 {\phi_1}_y - y {\phi_1}_x(1) {\phi_1}_y - y {\phi_2}_y\right)\phi_y(1)\\
&\begin{aligned}\quad+\Big(&\mu_2 + \mu_0 {\phi_1}_x(1)^2 + \bar{\phi}_2 \mu_0 - {\phi_1}_y(1)^2 + \mu_0 {\phi_2}_x(1) \\ &- i\sigma {\phi_2}_y(1)  - \mu_0 {\phi_1}_y(1) {\eta_1}_x - i\sigma {\phi_1}_x(1) {\phi_1}_y(1) + i\sigma {\phi_1}_y(1) \eta_1 \Big)\upsilon+\mu_0 y {\phi_1}_y {\eta_1}_x\upsilon(1)
\end{aligned}\\
&\quad +\left(i\sigma y {\phi_2}_y - y {\phi_1}_y {\eta_1}_x  + y {\phi_1}_y(1) {\phi_1}_y  +i\sigma y {\phi_1}_x(1) {\phi_1}_y - i\sigma y \eta_1 {\phi_1}_y\right)\eta,\\
\intertext{}
(&\mathbf{B}^{(0,2)}(x;\sigma)\mathbf{u})_2\\
&\begin{aligned}=\Big(&\mu_0^2 \sigma^2 {\phi_1}_x(1)^2 - \mu_0 \mu_2 \sigma^2 - \sigma^4 {\phi_1}_y(1)^2 + \bar{\phi}_2 \mu_0^2 \sigma^2 - i\mu_0^2 \sigma {\phi_2}_{xx}(1) + i\mu_0 \sigma^3 {\phi_2}_y(1)\\
&+ \mu_0 \sigma^2 {\phi_1}_y(1)^2+ \mu_0^2 \sigma^2 {\phi_2}_x(1) + i\mu_0^2 \sigma {\phi_1}_{xy}(1) {\eta_1}_x  - i\mu_0^2 \sigma {\phi_1}_x(1) {\phi_1}_{xx}(1) + \mu_0 \sigma^2 {\phi_1}_y(1) {\phi_1}_{xx}(1) \\
&+ i\mu_0 \sigma^3 {\phi_1}_x(1) {\phi_1}_y(1)  - i\mu_0 \sigma^3 {\phi_1}_y(1) \eta_1  - \mu_0^2 \sigma^2 {\phi_1}_y(1) {\eta_1}_x + i\mu_0^2 \sigma {\phi_1}_y(1) {\eta_1}_{xx} \Big)\mu_0^{-3}\phi \end{aligned}\\
&\begin{aligned}\quad+\Big(&2\sigma^2 {\phi_1}_y(1)^2 - \mu_0 {\phi_2}_{xy}(1) - 2i\mu_0 \sigma {\phi_2}_y(1)  
- \mu_0 {\phi_1}_y(1) {\phi_1}_{xx}(1) + \mu_0 {\phi_1}_y(1) {\eta_1}_x \\
&+ 2 \mu_0 \eta_1 {\phi_1}_{xy}(1) - i\sigma {\phi_1}_y(1) {\phi_1}_{xy}(1) 
-2i \mu_0 \sigma {\phi_1}_x(1) {\phi_1}_y(1)  + 4i\mu_0 \sigma {\phi_1}_y(1) \eta_1 \Big)\mu_0^{-2}\phi_y\end{aligned}\\
&\quad+(\sigma^2 y {\phi_1}_y {\eta_1}_x + i\sigma y {\eta_1}_x {\phi_1}_{xy})\mu_0^{-1}\phi(1)\\
&\begin{aligned}\quad+\Big(&\mu_0^2 {\phi_2}_x(1) - \mu_0 {\phi_1}_y(1)^2 + 2 \mu_0^2 \eta_2 + \mu_0 \mu_2 - \mu_0^2 {\phi_1}_y(1)^2 - 3 \mu_0^2 \eta_1^2 + \sigma^2 {\phi_1}_y(1)^2+ \bar{\phi}_2 \mu_0^2 \\
&- i\mu_0 \sigma {\phi_2}_y(1)  - \mu_0^2 {\phi_1}_y(1) {\eta_1}_x - 2 \mu_0^2 {\phi_1}_x(1) \eta_1 + i\mu_0 \sigma {\phi_1}_x(1) {\phi_1}_y(1)+ 3i\mu_0 \sigma {\phi_1}_y(1) \eta_1 \Big)\mu_0^{-3}\phi_{yy} \end{aligned}\\
&\begin{aligned}\quad +\Big(&\mu_0 {\phi_1}_y(1) {\phi_1}_y - \mu_0 y {\phi_2}_{xy} + \mu_0 y^2 {\eta_1}_x {\phi_1}_{yy} + \mu_0 y {\phi_1}_y {\eta_1}_x \\
&- \sigma^2 y {\phi_1}_y(1) {\phi_1}_y + 2 \mu_0 y \eta_1 {\phi_1}_{xy} + \mu_0 y {\phi_1}_y(1) {\phi_1}_{yy} \\
&- i\sigma y {\phi_1}_y(1) {\phi_1}_{xy}  +i \mu_0 \sigma y {\phi_2}_y 
+  i\mu_0 \sigma y {\phi_1}_x(1) {\phi_1}_y - 2i\mu_0 \sigma y \eta_1 {\phi_1}_y \Big)\mu_0^{-2}\phi_y(1) 
\end{aligned} \\
&\begin{aligned}\quad +\Big(&\mu_0^2 {\phi_1}_y(1){\eta_1}_{xx} - \mu_0^2 {\phi_2}_{xx}(1) 
+ 2 \mu_0 {\phi_1}_y(1) {\phi_1}_{xy}(1) -i \mu_0^2 \sigma {\phi_2}_x(1) \\ 
&+ \mu_0^2 {\phi_1}_{xy}(1) {\eta_1}_x - \mu_0^2 {\phi_1}_x(1) {\phi_1}_{xx}(1) 
- \sigma^2 {\phi_1}_y(1) {\phi_1}_{xy}(1) - i\mu_0^2 \sigma {\phi_1}_x(1)^2 \\ 
&-i \bar{\phi}_2 \mu_0^2 \sigma  + i\mu_0 \sigma {\phi_2}_{xy}(1)  + i\mu_0^2 \sigma {\phi_1}_y(1) {\eta_1}_x 
-i \mu_0 \sigma {\phi_1}_y(1) {\eta_1}_x  - i\mu_0 \sigma \eta_1 {\phi_1}_{xy}(1) \Big)\mu_0^{-2}\upsilon \end{aligned}\\
&\quad +\Big(2 {\phi_1}_y(1) \eta_1 - {\phi_1}_x(1) {\phi_1}_y(1) - {\phi_2}_y(1) + y {\eta_2}_x - y \eta_1 {\eta_1}_x\Big)\upsilon_y\\
&\quad +\left(y {\eta_1}_x {\phi_1}_{xy} - i\sigma y {\phi_1}_y {\eta_1}_x \right)\upsilon(1)\\
&\begin{aligned}\quad+\Big(&2 \mu_0 {\phi_2}_{yy} - 2 \mu_0 {\phi_1}_x(1) {\phi_1}_{yy} 
- 6 \mu_0 \eta_1 {\phi_1}_{yy} + 2i\sigma {\phi_1}_y(1) {\phi_1}_{yy} \\ 
&- \sigma^2 y {\phi_1}_y(1) {\phi_1}_{xy} + \mu_0 \sigma^2 y {\phi_2}_y
 - i\mu_0 \sigma {\phi_1}_y(1) {\phi_1}_y + i\mu_0 \sigma y {\phi_2}_{xy} - \mu_0 y {\eta_1}_x {\phi_1}_{xy} \\
 &+ i\sigma^3 y {\phi_1}_y(1) {\phi_1}_y + \mu_0 y {\phi_1}_y(1) {\phi_1}_{xy} + \mu_0 \sigma^2 y {\phi_1}_x(1) {\phi_1}_y- \mu_0 \sigma^2 y \eta_1 {\phi_1}_y\\
& -i \mu_0 \sigma y \eta_1 {\phi_1}_{xy}  - i\mu_0 \sigma y {\phi_1}_y(1) {\phi_1}_{yy} - i\mu_0 \sigma y^2 {\eta_1}_x {\phi_1}_{yy} 
- i\mu_0 \sigma y {\phi_1}_y(1) {\phi_1}_y \Big)\mu_0^{-2}\eta,\end{aligned}\\
\intertext{and}
(&\mathbf{B}^{(0,2)}(x;\sigma)\mathbf{u})_3\\
&=i\sigma\left( {\eta_2}_x + 2 {\phi_1}_x(1) {\eta_1}_x\right)\phi(1)\\
&\quad+\Big(\eta_2 - {\phi_2}_x(1) - \bar{\phi}_2 + {\phi_1}_x(1) \eta_1 - {\phi_1}_x(1)^2 - \eta_1^2 
+ 2 {\phi_1}_y(1) {\eta_1}_x\Big)\phi_y(1)\\
&\quad+\left(\mu_0 {\eta_2}_x + 2 \mu_0 {\phi_1}_x(1) {\eta_1}_x - i\sigma {\phi_1}_y(1) {\eta_1}_x \right)\upsilon(1)\\
&\quad+\Big(i\sigma {\phi_1}_x(1)^2 + {\phi_2}_y(1)  + {\phi_1}_x(1) {\phi_1}_y(1)  -2 {\phi_1}_y(1) \eta_1  + i\bar{\phi}_2 \sigma + i\sigma {\phi_2}_x(1) - i\sigma {\phi_1}_y(1) {\eta_1}_x\Big)\eta,
\end{align*}
where $\phi_2$ and $\eta_2$ are in \eqref{def:stokes2}. \\
We do not include the formulae of $\mathbf{B}^{(m,n)}(\sigma)$, $m+n\geq3$.

\newpage

\section{Expansion of \texorpdfstring{$\mathbf{w}_{k}^{(m,n)}(x;0)$}{Lg}}\label{A:reduction0}

Recalling \eqref{def:b;exp0}, we solve \eqref{eqn:red0} and \eqref{def:f0}, and make straightforward calculations to show that, to the lowest orders,
\begin{align*}
\mathbf{w}_{1}^{(1,0)}(x;0)=&e^{-i\kappa x}\begin{pmatrix}
b^{(1,0)}_{1,1}y\sinh(\kappa y)+b^{(1,0)}_{1,2}+b^{(1,0)}_{1,3}\cosh(\kappa y)\\
-i\tanh(\kappa) \big(b^{(1,0)}_{1,1}y\sinh(\kappa y)+b^{(1,0)}_{1,2}+b^{(1,0)}_{1,3}\cosh(\kappa y)\big)\\
-i\tanh(\kappa)\big(b^{(1,0)}_{1,1}\sinh(\kappa)+b^{(1,0)}_{1,2}+b^{(1,0)}_{1,3}\cosh(\kappa)\big)
\end{pmatrix},\\
\mathbf{w}_{2}^{(1,0)}(x;0)=&e^{i\kappa x}\begin{pmatrix}
-\big(b^{(1,0)}_{1,1}y\sinh(\kappa y)+b^{(1,0)}_{1,2}+b^{(1,0)}_{1,3}\cosh(\kappa y)\big)\\
-i\tanh(\kappa)\big(b^{(1,0)}_{1,1}y\sinh(\kappa y)+b^{(1,0)}_{1,2}+b^{(1,0)}_{1,3}\cosh(\kappa y)\big)\\
-i\tanh(\kappa)\big(b^{(1,0)}_{1,1}\sinh(\kappa)+b^{(1,0)}_{1,2}+b^{(1,0)}_{1,3}\cosh(\kappa)\big)
\end{pmatrix},
\intertext{and}
\mathbf{w}_{3}^{(1,0)}(x;0)=&\begin{pmatrix}
{\displaystyle -\frac{\kappa c}{s - \kappa c}y^2+\frac{\kappa c(3s - \kappa c)}{3(s - \kappa c)^2}
+\frac{4c}{\kappa(\kappa - cs)}\cosh(\kappa y)}\\ 0\\ 0 \end{pmatrix},
\quad \mathbf{w}_{4}^{(1,0)}(x;0)=\begin{pmatrix} 0\\ 0\\ 0 \end{pmatrix},
\end{align*}
where
\[
b^{(1,0)}_{1,1}=-\frac{2i\kappa c^2}{\kappa - c s},
\quad
b^{(1,0)}_{1,2}=-\frac{2ic^2}{s - \kappa c}\quad\text{and}\quad 
b^{(1,0)}_{1,3}=-\frac{i(2\kappa c^4 + c^3s - 3\kappa c^2)}{(\kappa - cs)^2}.
\]
Throughout the section, we employ the notation
\[
s=\sinh(\kappa)\quad\text{and}\quad c=\cosh(\kappa)
\]
wherever it is convenient to do so.

Let $\mathbf{w}^{(0,1)}_k(x;0)=\begin{pmatrix}
(\mathbf{w}^{(0,1)}_k(x;0))_1\\(\mathbf{w}^{(0,1)}_k(x;0))_2\\(\mathbf{w}^{(0,1)}_k(x;0))_3\end{pmatrix}$, $k=1,2,3,4$, 
and we likewise make straightforward calculations to show that, to the lowest orders,
\begin{align*}
(\mathbf{w}^{(0,1)}_1(x;0))_1=&e^{-2i\kappa x}\big(b^{(0,1)}_{1,1}y\sinh(\kappa y)+b^{(0,1)}_{1,2}+b^{(0,1)}_{1,3}\cosh(\kappa y)+b^{(0,1)}_{1,4}\cosh(2\kappa y)\big),\\
(\mathbf{w}^{(0,1)}_1(x;0))_2=&e^{-2i\kappa x}\big(b^{(0,1)}_{2,1}(y+\tfrac{1}{2})\sinh(\kappa y)+b^{(0,1)}_{2,2}+b^{(0,1)}_{2,3}\cosh(\kappa y)+b^{(0,1)}_{2,4}\cosh(2\kappa y)\big)\\
&+b^{(0,1)}_{2,1}(y-\tfrac{1}{2})\sinh(\kappa y)+b^{(0,1)}_{2,5}+b^{(0,1)}_{2,6}\cosh(\kappa y),\\
(\mathbf{w}^{(0,1)}_1(x;0))_3=&\big[(\mathbf{w}^{(0,1)}_1(x;0))_2\big]_{y=1},
\intertext{and}
(\mathbf{w}^{(0,1)}_2(x;0))_1=&e^{-2i\kappa x}\big(b^{(0,1)}_{1,1}y\sinh(\kappa y)+b^{(0,1)}_{1,2}+b^{(0,1)}_{1,3}\cosh(\kappa y)+b^{(0,1)}_{1,4}\cosh(2\kappa y)\big),\\
(\mathbf{w}^{(0,1)}_2(x;0))_2=&-e^{-2i\kappa x}\big(b^{(0,1)}_{2,1}(y+\tfrac{1}{2})\sinh(\kappa y)+b^{(0,1)}_{2,2}+b^{(0,1)}_{2,3}\cosh(\kappa y)+b^{(0,1)}_{2,4}\cosh(2\kappa y)\big)\\
&-b^{(0,1)}_{2,1}(y-\tfrac{1}{2})\sinh(\kappa y)-b^{(0,1)}_{2,5}-b^{(0,1)}_{2,6}\cosh(\kappa y),\\
(\mathbf{w}^{(0,1)}_2(x;0))_3=&\big[(\mathbf{w}^{(0,1)}_2(x;0))_2\big]_{y=1},
\end{align*}
where
\begin{align*}
&\begin{aligned}&b^{(0,1)}_{1,1}=\kappa^2c,&&
b^{(0,1)}_{1,2}=-\frac{\kappa(4c^2s - \kappa c)}{4 s - 4\kappa c},\\
&b^{(0,1)}_{1,3}=\frac{2\kappa^2c^3 + \kappa^2c + \kappa c^2s}{2(\kappa - c s)},\qquad
&&b^{(0,1)}_{1,4}=\frac{3\kappa^2 c}{4s^3},\\&b^{(0,1)}_{2,1}=-i\kappa^2 s,&&b^{(0,1)}_{2,2}=\frac{i\kappa(\kappa c - \kappa c^3 + 2  c s - 2 c^3 s)}{2 \kappa c s - 2 c^2 + 2},\\
&b^{(0,1)}_{2,3}=\frac{i\kappa(4 \kappa - 3 \kappa c^4 - \kappa c^2 +  c s -  c^3 s)}{2 s (\kappa - c s)},&&b^{(0,1)}_{2,4}=-\frac{3i\kappa^2 }{2 s^2},\\
&b^{(0,1)}_{2,5}=\frac{i\kappa(\kappa c^2 + \kappa c^3 - \kappa - \kappa c + 2  c s - 2  c^3 s)}{2 \kappa c s - 2 c^2 + 2},&&b^{(0,1)}_{2,6}=\frac{i\kappa(2 \kappa s + c -c^3 - \kappa c^2 s)}{2 \kappa - 2 c s}.&&
\end{aligned}
\end{align*}
Additionally, we calculate that, to the lowest orders,
\begin{align*}
(\mathbf{w}^{(0,1)}_3(x;0))_1=&\sin(\kappa x)\big(b^{(0,1)}_{3,1}y\sinh(\kappa y)+b^{(0,1)}_{3,2}+b^{(0,1)}_{3,3}\cosh(\kappa y)\big),\\
(\mathbf{w}^{(0,1)}_3(x;0))_2=&\tanh(\kappa)\cos(\kappa x)\big(b^{(0,1)}_{3,1}y\sinh(\kappa y)+b^{(0,1)}_{3,2}+b^{(0,1)}_{3,3}\cosh(\kappa y)\big),\\
(\mathbf{w}^{(0,1)}_3(x;0))_3=&\big[(\mathbf{w}^{(0,1)}_3(x;0))_2\big]_{y=1},
\end{align*}
where
\begin{align*}
&b^{(0,1)}_{3,1}=\frac{- 2 \kappa^2 c^2 - \kappa c s}{\kappa - c s},\qquad
b^{(0,1)}_{3,2}=\frac{-c (s + 2 \kappa c)}{s - \kappa c},\\
&b^{(0,1)}_{3,3}=\frac{c^2 - c^4 + 6 \kappa^2 c^2 - 4 \kappa^2 c^4 + 3 \kappa c s - 4 \kappa c^3 s}{2 \kappa^2 + 2 c^2 s^2 - 4 \kappa c s},
\end{align*}
and $\mathbf{w}^{(0,1)}_4(x;0)=\mathbf{0}$.\\
We do not include the formulae of $\mathbf{w}^{(m,n)}_k(x;0)$, $k=1,2,3,4$, and $m+n\geq2$.

\section{Expansion of \texorpdfstring{$\mathbf{w}_{k}^{(m,n)}(x;\sigma)$}{Lg}}\label{A:reduction;high}
Throughout the section, $\sigma>\sigma_c$. Inserting \eqref{def:a;expH}, \eqref{def:b;expH} and \eqref{def:B;expH} into the latter equation of \eqref{eqn:LB;u12}, at the order of $\delta^m\eps^n$, $m+n\geq 1$, let $\mathbf{w}_{k}^{(m,n)}(x;\sigma)=\begin{pmatrix}\phi\\ \upsilon \\ \eta\end{pmatrix}$, by abuse of notation, and we arrive at
\ba \label{eqn:redH}
&\begin{aligned}\phi_{xx}+\phi_{yy}=&i\sigma((\mathbf{1}-\boldsymbol{\Pi}(\sigma))\mathbf{f}_{k}^{(m,n)}(x;\sigma))_1+{((\mathbf{1}-\boldsymbol{\Pi}(\sigma))\mathbf{f}_{k}^{(m,n)}(x;\sigma))_1}_x\\
&+\mu_0((\mathbf{1}-\boldsymbol{\Pi}(\sigma))\mathbf{f}_{k}^{(m,n)}(x;\sigma))_2\end{aligned}&&\text{for $0<y<1$},\\
&\upsilon=\mu_0^{-1}\phi_x-i\sigma\mu_0^{-1}\phi-\mu_0^{-1}((\mathbf{1}-\boldsymbol{\Pi}(\sigma))\mathbf{f}_{k}^{(m,n)}(x;\sigma))_1&&\text{for $0<y<1$},\\
&\eta_x=-\phi_y+i\sigma\eta+((\mathbf{1}-\boldsymbol{\Pi}(\sigma))\mathbf{f}_{k}^{(m,n)}(x;\sigma))_3&&\text{at $y=1$},\\
&\eta=\upsilon&&\text{at $y=1$}, \\
&\phi_y=0&&\text{at $y=0$}.
\ea
We solve the first and the last equations of \eqref{eqn:redH}, for instance, by the method of undetermined coefficients, subject to that $\mathbf{w}_{k}^{(m,n)}(x;\sigma)\in(\mathbf{1}-\boldsymbol{\Pi}(\sigma))Y$, so that $\boldsymbol{\Pi}(\sigma)\mathbf{w}_{k}^{(m,n)}(x;\sigma)=\mathbf{0}$, to determine $\phi$, and we determine $\upsilon$ and $\eta$ by the second and fourth equations of \eqref{eqn:redH}. 

A straightforward calculation reveals that, to the lowest orders,
\begin{equation}\label{def:w(0,1)1}
\begin{aligned}
&\begin{aligned}
(\mathbf{w}^{(0,1)}_1(x;\sigma))_1=e^{ik_2x}\big(&
\big(b^{(0,1)}_{1,1} \sin(\kappa x)+b^{(0,1)}_{1,2} \cos(\kappa x)\big)\cosh(k_2 y) \\
&+ b^{(0,1)}_{1,3}\sin(\kappa x)\,y\sinh(\kappa y) + k_2\kappa c\cos(\kappa x)\,y\sinh(k_2 y)\\
&+\big(b^{(0,1)}_{1,4} \sin(\kappa x)+ b^{(0,1)}_{1,5} \cos(\kappa x)\big)\cosh(k_4 y) \\
&+ b^{(0,1)}_{1,6} e^{i\kappa x} \cosh((k_2 + \kappa)y)
+ b^{(0,1)}_{1,7} e^{-i\kappa x} \cosh((k_2 - \kappa)y)\big),
\end{aligned}\\
&\begin{aligned}
(\mathbf{w}^{(0,1)}_1(x;\sigma))_2=e^{ik_2x} \big(
& (b^{(0,1)}_{1,8} \sin(\kappa x)+b^{(0,1)}_{1,9} \cos(\kappa x))\cosh(k_2 y) \\
& +\big(b^{(0,1)}_{1,10} \sin(\kappa x)+ b^{(0,1)}_{1,11} \cos(\kappa x)\big)\cosh(k_4 y) \\
&+ \big(b^{(0,1)}_{1,12}\sin(\kappa x) +\tanh(\kappa)b^{(0,1)}_{1,3}\cos(\kappa x)\big)y \sinh(\kappa y)\\
&-k_2\kappa s\sin(\kappa x)\sinh(k_2y)+  ik_2 s(k_2 - \sigma) \cos(\kappa x)\,y\sinh(k_2 y) \\
&+ib^{(0,1)}_{1,6}(k_2 + \kappa - \sigma)\mu_0^{-1}e^{i\kappa x}\cosh((k_2 + \kappa)y)\\
&+ib^{(0,1)}_{1,7}(k_2- \kappa - \sigma)\mu_0^{-1}e^{-i\kappa x}\cosh((k_2 - \kappa)y)\big),
\end{aligned}\\
&(\mathbf{w}^{(0,1)}_1(x;\sigma))_3=\big[(\mathbf{w}^{(0,1)}_1)_2\big]_{y=1},
\end{aligned}
\end{equation}
where
\begin{align*}
b^{(0,1)}_{1,1}=&\frac{- 4i k_2^2 (k_2^2 c_2 s  + \sigma^2 c_2 s  - k_2 \kappa c s_2  - 2k_2 \sigma c_2 s )}{\kappa s_2 (4 k_2^2 - \kappa^2)} - \frac{p^{(0,1)}_{1,2} c_2 (k_2 - \sigma)^2}{k_2 \kappa s_2},\\
b^{(0,1)}_{1,2}=&\frac{2 k_2 (k_2^2 c_2 s + \sigma^2 c_2 s - k_2 \kappa c s_2 - 2 k_2 \sigma c_2 s)}{s_2 (4 k_2^2 - \kappa^2)}+ \frac{p^{(0,1)}_{1,1} c_2 (k_2 - \sigma)^2}{k_2 \kappa s_2},\\
b^{(0,1)}_{1,3}=&i\kappa c_2 (k_2-\sigma), \\
b^{(0,1)}_{1,4}=&\frac{ic_4 (k_4 - \sigma)^2 ((k_4 - k_2) p^{(0,1)}_{1,3} + i\kappa p^{(0,1)}_{1,4} )}{k_4 s_4 (2 k_2 k_4 - k_2^2 - k_4^2 + \kappa^2)},\\
b^{(0,1)}_{1,5}=&\frac{-ic_4 (k_4 - \sigma)^2 ((k_2 - k_4) p^{(0,1)}_{1,4} + i\kappa p^{(0,1)}_{1,3} )}{k_4 s_4 (2 k_2 k_4 - k_2^2 - k_4^2 + \kappa^2)},
\intertext{and}
b^{(0,1)}_{1,8}=&i(k_2 - \sigma)\mu_0^{-1}b^{(0,1)}_{1,1}-\kappa \mu_0^{-1}b^{(0,1)}_{1,2}+p^{(0,1)}_{1,1}-\sigma c^{-1}(c^2 - 1) (k_2 - \sigma),\\
b^{(0,1)}_{1,9}=&\kappa \mu_0^{-1}b^{(0,1)}_{1,1}+i(k_2 - \sigma)\mu_0^{-1}b^{(0,1)}_{1,2}+p^{(0,1)}_{1,2}-ik_2\kappa c,\\
b^{(0,1)}_{1,10}=&i(k_2 - \sigma)\mu_0^{-1}b^{(0,1)}_{1,4}-\kappa \mu_0^{-1}b^{(0,1)}_{1,5}+p^{(0,1)}_{1,3}, \\
b^{(0,1)}_{1,11}=&\kappa\mu_0^{-1} b^{(0,1)}_{1,4}+i(k_2 - \sigma)\mu_0^{-1}b^{(0,1)}_{1,5}+p^{(0,1)}_{1,4}.
\end{align*}
Here and elsewhere, we employ the notation
\begin{align*}
&s=\sinh(\kappa), & &s_2=\sinh(k_2),& &s_4=\sinh(k_4), &&s(1)=\sinh(1),\\
&c=\cosh(\kappa),& &c_2=\cosh(k_2), & &c_4=\cosh(k_4),& &s_2(2)=\sinh(2k_2)
\end{align*}
whenever it is convenient to do so.
Also
\begin{align*}
\begin{aligned}
b^{(0,1)}_{1,6}=\Big(&2 k_2^3 \kappa^3 c_2^2 s_2 + 6 k_2^4 \kappa^2 c_2^2 s_2 - 4 k_2^6 c_2^3 c s + 2 \kappa^2 \sigma^4 c_2^2 s_2 - 2 \kappa^3 \sigma^3 c_2^2 s_2 - 2 k_2^3 \kappa^2 c_2^3 c^2 + 4 k_2^4 \kappa c_2 c^2\\& + 4 k_2^5 \kappa c_2^2 s_2 + 4 k_2^6 c_2 c s + 2 k_2^3 \kappa^2 c_2 c^2 - 4 k_2^4 \kappa c_2^3 c^2 + 10 k_2^5 \kappa c_2 c s + 4 k_2 \kappa \sigma^4 c_2^2 s_2 - 16 k_2^4 \kappa \sigma c_2^2 s_2 \\&- 16 k_2^5 \sigma c_2 c s + 2 k_2^3 \kappa^3 c_2^2 c^2 s_2 + 8 k_2^4 \kappa^2 c_2^2 c^2 s_2 - 2 \kappa^2 \sigma^4 c_2^2 c^2 s_2 - 2 k_2 \kappa^2 \sigma^2 c_2 c^2 - 4 k_2^2 \kappa \sigma^2 c_2 c^2 \\&+ 4 k_2^2 \kappa^2 \sigma c_2 c^2 + 2 k_2^3 \kappa^3 c_2 c s + 8 k_2^4 \kappa^2 c_2 c s - 6 k_2^5 \kappa c_2^3 c s - 12 k_2 \kappa^2 \sigma^3 c_2^2 s_2 + 6 k_2 \kappa^3 \sigma^2 c_2^2 s_2 \\&- 16 k_2^2 \kappa \sigma^3 c_2^2 s_2 - 6 k_2^2 \kappa^3 \sigma c_2^2 s_2 + 24 k_2^3 \kappa \sigma^2 c_2^2 s_2 - 20 k_2^3 \kappa^2 \sigma c_2^2 s_2 + 4 k_2^2 \sigma^4 c_2 c s - 16 k_2^3 \sigma^3 c_2 c s \\&+ 24 k_2^4 \sigma^2 c_2 c s + 16 k_2^5 \sigma c_2^3 c s + 2 k_2 \kappa^2 \sigma^2 c_2^3 c^2 + 4 k_2^2 \kappa \sigma^2 c_2^3 c^2 - 4 k_2^2 \kappa^2 \sigma c_2^3 c^2 + 8 k_2^5 \kappa c_2^2 c^2 s_2 \\&- 2 k_2^3 \kappa^3 c_2^3 c s - 6 k_2^4 \kappa^2 c_2^3 c s + 24 k_2^2 \kappa^2 \sigma^2 c_2^2 s_2 - 4 k_2^2 \sigma^4 c_2^3 c s + 16 k_2^3 \sigma^3 c_2^3 c s - 24 k_2^4 \sigma^2 c_2^3 c s \\&+ 2 \kappa^2 \sigma^4 c_2^3 c s + 4 k_2^5 c_2^2 c s_2 s + 6 k_2^2 \kappa^2 \sigma^2 c_2^2 c^2 s_2 - 4 k_2 \kappa \sigma^4 c_2^2 c^2 s_2 - 16 k_2^4 \kappa \sigma c_2^2 c^2 s_2 + 16 k_2^2 \kappa^2 \sigma^2 c_2 c s\\& - 4 k_2 \kappa^2 \sigma^3 c_2^3 c s - 2 k_2 \kappa^3 \sigma^2 c_2^3 c s + 4 k_2^2 \kappa^3 \sigma c_2^3 c s - 12 k_2^3 \kappa \sigma^2 c_2^3 c s + 12 k_2^3 \kappa^2 \sigma c_2^3 c s + 2 k_2^4 \kappa c_2^2 c s_2 s\\& - 4 k_2 \sigma^4 c_2^2 c s_2 s - 8 k_2^4 \sigma c_2^2 c s_2 s - 2 \kappa \sigma^4 c_2^2 c s_2 s + 2 k_2 \kappa \sigma^4 c_2 c s - 32 k_2^4 \kappa \sigma c_2 c s + 4 k_2 \kappa^2 \sigma^3 c_2^2 c^2 s_2 \\&+ 2 k_2 \kappa^3 \sigma^2 c_2^2 c^2 s_2 + 8 k_2^2 \kappa \sigma^3 c_2^2 c^2 s_2 - 4 k_2^2 \kappa^3 \sigma c_2^2 c^2 s_2 + 4 k_2^3 \kappa \sigma^2 c_2^2 c^2 s_2 - 16 k_2^3 \kappa^2 \sigma c_2^2 c^2 s_2\\& - 4 k_2^2 \kappa^2 \sigma^2 c_2^3 c s + 8 k_2^2 \sigma^3 c_2^2 c s_2 s - 4 k_2 \kappa^2 \sigma^3 c_2 c s + 2 k_2 \kappa^3 \sigma^2 c_2 c s - 16 k_2^2 \kappa \sigma^3 c_2 c s - 4 k_2^2 \kappa^3 \sigma c_2 c s \\&+ 36 k_2^3 \kappa \sigma^2 c_2 c s - 20 k_2^3 \kappa^2 \sigma c_2 c s + 2 k_2 \kappa \sigma^4 c_2^3 c s + 16 k_2^4 \kappa \sigma c_2^3 c s - 8 k_2^2 \kappa \sigma^2 c_2^2 c s_2 s + 8 k_2 \kappa \sigma^3 c_2^2 c s_2 s\Big) \\ 
\cdot\Big(&12 k_2^2 \kappa^2 c_2^3 c^2 - 4 \kappa^2 \sigma^2 c_2^3 c^2 - 4 k_2 \kappa^3 c_2 c^2 - 8 k_2^3 \kappa c_2 c^2 - 8 k_2^4 c s_2 s \\&- 12 k_2^2 \kappa^2 c_2 c^2 + 4 k_2 \kappa^3 c_2^3 c^2 + 8 k_2^3 \kappa c_2^3 c^2 + 4 \kappa^2 \sigma^2 c_2 c^2 + 8 k_2 \kappa \sigma^2 c_2 c^2 - 4 k_2 \kappa^3 c s_2 s - 20 k_2^3 \kappa c s_2 s \\&+ 16 k_2^3 \sigma c s_2 s - 8 k_2 \kappa \sigma^2 c_2^3 c^2 - 16 k_2^2 \kappa^2 c s_2 s - 8 k_2^2 \sigma^2 c s_2 s + 4 k_2 \kappa^3 c_2^2 c s_2 s + 8 k_2^3 \kappa c_2^2 c s_2 s \\&- 4 k_2 \kappa \sigma^2 c s_2 s + 8 k_2 \kappa^2 \sigma c s_2 s + 24 k_2^2 \kappa \sigma c s_2 s + 12 k_2^2 \kappa^2 c_2^2 c s_2 s - 4 \kappa^2 \sigma^2 c_2^2 c s_2 s - 8 k_2 \kappa \sigma^2 c_2^2 c s_2 s\Big)^{-1}
\intertext{and}
b^{(0,1)}_{1,7}=\Big(&6 k_2^4 \kappa^2 c_2^2 s_2 - 2 k_2^3 \kappa^3 c_2^2 s_2 + 4 k_2^6 c_2^3 c s + 2 \kappa^2 \sigma^4 c_2^2 s_2 + 2 \kappa^3 \sigma^3 c_2^2 s_2 - 2 k_2^3 \kappa^2 c_2^3 c^2 - 4 k_2^4 \kappa c_2 c^2\\& - 4 k_2^5 \kappa c_2^2 s_2 - 4 k_2^6 c_2 c s + 2 k_2^3 \kappa^2 c_2 c^2 + 4 k_2^4 \kappa c_2^3 c^2 + 10 k_2^5 \kappa c_2 c s - 4 k_2 \kappa \sigma^4 c_2^2 s_2 + 16 k_2^4 \kappa \sigma c_2^2 s_2 \\&+ 16 k_2^5 \sigma c_2 c s - 2 k_2^3 \kappa^3 c_2^2 c^2 s_2 + 8 k_2^4 \kappa^2 c_2^2 c^2 s_2 - 2 \kappa^2 \sigma^4 c_2^2 c^2 s_2 - 2 k_2 \kappa^2 \sigma^2 c_2 c^2 + 4 k_2^2 \kappa \sigma^2 c_2 c^2\\& + 4 k_2^2 \kappa^2 \sigma c_2 c^2 + 2 k_2^3 \kappa^3 c_2 c s - 8 k_2^4 \kappa^2 c_2 c s - 6 k_2^5 \kappa c_2^3 c s - 12 k_2 \kappa^2 \sigma^3 c_2^2 s_2 - 6 k_2 \kappa^3 \sigma^2 c_2^2 s_2\\& + 16 k_2^2 \kappa \sigma^3 c_2^2 s_2 + 6 k_2^2 \kappa^3 \sigma c_2^2 s_2 - 24 k_2^3 \kappa \sigma^2 c_2^2 s_2 - 20 k_2^3 \kappa^2 \sigma c_2^2 s_2 - 4 k_2^2 \sigma^4 c_2 c s + 16 k_2^3 \sigma^3 c_2 c s \\&- 24 k_2^4 \sigma^2 c_2 c s - 16 k_2^5 \sigma c_2^3 c s + 2 k_2 \kappa^2 \sigma^2 c_2^3 c^2 - 4 k_2^2 \kappa \sigma^2 c_2^3 c^2 - 4 k_2^2 \kappa^2 \sigma c_2^3 c^2 - 8 k_2^5 \kappa c_2^2 c^2 s_2\\& - 2 k_2^3 \kappa^3 c_2^3 c s + 6 k_2^4 \kappa^2 c_2^3 c s + 24 k_2^2 \kappa^2 \sigma^2 c_2^2 s_2 + 4 k_2^2 \sigma^4 c_2^3 c s - 16 k_2^3 \sigma^3 c_2^3 c s + 24 k_2^4 \sigma^2 c_2^3 c s\\& - 2 \kappa^2 \sigma^4 c_2^3 c s - 4 k_2^5 c_2^2 c s_2 s + 6 k_2^2 \kappa^2 \sigma^2 c_2^2 c^2 s_2 + 4 k_2 \kappa \sigma^4 c_2^2 c^2 s_2 + 16 k_2^4 \kappa \sigma c_2^2 c^2 s_2 - 16 k_2^2 \kappa^2 \sigma^2 c_2 c s\\& + 4 k_2 \kappa^2 \sigma^3 c_2^3 c s - 2 k_2 \kappa^3 \sigma^2 c_2^3 c s + 4 k_2^2 \kappa^3 \sigma c_2^3 c s - 12 k_2^3 \kappa \sigma^2 c_2^3 c s - 12 k_2^3 \kappa^2 \sigma c_2^3 c s + 2 k_2^4 \kappa c_2^2 c s_2 s \\&+ 4 k_2 \sigma^4 c_2^2 c s_2 s + 8 k_2^4 \sigma c_2^2 c s_2 s - 2 \kappa \sigma^4 c_2^2 c s_2 s + 2 k_2 \kappa \sigma^4 c_2 c s - 32 k_2^4 \kappa \sigma c_2 c s + 4 k_2 \kappa^2 \sigma^3 c_2^2 c^2 s_2\\& - 2 k_2 \kappa^3 \sigma^2 c_2^2 c^2 s_2 - 8 k_2^2 \kappa \sigma^3 c_2^2 c^2 s_2 + 4 k_2^2 \kappa^3 \sigma c_2^2 c^2 s_2 - 4 k_2^3 \kappa \sigma^2 c_2^2 c^2 s_2 - 16 k_2^3 \kappa^2 \sigma c_2^2 c^2 s_2 \\&+ 4 k_2^2 \kappa^2 \sigma^2 c_2^3 c s - 8 k_2^2 \sigma^3 c_2^2 c s_2 s + 4 k_2 \kappa^2 \sigma^3 c_2 c s + 2 k_2 \kappa^3 \sigma^2 c_2 c s - 16 k_2^2 \kappa \sigma^3 c_2 c s \\&- 4 k_2^2 \kappa^3 \sigma c_2 c s + 36 k_2^3 \kappa \sigma^2 c_2 c s + 20 k_2^3 \kappa^2 \sigma c_2 c s \\ &+ 2 k_2 \kappa \sigma^4 c_2^3 c s + 16 k_2^4 \kappa \sigma c_2^3 c s - 8 k_2^2 \kappa \sigma^2 c_2^2 c s_2 s + 8 k_2 \kappa \sigma^3 c_2^2 c s_2 s\Big) 
\end{aligned}
\end{align*}

\begin{align*}
\begin{aligned}
\cdot\Big(&4 \kappa^2 \sigma^2 c_2^3 c^2 - 12 k_2^2 \kappa^2 c_2^3 c^2 - 4 k_2 \kappa^3 c_2 c^2 - 8 k_2^3 \kappa c_2 c^2 - 8 k_2^4 c s_2 s \\&+ 12 k_2^2 \kappa^2 c_2 c^2 + 4 k_2 \kappa^3 c_2^3 c^2 + 8 k_2^3 \kappa c_2^3 c^2 - 4 \kappa^2 \sigma^2 c_2 c^2 + 8 k_2 \kappa \sigma^2 c_2 c^2 + 4 k_2 \kappa^3 c s_2 s + 20 k_2^3 \kappa c s_2 s \\&+ 16 k_2^3 \sigma c s_2 s - 8 k_2 \kappa \sigma^2 c_2^3 c^2 - 16 k_2^2 \kappa^2 c s_2 s - 8 k_2^2 \sigma^2 c s_2 s - 4 k_2 \kappa^3 c_2^2 c s_2 s - 8 k_2^3 \kappa c_2^2 c s_2 s \\
&+ 4 k_2 \kappa \sigma^2 c s_2 s + 8 k_2 \kappa^2 \sigma c s_2 s - 24 k_2^2 \kappa \sigma c s_2 s + 12 k_2^2 \kappa^2 c_2^2 c s_2 s - 4 \kappa^2 \sigma^2 c_2^2 c s_2 s + 8 k_2 \kappa \sigma^2 c_2^2 c s_2 s\Big)^{-1},
\end{aligned}
\end{align*}
where
\begin{align*}
&\begin{aligned}
p^{(0,1)}_1=-(&k_2^6 \kappa s + 4 k_2^2 \kappa^4 c + k_2^2 \kappa^6 c - 2 k_2^4 \kappa^4 c + k_2^6 \kappa^2 c - k_2^2 \kappa^5 s + 2 k_2^2 \kappa^5 c_2^2 s - 2 k_2^4 \kappa^3 c_2^2 s\\& + 2 k_2^3 \kappa^4 \sigma c - k_2^5 \kappa^2 \sigma c + 2 k_2^3 \kappa^3 \sigma s - k_2 \kappa^6 \sigma c - k_2 \kappa^5 \sigma s - k_2^5 \kappa \sigma s - 4 k_2^2 \kappa^4 c_2^2 c + k_2 \kappa^6 c_2 c s_2 \\&- \kappa^6 \sigma c_2 c s_2 + 3 k_2 \kappa^5 c_2 s_2 s + k_2^5 \kappa c_2 s_2 s - \kappa^5 \sigma c_2 s_2 s - 4 k_2^3 \kappa^4 c_2 c s_2 \\  &+ 3 k_2^5 \kappa^2 c_2 c s_2 - k_2^4 \kappa \sigma c_2 s_2 s + 2 k_2^2 \kappa^4 \sigma c_2 c s_2 - k_2^4 \kappa^2 \sigma c_2 c s_2 + 2 k_2^2 \kappa^3 \sigma c_2 s_2 s)\\
&\cdot((k_2^2 - \kappa^2)^2 (k_2 s_2(2) + \sigma s_2(2) + 2 k_2 \sigma - 2 k_2^2))^{-1},
\end{aligned}\\
&\begin{aligned}
p^{(0,1)}_2=-ik_2 (&\kappa^6 c_2^3 s - 2 k_2^6 c_2 s - \kappa^6 c_2 s - 4 k_2^2 \kappa^4 c_2^3 s + 3 k_2^4 \kappa^2 c_2^3 s + 2 k_2^5 \sigma c_2 s + \kappa^6 \sigma s_2 s - 4 k_2^2 \kappa^3 c_2 c\\& + 2 k_2^2 \kappa^4 c_2 s + k_2^4 \kappa^2 c_2 s + 4 k_2^2 \kappa^3 c_2^3 c - 4 k_2^3 \kappa^2 \sigma c_2 s - 2 k_2^2 \kappa^4 \sigma s_2 s + k_2^4 \kappa^2 \sigma s_2 s \\ &- 2 k_2 \kappa^5 c_2^2 c s_2 - 4 k_2^5 \kappa c_2^2 c s_2 - 2 k_2 \kappa^4 c_2^2 s_2 s + 2 k_2 \kappa^4 \sigma c_2 s + 6 k_2^3 \kappa^3 c_2^2 c s_2 - 2 k_2^3 \kappa^2 c_2^2 s_2 s) \\
&\cdot (c_2 (k_2^2 - \kappa^2)^2(k_2 s_2(2) + \sigma s_2(2) + 2 k_2 \sigma - 2 k_2^2))^{-1},
\end{aligned}\\
&\begin{aligned}
p^{(0,1)}_3=p_{2,2}\bigg(&\frac{k_2 \kappa c_2 (k_2 - \sigma) (k_4^2 c_4 s + \kappa^2 c_4 s + k_4 \kappa^3 c s_4 + k_4^3 \kappa c s_4 - 2 k_4^2 \kappa^2 c_4 s - 2 k_4 \kappa c s_4)}{(k_4^2 - \kappa^2)^2}\\
&+\frac{k_2^2 \kappa \sigma c_4 s_2 (k_4^2 - 1) (k_4 - \sigma)^2 (k_4^2 c_4 s + \kappa^2 c_4 s - 2 k_4 \kappa c s_4)}{k_4 s_4(k_4^2 - \kappa^2)^2(k_4 + \sigma)(k_2 - \sigma)}\\
&+\frac{k_2 \kappa c_2 (k_2 - \sigma) (\kappa c_4 c - k_4 s_4 s - k_4^2 c_4 s - k_4^2 \kappa c_4 c + k_4 \kappa^2 s_4 s + k_4 \kappa c s_4)}{k_4^2 - \kappa^2} \\
&+\frac{k_2 \kappa c_4 s_2 (k_4^2 - 1) (k_4 - \sigma)^2 (2 \kappa^2 c_4 s - 2 k_4 \kappa c s_4 + k_2 \kappa \sigma c_4 c - k_2 k_4 \sigma s_4 s)}{k_4 s_4(k_4 + \sigma)(k_2 - \sigma)(k_4^2 - \kappa^2)}\\
&+\frac{\kappa c_2 s (k_2 - \sigma)^2 (k_2^2 c_4 s_2 + k_4^2 c_4 s_2 + k_2 k_4^3 c_2 s_4 + k_2^3 k_4 c_2 s_4 - 2 k_2^2 k_4^2 c_4 s_2 - 2 k_2 k_4 c_2 s_4)}{s_2(k_2^2 - k_4^2)^2} \\
&-\frac{k_2 \kappa c_4 s (k_4^2 - 1) (k_2 - \sigma)^2 (k_4 - \sigma)^2 (k_2^2 c_4 s_2 + k_4^2 c_4 s_2 - 2 k_2 k_4 c_2 s_4)}{k_4 s_4(k_4 + \sigma)(k_2 - \sigma)(k_2^2 - k_4^2)^2}\\
&\begin{aligned}-\frac{\kappa s (k_2 - \sigma)}{s_2(k_2^2 - k_4^2)} (&k_2^2 c_2 - k_2 \sigma c_2 - k_2^2 k_4^2 c_2- k_2 \sigma c_4 s_2^2 + k_2^2 k_4 c_2^2 s_4 + k_2 k_4^2 \sigma c_2  - k_2 k_4^2 c_2 c_4 s_2\\
& + k_2 k_4^2 \sigma c_4 s_2^2 + k_4^2 \sigma c_2 c_4 s_2 - k_2 k_4 \sigma c_2^2 s_4 + k_4 \sigma c_2 s_2 s_4 - k_2^2 k_4 \sigma c_2 s_2 s_4)\end{aligned}\\
&\begin{aligned}+&\frac{k_2 \kappa c_4 (k_4^2 - 1) (k_4 - \sigma)^2}{ k_4 c_2 s_4(k_4 + \sigma)(k_2 - \sigma)(k_2^2 - k_4^2)}\\ 
&\cdot(2 k_2^3 c_2^2 c_4 s - k_2^3 c_2 s + k_2 \sigma^2 c_2 s - 2 k_2^2 \sigma c_2^2 c_4 s 
+ k_2 \sigma^2 c_4 s_2^2 s + k_2^2 \sigma c_4 s_2^2 s - k_4 \sigma^2 c_2 s_2 s_4 s  \\ 
&\quad+ k_2 k_4 \kappa c_2^2 c s_4 - k_4 \kappa \sigma c_2^2 c s_4 - k_2^2 \kappa c_2 c_4 c s_2 - 2 k_2^2 k_4 c_2 s_2 s_4 s + k_2 k_4 \sigma c_2 s_2 s_4 s + k_2 \kappa \sigma c_2 c_4 c s_2) \end{aligned}\\
&+\frac{\kappa c_2 c_4 s (k_4^2 - 1) (k_2 - \sigma) (s_2 - k_2 c_2 + \sigma c_2)}{s_2 (k_4 + \sigma)}\bigg)
- p_{1,2}\kappa s(1) c_2 s (k_2 - \sigma)^2,
\end{aligned}
\end{align*}
and
\begin{align*}
&\begin{aligned}
p^{(0,1)}_{4}=ip_{2,2}\bigg(&
\frac{k_2^2 \kappa^2 c_4 s_2 (k_4^2 - 1) (k_4 - \sigma)^2 (k_4^2 c_4 s + \kappa^2 c_4 s - 2 k_4 \kappa c s_4)}{k_4 s_4(k_4 + \sigma)(k_2 - \sigma)(k_4^2 - \kappa^2)^2}\\
&+\frac{k_2^2 \kappa^2 c_4 s_2 (k_4^2 - 1) (\kappa c_4 c - k_4 s_4 s) (k_4 - \sigma)^2}{k_4 s_4(k_4 + \sigma)(k_2 - \sigma)(k_4^2 - \kappa^2)}\\
&+\frac{\kappa c_2 c (k_2 - \sigma)^2(k_2 c_4 s_2 - k_4 c_2 s_4 - k_2 k_4^2 c_4 s_2 + k_2^2 k_4 c_2 s_4)}{s_2(k_2^2 - k_4^2)}\\
&\begin{aligned}+&\frac{k_2 c_4 (k_4^2 - 1) (k_4 - \sigma)^2}{k_4 c_2 s_4(k_4 + \sigma)(k_2 - \sigma)(k_2^2 - k_4^2)} \\
& \cdot (k_2^2 \kappa^2 c_2 s - 2 k_2^2 k_4 c_2^2 s_4 s + 2 k_2^3 c_2 c_4 s_2 s - k_2 \kappa^2 \sigma c_2 s  \\ &+ k_2 \kappa^2 \sigma c_4 s - k_2^2 \kappa^2 c_2^2 c_4 s + 2 k_2 k_4 \sigma c_2^2 s_4 s 
- k_2^2 k_4 \kappa c_2^2 c s_4  \\ &+ k_2^3 \kappa c_2 c_4 c s_2 + k_4 \kappa \sigma^2 c_2^2 c s_4 - 2 k_2^2 \sigma c_2 c_4 s_2 s - k_2 \kappa \sigma^2 c_2 c_4 c s_2 + k_2 k_4 \kappa^2 c_2 s_2 s_4 s)\end{aligned}\\
&+\frac{c_2 c_4 (k_4^2 - 1) (k_2 - \sigma)}{k_4 + \sigma}(\sigma s - k_2 s + k_2 \kappa c)\bigg)
+i p_{1,2}\kappa  s_2^{-1}s(1) c_2^2 c (k_2 - \sigma)^2.
\end{aligned}
\end{align*}
We do not include the formulae of $\mathbf{w}^{(0,1)}_2(x;\sigma)$, $\mathbf{w}^{(1,0)}_{1,2}(x;\sigma)$ and $\mathbf{w}^{(m,n)}_k(x;\sigma)$, for $k=1,2$, and $m+n\geq2$.

\section{Rigorous enclosure of solution to a nonlinear equation}\label{rigor_enclosure}

We explain how to apply a Newton-Kantorovich theorem (see, for instance, \cite{Castelli2018,newton2,CHURCH2022133072} for applications to computer-assisted proofs) to give an a posteriori rigorous error bound on the numerical approximation of the solution of \eqref{F_func}.

\begin{theorem}\label{thm:NK}
Let $X$ and $Y$ denote Banach spaces and $F:X\to Y$ be continuously differentiable. Let $x_0\in X$. Let $A^\dagger \in \mathcal{B}(X,Y)$, a bounded linear operator, and $A\in \mathcal{B}(Y,X)$ be injective. Suppose that 
\begin{enumerate}
\item $\|AF(x_0)\|_{X}\leq r_1$ for some $r_1>0$;
\item $\|\Id-AA^\dagger\|_{\mathcal{B}(X,X)}\leq r_2$ for some $r_2>0$;
\item $\|A(A^\dagger-DF(x_0))\|_{\mathcal{B}(X,X)}\leq r_3$ for some $r_3>0$, where $DF$ denotes the Fr\'echet derivative of $F$; and moreover
\item $\sup_{x\in \overline{B_r(x_0)}}\|A(DF(x)-DF(x_0))\|_{\mathcal{B}(X,X)}\leq b(r)r$ for $r>0$ for some function $b(r)>0$.
\end{enumerate}
If 
\[
b(r)r^2-(1-r_2-r_3)r+r_1<0\quad\text{for some $r>0$,}
\]
then there exists a unique $x_*\in B_r(x_0)$ such that $F(x_*)=0$.
\end{theorem}

For $\kappa>0$ suitably represented by an interval, let $x_0=(k_4,\sigma)$ make an initial guess of a solution of $F$ \eqref{F_func}, where $k_4$ and $\sigma$ are represented by intervals. We apply Theorem~\ref{thm:NK} to verify that there exists a unique solution of $F$ \eqref{F_func} in $B_r(x_0)$ for some $r>0$. We remark that $A$ and $A^\dagger$ can be chosen arbitrarily so long as the hypotheses of Theorem~\ref{thm:NK} hold. We choose $A^\dagger$ to be the midpoint of the Jacobian matrix of $F$ \eqref{F_func}, evaluated at $x_0$, and $A$ to be the inverse of $A^\dagger$. We take the Euclidean norm $\|\cdot\|$ on $\mathbb{R}^2$ and the induced $2$-norm $\|\cdot\|_2$ on $\mathcal{B}(\mathbb{R}^2,\mathbb{R}^2)$. Recall that 
\[
\|A\|_{2}\leq \|A\|_F\leq 2\|A\|_\infty \quad\text{for any $A\in\mathcal{B}(\mathbb{R}^2,\mathbb{R}^2)$},
\]
where $\|\cdot\|_F$ denotes the sub-multiplicative Frobenius norm and $\|\cdot\|_\infty$ the maximum norm.

For $x\in \overline{B_r(x_0)}$, we calculate
\begin{align*}
\|A(DF(x)-DF(x_0))\|_{2}\leq &\|A(DF(x)-DF(x_0))\|_F
\leq \|A\|_F\|DF(x)-DF(x_0)\|_F \\
\leq & 2\|A\|_F\|DF(x)-DF(x_0)\|_\infty \\
\leq & 2\|A\|_F \max_{1\leq m,n\leq 2}\sup_{\tilde{x}\in B(\frac12(x+x_0), \frac12\|x-x_0\|) }\| D(DF)_{mn}(\tilde{x})\|_F\|x-x_0\| \\
\leq & 2r \|A\|_F \max_{1\leq m,n\leq 2} \sup_{x\in\overline{B_r(x_0)}}\|D(DF)_{mn}(x)\|_F.
\end{align*}
Therefore we take $b(r)= 2\|A\|_F \max_{1\leq m,n\leq 2} \sup_{x\in\overline{B_r(x_0)}}\|D(DF)_{mn}(x)\|_F$. MATLAB scripts can be made available upon request.

\end{appendix}

\newpage

\bibliographystyle{amsplain}
\bibliography{wwbib}

\end{document}